\documentclass[12pt, preprint]{elsarticle}
\usepackage{amsmath,amssymb,mathabx,verbatim,amsfonts,amscd}
\usepackage[margin=2.5cm]{geometry}
\usepackage{url}
\usepackage{hyperref}
\usepackage{tikz-cd} 
\usepackage{theorem}
\usepackage{cleveref}
\usepackage{enumerate}
\usepackage{enumitem}

\newcommand{\bigslant}[2]{{\raisebox{.2em}{$#1$}\left/\raisebox{-.2em}{$#2$}\right.}}

\newcommand{\ol}{\overline}

\newcommand{\Z}{\mathbb Z}
\newcommand{\F}{\mathbb F}
\newcommand{\K}{\mathbb{K}}

\renewcommand{\P}{\mathbb P}
\newcommand{\A}{\mathbb A}

\newcommand{\calO}{\mathcal{O}}
\newcommand{\calC}{\mathcal{C}}
\newcommand{\calB}{\mathcal{B}}

\newcommand{\frakM}{\mathfrak{M}}

\newcommand{\Gal}{\mathrm{Gal}}

\newcommand{\lto}{\longrightarrow}
\newcommand{\Id}{\mathrm{Id}}

\newcommand{\Div}{\mathrm{Div}}

\newcommand{\absdeg}{\mathrm{absdeg}}
\newcommand{\essdeg}{\mathrm{essdeg}}
\newcommand{\divi}{\mathrm{div}}
\newcommand{\PGL}{\mathrm{PGL}_2}
\newcommand{\Fqq}{\F_{8}}

\newcommand{\p}{Q}
\newcommand{\Pp}{R}
\newcommand{\Bp}{B}
\newcommand{\Zp}{Z}

\renewcommand{\S}{P}
\newcommand{\ti}{\zeta}  
\newcommand{\paramdeg}{\varepsilon}
\newcommand{\totaldeg}{\delta}
\newcommand{\smt}[4]{\left( \begin{smallmatrix} #1 &#2\\ #3 &#4\\ \end{smallmatrix} \right) }
\newcommand{\spv}[2]{\left[ \begin{smallmatrix} #1 \\ #2 \end{smallmatrix} \right]}
\newcommand{\mt}[4]{\left( \begin{matrix} #1 &#2\\ #3 &#4\\ 
\end{matrix} \right) }

\newcommand{\PGLq}{\mathrm{PGL}_2(\F_q)}
\newcommand{\CR}{\mathrm{XR}}
\newcommand{\changee}[1]{#1}

\newtheorem{Theorem}[subsection]{Theorem}
\newtheorem{Proposition}[subsection]{Proposition}
\newtheorem{Lemma}[subsection]{Lemma}
\newtheorem{Corollary}[subsection]{Corollary}

\newtheorem{subLemma}[subsubsection]{Lemma}

\newtheorem{subClaim}[subsubsection]{Claim}

\theoremstyle{remark}
\newdefinition{Definition}[subsection]{Definition}
\newdefinition{Remark}[subsection]{Remark}
\newdefinition{subDefinition}[subsubsection]{Definition}
\newdefinition{subRemark}[subsubsection]{Remark}

\newenvironment{change}{}

\makeatletter
\let\c@equation=\c@subsubsection

\makeatother

\makeatletter
\newenvironment{eqn}{\refstepcounter{subsection}
$$}{\leqno{\rm(\thesubsection)}$$\global\@ignoretrue}
\makeatother  

\makeatletter 
\newenvironment{subeqn}{\refstepcounter{subsubsection}
$$}{\leqno{\rm(\thesubsubsection)}$$\global\@ignoretrue}
\makeatother

\newenvironment{prf}[1]{\trivlist
\item[\hskip \labelsep{\it
	#1\hspace*{.3em}}]}{~\hspace{\fill}~$\square$\endtrivlist}

\newenvironment{proof}{\begin{prf}{\bf Proof}}{\end{prf}}

\begin{document}
\begin{frontmatter}
\title{A provably quasi-polynomial algorithm for the discrete logarithm problem in finite fields of small characteristic}
\author[1]{Guido Lido}
\ead[1]{guidomaria.lido@gmail.com}
\affiliation[1]{organization={University of Roma Tor Vergata},
addressline={Via della Ricerca Scientifica, 1},
postcode={00133},
city={Roma},
country={Italy}}
\begin{abstract}
We describe a provably quasi-polynomial algorithm to compute discrete logarithms in the multiplicative groups of finite fields of small characteristic, that is finite fields whose characteristic is logarithmic in the order. We partially follow the heuristically quasi-polynomial algorithm presented by Barbulescu, Gaudry, Joux
and Thom\'e. The main difference is to use a presentation of the finite field based on elliptic curves: the abundance of elliptic curves ensures the existence of such a presentation. 

\end{abstract}



\end{frontmatter}

\section{Introduction}

A general formulation for the discrete logarithm problem is the following: given a group $G$, a generator $g \in G$ and another element $h \in G$, find an integer $z$ such that 
$g^z = h$. The hardness of this problem, which depends on the choice of $G$, is relevant for public-key cryptography since the very beginning of it \cite{DifHel}.

We are concerned with the cases where $G$ is the multiplicative group of a finite field of \emph{small characteristic}, which, for us, means a field of characteristic $p$ and cardinality $p^n$ for some integers $n > p$. Our main result is the following.
\begin{Theorem}\label{maintheorem1}
There exists a probabilistic algorithm, described in Section \ref{sec_algo}, that solves the discrete logarithm problem in $K^\times$ for all finite fields $K$ of small characteristic in expected time
\[
( \log \# K)^{O(\log \log \# K )} \,.
\]
\end{Theorem}
An algorithm with the above complexity is called \emph{quasi-polynomial}. 

The first heuristic quasi-polynomial algorithm solving the discrete logarithm in small characteristic was presented by Barbulescu, Gaudry, Joux and Thom\'e in \cite{BGJT}.
One of the main ideas is to look for a ``simple'' description of the Frobenius automorphism $\phi \colon K \to K$ and, if one can find such a simple description, to use it in an index calculus algorithm to find relations among the elements of the factor base more easily. This idea was originally contained in Joux's paper \cite{Joux},
\changee{ and it was the starting point of record computations and further refinements, among which \cite{BGJT}, \cite{ZKG}, \cite{comp1}, \cite{comp2}, \cite{compFrancisco1}, \cite{JP} and \cite{gologlu}. In particular, the authors of \cite{gologlu} are the first to substitute the $\rm{PGL_2}$-action  used in \cite{Joux} with \cite{HidePGL2}: in loc. cit. Bluher characterizes the splitting of certain polynomials in a way that also helps us proving that a certain type of multiplicative relations can always be found.}

\changee{
The ideas used in \cite{Joux} to generate multiplicative relations are used differently in \cite{BGJT}, devising a descent procedure that makes the relation generation more efficient, hence the algorithm quasi-polynomial. Another algorithm was then presented in \cite{ZKG}, where the descent procedure is formalized more precisely and its success is rigorously proven using \cite{HidePGL2}}: the complexity is proven to be quasi-polynomial when there exists a ``simple'' description of the Frobenius automorphism (in loc. cit. the notion of ``simple'' is more precise than in \cite{BGJT}). In particular, we could deduce Theorem \ref{maintheorem1} if we knew that all finite fields of small characteristic $K$ can be embedded in a slightly larger field $K'$ admitting a presentation as in \cite{ZKG}. We are not aware of any proof of this fact, despite computations like \cite[Table 1]{Joux} support it and \cite{Micheli} gives a partial answer.

The author's first incomplete attempt to prove Theorem \ref{maintheorem1} is his master's thesis \cite{tesi}, which already contained the main idea of this article: using a different presentation for the finite field, stated in terms of the action of the Frobenius on an elliptic curve, and adapting the algorithms. Such ``elliptic presentations'' were first introduced in \cite{EllPres} and, since over a finite field $\F_q$ there are many non-isogenous elliptic curves, it is easy to prove that all finite fields of small characteristic can be embedded in a slightly larger field admitting such a presentation. The algorithm in \cite{tesi} adapts the approach in \cite{ZKG} to finite fields with an elliptic presentations, but the proof of its correctness and quasi-polynomiality was not complete.

The idea of using elliptic presentations in this context has been independently developed by Kleinjung and Wesolowski, who prove  Theorem \ref{maintheorem1} in \cite{KW2}. 
\changee{They provide an algorithm that works with ``elliptic curve models'', whose definition only differs from Definition \ref{def:ell_pres} by a few details.  In Theorem 2.4. in loc. cit. they prove, with a similar approach to Proposition \ref{propfieldswithpresentation}, that elliptic curve models can be found up to relatively small extension.}
Other than this, the main difference between the present approach and the one in \cite{KW2} is the proof of correctness and the quasi-polynomiality of the algorithms. In both approaches it is a matter of showing the irreducibility of certain curves:  \cite{KW2} describes those curves as components of certain fibered products (see also \cite{KW}), in this paper we prove irreducibility by some Galois theory over function fields. \changee{In both approaches cumbersome computations appear; we tried to confine them} in the proofs of Proposition \ref{prop:we_can_use_other_prop} and Claims \ref{hope_distinct_points}, \ref{hope_no_conics}, \ref{claim_43_use_lemma}. Finally, another small difference between \cite{KW2} and the present work is that in our algorithm the number of ``traps'' is finite and small, so that they can be included in the factor base, while in  \cite{KW2} there are infinitely many ``traps'' that need to be avoided by the algorithm. 


The practical side of this story has been investigated in \cite{JEll}, where Joux and Pierrot propose a practical algorithm for the discrete logarithm based on elliptic presentations and apply it successfully in $\mathbb F_{3^{1345}}^\times$. Their experiments indicate that the efficiency of this algorithm is inferior, yet comparable, to the one in \cite{JP}. 
\changee{Indeed, while elliptic presentations are useful to \emph{prove} quasi-polynomiality, the record computations have been obtained with ``linear'' or ``quadratic'' presentations as in \cite{ZKG}: in \cite{granger2021computation} we see the computation of logs over the field with $2^{30750}$ elements, and in characteristic $3$ we cite \cite{compFrancisco2}, containing computations over the field with $3^{3054}$ elements. }

\paragraph{Structure of the paper}
The structure of the article is as follows. In Section \ref{sec:ell_pres} we define elliptic presentations and we prove that all finite fields of small characteristic can be embedded in a slightly larger field admitting an elliptic presentation. 
Section \ref{sec:traps} is about ``traps'' (cases that need different treatment in the algorithm).
In Section \ref{sec:divs} we describe the general setup of our algorithm and we explain how to pass from a factor base made of irreducible polynomials in $\F_q[x]$ to a factor base made of irreducible divisors on an elliptic curve $E/\F_q$. 
In Section \ref{sec_algo} we give our algorithm, stated in terms of a descent procedure that is described in Section \ref{sec:idea_descent}. A more precise statement about the complexity of the main algorithm is given in Theorem \ref{maintheorem3}. Our descent procedure consists of two steps, presented and analysed in Section \ref{sec:idea_descent} under an assumption on certain varieties. These assumptions are proven in the subsequent sections: Section \ref{sec:lemma} gives a technical lemma, Section \ref{sec:3-2} proves the assumptions needed for the second and easier step and  Section \ref{sec:43} proves the assumptions needed for the first step. \changee{Section \ref{sec:example} describes our setup through some explicit computations of elliptic presentations and of the descent procedure.}

\section{Elliptic presentations} \label{sec:ell_pres}
One of the main ideas in \cite{Joux} and in \cite{BGJT}, is to present a field $K$ using two subfields $\F_q\subsetneq \F_Q \subseteq K$ of order $q,Q$ (both ``small'' compared to $\# K$) and an element $x_1 \in K $ generating the extension $\F_Q\subset K$ such that the $q$-th Frobenius acts on $x_1$ in a simple way, namely $x_1^q=f(x_1)$ for some $f \in \F_q(x)$ of \changee{small} degree. We  now define a presentation based on a similar idea: describing $K$ as $\F_q(x_1,y_1)$ where $\F_q$ is a finite field of order $q$ ``small'' compared to $\# K$ and $x_1,y_1$ are two elements of $K$ on which the  $q$-th Frobenius acts in a ``simple'' way. 

Let $\F_q$ be a finite field of cardinality $q$, and let $K/\F_q$ be a field extension of degree $n$. Suppose there exists an elliptic curve $E/\F_q$ defined by a Weierstrass equation and a point $P_0 \in E(\F_q)$ of order $n$. Denoting $\phi$ the $q$-th Frobenius on the elliptic curve $E$, the map $E\to E$ given by $P \mapsto \phi(P) {-} P$ is surjective. Therefore there is a point $P_1 = (x_1,y_1)$ on $E$ such that $\phi(P_1)= P_1 + P_0$. Hence
\begin{eqn}\label{eq:proof_deg_n}
(x_1^{q^i}, y_1^{q^i}) = \phi^i(P_1)=P_1 + i\cdot P_0 \quad \text{for every }i\in \Z\,,
\end{eqn}
implying that the field extension $\F_q \subset \F_q(x_1,y_1)$ has degree $n$. Hence $\F_q(x_1,y_1)$ is isomorphic to $K$.
Moreover the $q$-th Frobenius acts on the pair $(x_1, y_1)$ in a ``simple'' way in the following sense: the addition formulas on $E$ give polynomials $f_1,f_2,f_3 \in \F_q(x,y)$ of small degree such that
$x_1^q = f_1(x_1,y_1)/f_3(x_1,y_1)$ and $y_1^q = f_2(x_1, y_1)/f_3(x_1,y_1)$.
With this heuristic in mind, we give the following definition.

\begin{Definition}\label{def:ell_pres}
Let $E/\F_q$ be an elliptic curve defined by a Weierstrass polynomial in $\F_q[x,y]$, let $P_0$ be a $\F_q$-point on $E$ and let $\phi\colon E \to E$ be the $q$-th Frobenius. An \emph{$(E/\F_q,P_0)$-presentation} of a finite field $K$ is an ideal $\frakM \subset \F_q[x,y]$ such that 
\begin{enumerate}[label=(\roman*)]
\item\label{def:ell_pres_1} $K$ is isomorphic to $\F_q[x,y]/\frakM$ with a chosen isomorphism;
\item there exists a point $P_1 =(x_1,y_1)$ on $E$ such that $\phi(P_1)=P_1 + P_0$ and such that  $\frakM = \{ f \in \F_q[x,y]: f(x_1, y_1)=0 \}$;
\item\label{def:ell_pres_3} $q>2$ and $[K:\F_q]>2$.
\end{enumerate}
\end{Definition}
We sometimes omit the dependence on $(E/\F_q,P_0)$ and we simply write ``elliptic presentation''. 
The hypothesis $q>2$ is used in the proof of Claim \ref{hope_distinct_points}, while the hypothesis $[K:\F_q]>2$ is used in the following Remark, which is used in step $3$ of the main algorithm to lift elements of $K$ to polynomials in $\F_q[x]$, in particular  to polynomials of degree a power of two. 

\begin{Remark}\label{rem:mu}
If $\frakM$ is an elliptic presentation, the inclusion $\F_q[x] \to \F_q[x,y]$ induces an isomorphism $\F_q[x]/\mu \cong \F_q[x,y]/\frakM$ for a certain $\mu \in \F_q[x]$. 

Proving this is equivalent to proving that $x$ generates the field extension $\F_q \subset \F_q[x,y]/\frakM$. Using the notation in Definition \ref{def:ell_pres}, this is equivalent to proving that $\F_q(x_1)$ is equal to $\F_q(x_1,y_1)$. If, for the sake of contradiction, this is not the case, then the Weierstrass equation satisfied by $x_1$ and $y_1$ implies that the extension $\F_q(x_1)\subset \F_q(x_1,y_1)$ has degree $2$, hence $[\F_q(x_1):\F_q]=\tfrac n2$, where $n:=[\F_q(x_1,y_1):\F_q] = [K:\F_q]$. 
Using Equation \ref{eq:proof_deg_n}, we deduce that 
\[
x(P_1) = x_1 = x_1^{q^{n/2}} = x(\phi^{n/2}P_1) = x(P_1 + \tfrac n2 P_0) \quad \implies \quad  P_1 + \tfrac n2 P_0 = \pm P_1\,.
\]
Since, by Equation \ref{eq:proof_deg_n}, the order of $P_0$ is equal to $n$, we have $P_1 + \tfrac n2 P_0 = -P_1$, implying that  $2P_1$ lies in $E(\F_q)$. Therefore $P_0$ has order $2$, contradicting the hypothesis $n = [K:\F_q]>2$ in \ref{def:ell_pres_3}.
\end{Remark}

We now show that any finite field $K$ of small characteristic can be embedded in a ``slightly larger'' field admitting an elliptic presentation with $q$  ``small'' compared to $\# K$.

\begin{Proposition}\label{propfieldswithpresentation}
For any finite field $K$ of small characteristic there exists an extension $K \subset K'$ having an elliptic presentation  $\frakM \subset \F_q[x,y]$ of $K'$ such that 
\[
\log (\# K') \leq 13 \log (\# K) \log \log(\# K) \quad \text{and}\quad  q \leq  \log (\# K')^{4} \,.
\]
Moreover such $K'$ and its presentation can be computed in polynomial time in $\log(\# K)$.
\end{Proposition}
\begin{proof}
Let  $\# K = p^n$ for a prime $p$ and an integer $n> p$.  Put $k_0:=\lceil\log_p n \rceil$ and $q:=p^{2k_0}$, so that $n$ has a multiple $n_1$ in the interval $[q-\sqrt{q}+1, q+1 ]$. If $n_1 \equiv 1\bmod p$ we define $n_2:=n_1+n$, otherwise we define $n_2:=n_1$. Since $n_2$ is an integer contained in the Hasse interval $ [q - 2\sqrt q + 1 ; q+2\sqrt q+1]$ that is not congruent to $1$ modulo $p$, by
 \cite[Theorems 1a, 3]{Ruc}\footnote{there is a small typo in Theorem 1b of loc. cit., which is based on \cite[Theorem 4.2]{waterhouse1969abelian}: in item (a) there, not all orders in $K$ are possible endomorphism rings, but all orders containing $\pi$ are. This does not affect Theorems 1a, 3} there exists an elliptic curve $E/\F_q$ whose group of rational points $E(\F_q)$ is cyclic of order $n_2$. Since $n$ divides $n_2$, there exists a point $P_0 \in E(\F_q)$ of order $n$.

We can assume $E$ is defined by a Weierstrass polynomial. 
\changee{
The map $\phi-\Id\colon E \to E$, mapping $P \mapsto \phi(P){-} P $ is surjective, since it is an algebraic map of curves and, being not constant, its image has dimension one.} 
In particular there exists a point $(x_1, y_1)=P_1 \in E(\ol{\F_q})$ such that $\phi(P_1)=P_1 + P_0$.

Under the definition
\[
\frakM:= \{ f \in \F_q[x,y]: f(x_1, y_1)=0 \}\,, \quad  K':=\F_q(x_1, y_1) \subset \ol{\F_q}\,, 
\]
it is clear that the map $\F_q[x,y]\to K$ sending $x\mapsto x_1, y \mapsto y_1$ induces an isomorphism $\F_q[x,y]/\frakM \cong K'$.
To prove that $\frakM$ is an elliptic presentation of $K'$, we are left to show that  $q>2$ and $[K':\F_q]>2$: the first is is a consequence of the inequality $k_0=\lceil\log_p n \rceil >1$, the latter is true since, by (\ref{eq:proof_deg_n}), the degree of $\F_q \subset K'$ is equal to the order $n$ of $P_0$, and $n>p\ge 2$.

Since $[K':\F_q]=n$ divides $[K':\F_p]$, the field $K'$ has a subfield with $p^n$ elements. In other words $K$ can be embedded in $K'$. Moreover we have
\[
\begin{aligned}
\log(\# K') & = n \log q < 2n \log(p) (\log_p(n) {+} 1)  \leq 4 \changee{n} \log(p) \log(n) \\
&  \leq 13 \log(\# K) \log \log (\# K) , \\
q & =p^{2 \lceil \log_p n \rceil} < p^{2+2\log_p n} = (pn)^2 \leq n^4 < \log(q^n)^4 = \log(\# K')^4	\,. \qquad \qquad 
\end{aligned}
\]
We now prove that it is possible to compute such $K'$ and $\frakM$ in polynomial time in $\log(\# K)$. We describe a procedure following the abstract part of the proof. Computing $k_0,q,n_1$ is easy. We can construct a field $\F_q$ by testing the primality of all polynomials of degree $2k_0$ over~$\F_p$ until an irreducible $\nu$ is found and define $\F_q= \F_p[T]/\nu$; since there are less than $n^2$ polynomials of this type, this takes polynomial time. Similarly we can find an elliptic curve $E$ with an $\F_q$-point $P_0$ of order $n$ in polynomial time, by listing all possible Weierstrass equations (there are less than $q^6$), testing if they define an elliptic curve and, when they do, enumerating all their $\F_q$-points. Then, using the addition formula on $E$, we write down the ideal $I \subset  \F_q[x,y]$ whose vanishing locus inside $\A^2$ is the set of points $P=(x, y) \in E(\ol{\F_q})$ such that $\phi(P) = P + P_0$. As we showed before, the set of such points is non-empty, hence $I$ is a proper ideal and we can find a maximal ideal $\frakM$ containing $I$. We don't need general algorithms for primary decomposition since we can take $\frakM=(\mu(x), \lambda(x,y))$, with $(\mu)$ being an 
irreducible factor of the generator of the ideal $J{\cap}\F_q[x]$ and $\lambda(x,y)$ being an irreducible factor of the image of the Weierstrass equation of $E$ inside $(\F_q[x]/\mu)[y]$. Since the Weiestrass polynomial is monic in $y$, we can assume that $\lambda$ is monic in $y$ too. Hence there is a point $P_1 =(x_1, y_1)$ in the vanishing locus of $(\mu(x), \lambda(x,y))=\frakM$. Since $\frakM$ contains $I$, the point $P_1$ lies on $E$ and satisfies $\phi(P_1)=P_1+P_0$. The maximality of $\frakM$ implies that $\F_q[x,y](\frakM) = \F_q(x_1,y_1)=K'$. Hence $\frakM$ is the elliptic presentation we want.
\end{proof}

\begin{Remark}
\begin{change}
	A small difference between the proof of \cite[Theorem 2.4]{KW2} and the above proof is that the $q$ above is about $O(n^{2})$ instead of $ O(n^4)$ as in \cite{KW2}.  Indeed we use  R\"uck's refinement \cite[Theorem 3]{Ruc} of Waterhouse's result  \cite[Theorem 4.1]{waterhouse1969abelian}. In broad terms, \cite[Theorem 3]{Ruc} states that not only the cardinalities of ordinary elliptic curves over a finite field are all the possible ``ordinary cardinalities'' in the Hasse interval, moreover almost all possible group structures are reached.
\end{change}
\end{Remark}

\begin{notation}
For the rest of the article $\F_q$ is a finite field with $q$ elements, $\ol{\F_q}$ is its algebraic closure, $K$ is a finite extension of $\F_q$, the ideal $\frakM\subset \F_q[x,y]$ is a $(E/\F_q,P_0)$-presentation of $K$, the map $\phi\colon E \to E$ is the $q$-th Frobenius and $P_1=(x_1, y_1) \in E(\ol{\F_q})$ is a point such that $\frakM= \{ f \in \F_q[x,y]: f(x_1, y_1)=0 \}$. By $O_E$ we denote the neutral element of $E(\F_q)$.
\end{notation}

\section{Traps}\label{sec:traps}
As first pointed out in \cite{traps}, there are certain polynomials, called ``traps'' for which the descent procedure in \cite{BGJT} does not work. In \cite{BGJT} such traps are dealt with differently than the other polynomials. In \cite{ZKG} the notion of ``trap'' is extended: it includes not only polynomials for which the descent procedure is proven not to work, but also polynomials for which the strategy to prove the descent's correctness does not work. In \cite{ZKG} traps are avoided by the algorithm.

In the following sections we describe a descent procedure stated in terms of points and divisors on $E$ and there are certain points in $E(\ol{\F_q})$ that play the role of ``traps'', as in \cite{ZKG}. The definition of this subset of $E(\ol{\F_q})$ is rather cumbersome, but it is easy to deduce that we have less than $15q^4$ traps. In particular, in contrast to \cite{ZKG}, we can include them in the factor base without enlarging it too much.
\begin{Definition}\label{def:traps}
A point $P\in E(\ol{\F_q})$ is a \emph{trap} if it satisfies one of the following conditions:
\[
\begin{aligned}
2P =0\,, \quad \text{ or }\quad (2\phi-\Id)(\phi^2-\phi+\Id)(P) = P_0\,, \\ \text{ or } \quad (2\phi-\Id)(\phi+\Id)(P) = 2P_0 \quad 
\text{or }	(\phi^4-\Id)(P) = 4P_0\,, \\ \text{ or } 2(\phi^3-\Id)(P) = 6P_0\,, \quad \text{ or }\quad  (2\phi+\Id)(\phi-\Id)(P) = 2P_0\,.
\end{aligned}
\]
\end{Definition} 

In (\ref{eq:use_traps}) and at the beginning of the proof of Claim \ref{hope_distinct_points} we explain why these points interfere with our strategy of proof.

\section{Divisors and discrete logarithm}\label{sec:divs}
We recall that the Galois group of $\F_q$ acts on the group of divisors on $E$ by the formula
\[
\sigma \left(\sum_{P \in E(\ol{\F_q})} n_P P \right) =\sum_{P \in  E(\ol{\F_q})} n_P \, \sigma(P)\,.
\]
For any algebraic extension $\F_q \subset k$ we denote $\Div_k(E)$ the set of divisors \emph{defined over~$k$}, namely the divisors $D$ such that $\sigma D = D$ for all $\sigma \in \Gal(\ol{\F_q}/k)$. We say that a divisor is \emph{irreducible over~$k$} if it is the sum, with multiplicity $1$, of all the $\Gal(\ol{\F_q}/k)$-conjugates of some point $P \in E(\ol{\F_q})$. Every divisor defined over~$k$ is a $\Z$-combination of irreducible divisors over~$k$. 
We refer to \cite[Chapter $2$]{Sil} for the definitions of principal divisor and support of a divisor.

We need two quantities to describe the ``complexity'' of a divisor.
The first one is the \emph{absolute degree} of a divisor, 	 defined as as
\[ 
\absdeg\left( \sum_{P \in E(\ol{\F_q})} n_P (P)\right) := \sum_{P \in E(\ol{\F_q})} |n_P| \,. 
\]
The second quantity is analogous to the degree of the splitting field of a polynomial, but we decide to ``ignore'' trap points. 
Given a divisor $D\in \Div_k(E)$, we denote $D^{\mathrm{noTrap}}$ the part of $D$ that is supported outside the set of trap points, which is also defined over $k$ since the set of trap points is $\Gal(\ol{\F_q}/\F_q)$-invariant. We define the \emph{essential degree of $D$ over~$k$} to be the least common multiple of the degrees of the irreducible divisors appearing in  $D^\mathrm{noTrap}$. In other words, if we denote as $k(D^{\mathrm{noTrap}})$ the minimal algebraic extension $\widetilde k \supset k$ such that the support of $D$ is contained in $E(\widetilde k)$, then
\begin{eqn}\label{eq:essdeg}
\essdeg_{k} (D):= [k(D^{\mathrm{noTrap}}): k]  \,.	
\end{eqn}

If $D^\mathrm{noTrap} =0$ we take $\essdeg_{k}(D) = 1$.

Now consider the discrete logarithm problem in a field having an elliptic presentation $\frakM$. First of all, if $q$ is small compared to $\# K$, for example $q \leq (\log K)^{4}$ as in Proposition \ref{propfieldswithpresentation}, and if we are able to compute discrete logarithms in $K^\times / \F_q^\times$ in quasi-polynomial time, then we can also compute discrete logarithms in $K^\times$ in quasi-polynomial time. Hence in the rest of the article we are concerned with computing discrete logarithms in $K^\times / \F_q^\times$.  

Denoting $\F_q[x,y]_\frakM$ the localization of $\F_q[x,y]$ at the maximal ideal $\frakM$, we have 
\[
K \, \cong \, \F_q[x,y]/\frakM \,\cong\, \F_q[x,y]_\frakM/\changee{\frakM} \,.
\]
An element $f$ of $(\F_q[x,y]_\frakM)^\times$ defines a rational function on $E$ which is defined over~$\F_q$ and regular and non-vanishing in $P_1$.
We represent elements in $K^\times /\F_q^\times$ with elements of $\F_q(E)$ that are regular and non-vanishing on $P_1$. 

Let $g,h$ be elements of $\F_q(E)$ both regular and non-vanishing on $P_1$ and let us suppose that $g$ generates the group $K^\times/\F_q^\times$. Then the logarithm of $h$ in base $g$ is a well defined integer modulo $\tfrac{\# K{-}1}{q{-}1}$ that we denote $\log_{\frakM,g}(h)$ or simply $\log h$. Since we are working modulo $\F_q^\times$, the logarithm of $h$ only depends on the divisor of zeroes and poles of $h$: if $h' \in \F_q(E)$ satisfies $\divi(h) = \divi(h')$, then $h/h' \in \F_q^\times$ and consequently $\log(h) = \log(h')$. Hence, putting
\[
\log (\divi(h)) := \log (h) \,,
\]
we define the discrete logarithm as \changee{an} homomorphism whose domain is the subgroup of $\Div_{\F_q}(E)$ made of principal divisors, supported outside $P_1$ and whose image is $\Z/(\tfrac{\# K{-}1}{q{-}1})\Z$.
The kernel of this morphism is a subgroup of $\Div_{\F_q}(E)$, hence it defines the following equivalence relation on $\Div_{\F_q}(E)$
\begin{eqn}\label{eq_equiv_divisors}
\begin{aligned}
	D_1\sim D_2  & \iff D_1-D_2 \in \mathrm{Ker}(\log)  \\
	& \iff \exists f \in \F_q(E) \mbox{ such that }f(P_1) =1 \mbox{ and } \divi(f) = D_1 - D_2 \,.
\end{aligned}
\end{eqn}
We notice that this equivalence relation does not depend on $g$ and that, 
given rational functions $h_1, h_2 \in \F_q(E)$ regular and non-vanishing on $P_1$, we have  $\log h_1 = \log h_2$ if and only if $\divi(h_1) \sim \divi(h_2)$. Motivated by this, for all divisors $D_1,D_2 \in \Div_{\F_q}(E)$ we use the notation
\[
\log_\frakM D_1 = \log_\frakM D_2 \iff D_1\sim D_2 \,.
\]
Notice that we do not define the expression $\log_{\frakM} (D)$ or $\log_{\frakM,g}(D)$ for any $D$ in $\Div_{\F_q}(E)$, since the function $\log$ might not extend to a morphism $\Div_{\F_q}(E) \to \Z/(\tfrac{\# K{-}1}{q{-}1})\Z$. In our algorithm we use the	 equivalence relation (\ref{eq_equiv_divisors}) to recover equalities of the form $\log h_1 = \log h_2$. 

\section{The main algorithm} \label{sec_algo}
As in \cite{ZKG}, our algorithm is based on a descent procedure, stated in terms of divisors on $E$.

\begin{Theorem}\label{theo_descent}
There exists an algorithm, described in the proof, that takes as input an $(E/\F_q,P_0)$-presentation $\frakM$ and a divisor $D \in \Div_{\F_q}(E)$ such that $\essdeg_{\F_q}(D) = 2^{m}$ for some integer $m\ge \changee{9}$  and computes a divisor $D' \in \Div_{\F_q}(E)$  such that
\[  
\log_\frakM D= \log_\frakM D' \,, \quad (\essdeg_{\F_q} D') \mid 2^{m-1}\,, \quad \absdeg (D') \leq 4q^2 \absdeg D \,.
\]  
This algorithm is probabilistic and its expected runtime is polynomial in $q\cdot \absdeg(D)$.
\end{Theorem}

Applying repeatedly the algorithm of the above theorem we deduce the following result.
\begin{Corollary}\label{cor:descent}
There exists an algorithm, described in the proof, that takes as input an $(E/\F_q,P_0)$-presentation and a divisor $D \in \Div_{\F_q}(E)$ such that $\essdeg_{\F_q}D = 2^m$ for some integer $m$ and computes  a divisor $D' \in \Div_{\F_q}(E)$  such that 
\[
\log_\frakM D= \log_\frakM D' \,, \quad \essdeg_{\F_q} D' \mid \changee{2^8} \,,\quad \absdeg(D')\leq (2q)^{2m} \absdeg(D) \,. 
\]
This algorithm is probabilistic  and its expected runtime is polynomial in $q^m \absdeg(D)$.
\end{Corollary}

\changee{Notice that the definition of $\essdeg_{\F_q}D'$ ``excludes'' the trap points, hence the divisor $D'$ outputted in Corollary \ref{cor:descent} can contain trap points; the non-trap points in the support of $D'$ are contained in $E(\F_{q^{256}})$.}

The algorithm in \cite{ZKG} is based on the descent procedure \cite[Theorem 3]{ZKG}. Using the same ideas we use the descent procedure of the last corollary to describe our main algorithm, which computes discrete logarithms in finite fields with an elliptic presentation. 

The idea is setting up an index calculus with factor base the irreducible divisors whose essential degree divides $\changee{2^8}$. To collect relations we use a ``zig-zag descent'': for every $f = g^ah^b$, we first use the polynomial $\mu$ determined in Remark \ref{rem:mu} to find $f' \equiv f \bmod \frakM$ such that the essential degree of $\divi(f')$ is a power of $2$, and we then apply the descent procedure to express $\log(f)=\log(f')$ as the logarithm of sums of elements in the factor base. 

\subparagraph*{Main Algorithm}

Input: an $(E/\F_q,P_0)$-epresentation $\frakM \subset \F_q[x,y]$ of a field $K$ and two polynomials $g,h \in (\F_q[x,y] {-} \frakM)$ such that $g$ generates the group $\left( \F_q[x,y]/\frakM\right)^\times/\F_q^\times$.

Output: an integer $z$ such that 
\[
g^z \equiv \gamma \cdot h \pmod{\frakM} \quad \mbox{ for some }\gamma \in \F_q^\times \,,
\]
which is equivalent to $g^z=h$ in the group $K^\times/\F_q^\times$.

\begin{enumerate}
\item \emph{Preparation:} Compute the monic polynomial $\mu \in \F_q[x]$ generating the ideal $\frakM\cap \F_q[x]$. Compute polynomials $\tilde g, \tilde h \in \F_q[x]$ such that $\tilde g \equiv g$ and $\tilde h \equiv h$ modulo~$\frakM$. Put $c := \# E(\F_q)$, $n:= \deg \mu$ and $m := \lceil \log n \rceil +3  $.

\item \emph{Factor base:} List the irreducible divisors $D_1,\ldots, D_t\in \Div_{\F_q}(E)$ that do not contain $P_1$ and either have degree dividing $\changee{2^8}$ or are supported on the trap points.

\item \emph{Collecting relations}: For $j=1,\ldots, t{+}1$ do the following:

Pick random integers $\alpha_j,\beta_j \in \{1,\ldots, \tfrac{q^n-1}{q-1}\}$ and compute $\tilde{g}^{\alpha_j} \tilde h^{\beta_j}$. Pick random polynomials $f(x)$ of degree $2^m$ such that $f\equiv \tilde{g}^{\alpha_j} \tilde h^{\beta_j}  \pmod \mu$ until $f$ is irreducible. Apply the descent procedure in Corollary \ref{cor:descent} to find $v_j = (v_{j,1},\ldots,v_{j,t}) \in \Z^t$ such that 
\[
\log_{\frakM} \left( \divi(f) \right) = \log_\frakM \left(v_{j,1} D_1 + \ldots + v_{j,t} D_t \right) \,.
\]

\item \emph{Linear algebra}: Compute $d_1, \ldots, d_{t+1} \in \Z$ such that $\gcd(d_1, \ldots, d_{t+1})=1$ and 
\[ d_1v_1+\ldots +d_{t+1}v_{t+1} \equiv (0,\ldots,0) \pmod{  \tfrac{q^n-1}{q-1}c }\,.
\]
Put $a:= d_1 \alpha_1 + \ldots + d_{t+1}\alpha_{t+1} $ and $b := d_1 \beta_1 + \ldots + d_{t+1}\beta_{t+1}$.
\item \emph{Finished?}: If $b$ is not invertible modulo $\tfrac{q^n-1}{q-1}$ go back to step $3$, otherwise output
\[
z := -ab^{-1} \,\left(\bmod{\,\tfrac{q^n-1}{q-1}}\right)
\]
\end{enumerate}

\subparagraph*{Analysis of the main algorithm}
We first prove, assuming Theorem \ref{theo_descent}, that the algorithm, when it terminates, gives correct output. First of all we notice that, as explained in Remark \ref{rem:mu}, the polynomials $\mu, \tilde g$ and $\tilde h$ exist and that $\tilde g$ and $\tilde h$ define the same element as $g$, respectively $h$, in $K \cong \F_q[x,y]/\frakM$. 
Let $d_j, \alpha_j, \beta_j$ and $v_j$ be the integers and vectors of integers stored at the beginning of the fourth step the last time it is executed. By definition of $d_j$, we have 
\[
\sum_{j=1}^{t+1} \sum_{i=1}^t d_j v_{j,i} D_i =  \tfrac{q^n-1}{q-1} c \cdot 
D \,,
\]
for a certain $D \in \Div_{\F_q}(E)$.

\changee{For all $j$ the divisor $\sum_i v_{j,i}D_i$ has the same logarithm as a certain $\divi(f)$, hence the divisor  $\sum_i v_{j,i}D_i-\divi(f)$ is principal. In particular $\sum_i v_{j,i}D_i$ has degree $0$, hence also $D$ has degree $0$ and, since  $c= \# \mathrm{Pic}^0(E/\F_q)$, the divisor $cD$ is principal.}
Choosing $\lambda$ in $\F_q(E)$ such that $\divi(\lambda) = cD$, we have 
\begin{eqn}\label{eq:lambda}
\sum_{j=1}^{t+1} \sum_{i=1}^t d_j v_{j,i} D_i = \divi(\lambda^{\frac{q^n-1}{q-1}})\,.
\end{eqn}
Writing $\log$ for $\log_{\frakM,g}$, by definition of $v_j$ we have
\begin{equation*}\label{eq:v_j}
\log( g^{\alpha_j} h^{\beta_j}) = \log \left(\sum_{i=1}^t v_{j,i} D_i \right) \,.
\end{equation*}
This, together with Equation (\ref{eq:lambda}), imply the following equalities in $\Z/\tfrac{q^n-1}{q-1}\Z$ 
\[
\begin{aligned}
a+ b\log(h) &= \sum_{j=1}^{t+1} d_j(\alpha_j + \beta_j \log(h)) = \sum_{j=1}^{t+1} d_j \log( g^{\alpha_j} h^{\beta_j}) 
= \sum_{j=1}^{t+1} d_j \log \big(\sum_{i=1}^t v_{j,i} D_i \big) \\
& = \log \big( \sum_{j=1}^{t+1} \sum_{i=1}^t d_j v_{j,i} D_i \big) = \log\big(\divi(\lambda^{\frac{q^n-1}{q-1}})\big) =\tfrac{q^n-1}{q-1} \log(\lambda) =  0 \,,
\end{aligned}
\]
implying that the output $z$ of the algorithm is correct.

We now estimate the running time step by step. 
The first step can be performed with easy Groebner basis computations. Now the second step. We represent irreducible divisors $D$ not supported on $O_E$ in the following way: either $D$ is the vanishing locus of a prime ideal $(a(x),W(x,y))$ with $a$ monic and irreducible and $W$ the Weierstrass polynomial defining $E$, or $D$ is the vanishing locus of a prime ideal $(a(x),y-b(x))$ for some polynomials $a,b \in \F_q[x]$ and $a$ monic irreducible; in the first case $\deg  D = 2\deg a$, in the second case $\deg D = \deg a$. 
We can list all the irreducible divisors with degree dividing $\changee{2^8}$ by listing all monic irreducible polynomials $\mu_1, \ldots, \mu_r \in \F_q[x]$ of degree dividing $\changee{2^8}$ and, for each $i$ compute the prime ideals containing $(\mu_i, W)$, which amounts to factoring $W$ as a polynomial in $y$, considered over the field $\F_q[x]/\mu_i$. 
Listing all the divisors supported on the trap points can be done case by case. For example we can list the irreducible divisors supported on the set $S:= \{P \in E(\ol{\F_q}) : \phi^4(P)-P = 4P_0\}$ by writing down, with the addition formula on $E$, an ideal $J \subset \F_q[x,y]$ whose vanishing locus is $S \subset \A^2(\ol \F_q)$ and computing all the prime ideals containing $J$. The divisor $O_E$ appears among $D_1,\ldots, D_s$ because $O_E$ is a trap point.  
Since there are $q^{256}$ monic polynomials of degree $\changee{2^8}$ and at most $15q^4$ trap points and since, using \cite{Ber}, factoring a polynomial of degree $d$ in $\F_q[x]$ takes on average $O(\log(q)d^3)$ operations, the second step takes polynomial time in $q$. 
Moreover, we have $t \le 2q^{256}$.

Now the third step. By \cite[Theorem 5.1]{Dirichlet}, if $f(x)$ is a random polynomial of degree $2^m$ congruent to $\tilde g^{\alpha_j} \tilde h^{\beta_j}$ modulo~$\mu$, then the probability of $f$ being irreducible is at least $2^{-m-1}$. Therefore finding a good $f$ requires on average $O(2^m) = O(n)$ primality tests, hence $O(n^4\log q)$ operations. By assumption finding the vector $v_j$ requires polynomial time in $q^m 2^{m+1}$. We deduce that the third step has probabilistic complexity $t q^{O(\log n)} = q^{O(\log n)}$. 

The fourth step can be can be performed by computing a Hermite normal form of the matrix having the $v_j$'s as columns. Since $c\leq q{+}2\sqrt{q}+1$, the entries of the $v_j$ are at most as big as $4q^{n+1}$. Therefore the fourth step is polynomial in $t\log(q^n)$, hence polynomial in $n$.

The last step only requires arithmetic modulo $(q^n{-}1)/(q{-}1)$.

To understand how many times each step is repeated on average, we need to estimate the probability that, in the last step, $b$ is invertible modulo $(q^n{-}1)/(q{-}1)$ and to do so we look at the quantities in the algorithms as if they were random variables. The vector $(d_1,\ldots, d_{t+1})$ only depends on the elements $h^{\alpha_j}g^{\beta_j}$'s and on the randomness contained in the descent procedure and in step $2$. Since the $\alpha_j$'s and $\beta_j$'s are independent variables and since $g$ is a generator, we deduce that the vector $(\beta_1 , \ldots, \beta_{t+1})$ is independent of $(g^{\alpha_1}h^{\beta_1}, \ldots, g^{\alpha_{t+1}}h^{\beta_{t+1}})$, hence also independent of the vector $(d_1, \ldots, d_{t+1})$. Since   $(\beta_1 , \ldots, \beta_{t+1})$ takes on all values in $\{0,\ldots, q^n-1 \}^{t+1}$ with the same probability and $\mathrm{gcd}(d_1,\ldots, d_{t+1})=1$, then  
\[
b= d_1 \beta_1 + \ldots d_{t+1}\beta_{t+1} 
\]
takes all values in $\Z/ (q^n-1)\Z$  with the same probability. Hence
\[
\Big(\text{probability that $b$ is coprime to }\tfrac{q^n-1}{q-1} \Big) = \phi\left(\tfrac{q^n-1}{q-1}\right)/\tfrac{q^n-1}{q-1}  \gg \frac{1}{\log \log q^n}
\] 
When running the algorithm, the first and the second step get executed once and the other steps get executed the same number of times, say $r$, whose expected value is the inverse of the above probability. Since $r$ is  $O(\log \log (q^n))$ on average and each step has average complexity at most $q^{O(\log n)}$, the average complexity of the algorithm is $O(q^{O(\log n)})$.
Hence, assuming Theorem \ref{theo_descent} we have proved the following theorem.
\begin{Theorem}\label{maintheorem3}
The above Main Algorithm solves the discrete logarithm problem in the group $K^\times/\F_q^\times$ for all finite fields $K$ having an elliptic presentation $\frakM \subset \F_q[x,y]$. It runs in expected time $ q^{O(\log[K:\F_q])}$.
\end{Theorem}

Theorem \ref{maintheorem1} follows from Theorem \ref{maintheorem3} and Proposition  \ref{propfieldswithpresentation}: the latter states that any finite field of small characteristic $K$ can be embedded in a slightly larger field $K'$ having an elliptic presentation $\frakM \subset \F_q[x,y]$ such that $ q \leq \log (\# K')^4$ and Theorem \changee{\ref{maintheorem3}} implies that the discrete logarithm problem is at most quasi-polynomial for such a $K'$. Moreover, by Proposition \ref{propfieldswithpresentation}, such a $K'$, together with its elliptic presentation, can be found in polynomial time in $\log(\# K)$, by \cite{LenIso} we can compute an embedding $K \hookrightarrow K'$ in polynomial time in $\log(\# K)$ and by \cite[Theorem $15$]{RosSch} a random element $g' \in K'$ has probability $\phi(\# K')/\# K' \gg 1/ \log \log \# K'$ of being a generator of $K'$: hence, given elements $g,h \in K$, we can compute $\log_g(h)$ by embedding $K$ inside $K'$ and trying to compute the pair $(\log_{g'}g, \log_{g'}h)$ for different random values of $g' \in K'$.


Proposition \ref{propfieldswithpresentation} is proven, while the proof of Theorem \ref{maintheorem3} relies on the the existence of a descent procedure as described in Theorem \ref{theo_descent}: in the rest of the article, we describe this descent procedure.

\section{The descent procedure}\label{sec:idea_descent}
In this chapter we describe the descent procedure of Theorem \ref{theo_descent}: \changee{we first make} a couple of reductions, then we split the descent in two steps and, in both steps, we reduce our problem to the computation of $k$-rational points on certain varieties for a certain extension $k$ of $\F_q$. In Sections \ref{sec:lemma}, \ref{sec:3-2}, \ref{sec:43} we give the exact definition of these varieties and we prove that they have many $k$-rational points, which implies that our algorithm has a complexity as in Theorem \ref{theo_descent}.

Let $D$ be as in Theorem \ref{theo_descent}. \changee{We can write $D$ as a combination of divisors $D_i$ that are irreducible over~$\F_q$, apply the descent to the $D_i$'s and take a linear combination of the results to reconstruct a possible $D'$. 
Therefore, we can suppose that $D$ is irreducible over~$\F_q$.}

\changee{The set of trap points is Galois-stable, hence an irreducible divisor is either supported on this set or on the complement}.
If $D$ is supported on the set of trap points, in the algorithm we just pick $D'=D$. Hence we focus on the case where $D$ is irreducible and supported outside the trap points. 
In other words, denoting $2^m = 4l$, we can write 
\[
D = \p + \sigma \p + \ldots + \sigma^{4l-1}\p\,,
\]
for $\p$ a non-trap point on $E$ such that $[\F_q(\p ):\F_q] = 4l =2^{m}$ and $\sigma$ a generator of $\Gal(\F_q(\p)/\F_q)$. 

We can make a sort of ``base change to $k$''.
Let $k$ be the unique subfield of $\F_q(\p)$ such that $[k:\F_q]=l$ and let us define
\[
\tilde D := \p + \sigma^l\p + \sigma^{2l}\p + \sigma^{3l}\p \quad \in \Div_k(E)\,.
\] 
If we are able to find a divisor $\tilde D' \in \Div_k(E)$ and a rational function $g \in k(E)$ such that 
\begin{eqn}\label{eq:g=1_onGaloisconjugates}
\begin{aligned}
& \qquad \absdeg\tilde D' \leq 16 q^2\,, \qquad \qquad \essdeg_k \tilde D' \mid 2 \\
	& \divi(g) = \tilde D - \tilde D'\,, \qquad \qquad g(\tau(P_1)) = 1 \quad \mbox{for all }\tau \in \Gal(\ol{\F_q}/\F_q)\,, 
\end{aligned}
\end{eqn}
then the divisor
\[
D' := \tilde D' + \sigma (\tilde D') + \ldots \sigma^{l-1}(\tilde D') \,,
\]
satisfies the conditions in Theorem \ref{theo_descent}: 
the absolute and essential degree of $D'$ are easy to estimate and we have $\log_\frakM D = \log_\frakM D'$ because the rational function $f :=g g^{\sigma}\cdots g^{\sigma^{l-1}}$ satisfies $f(P_1)=1$ and $\divi(f) = D - D'$.

Hence, in order to prove Theorem \ref{theo_descent}, it is enough to describe a probabilistic algorithm that takes $k$ and $\tilde D$ as input and, in expected polynomial time in $ql$, computes $g, \tilde D'$ as in \ref{eq:g=1_onGaloisconjugates}.
Such an algorithm can be obtained by applying in sequence the algorithms given by the following two propositions. In other words, the descent procedure is split in two steps. 

\begin{Proposition}\label{firsthalfdescentprop}
There is an algorithm, described in the proof, with the following properties
\begin{itemize}
	\item it takes as input an $(E/\F_q,P_0)$-presentation, a finite extension $\F_q \subset k$ of degree  at least $80$ and a divisor $D \in \Div_k(E)$ such that $\essdeg_{k}D=4$
	\item it computes a rational function $g\in k(E)$ and a divisor $D' = D_1 + D_2$ in $\Div_k(E)$ such that 
	\[\essdeg_{k} (D_1) {\mid} 3, \,\,\, \essdeg_{k} (D_2) {\mid} 2 ,\,\,\,   \absdeg D_1 + \absdeg D_2 \leq 2q \,\absdeg D, \]
	and 
	\begin{equation*}
	\begin{aligned}
 & D - D' = \divi(g) \text{ for a function $g$ such that }\\ 
 &\quad g(P) = 1  \mbox{ for all }P \in E(\ol{\F_q}) \mbox{ satisfying }\phi(P)=P+P_0;
	\end{aligned}
	\end{equation*}
	\item  it is probabilistic with expected runtime polynomial in $q\log(\# k)\absdeg(D)$.
\end{itemize}
\end{Proposition}

\begin{Proposition}\label{secondhalfdescentprop}
There is an algorithm, described in the proof, with the following properties
\begin{itemize}
	\item it  takes as input an $(E/\F_q,P_0)$-presentation, an extension of finite fields $\F_q \subset k$ of degree at least $80$ and a divisor $D \in \Div_k(E)$ such that $\essdeg_{k}D=3$;
	\item it computes a rational function $g\in k(E)$ and a divisor $D' \in \Div_k(E)$ such that 
	\[
	\essdeg_{k} (D') \mid 2 \,, \quad \absdeg (D') \leq 2q \, \absdeg (D)  ,
	\]
	and 
	\[
	\begin{aligned}
 & D - D' = \divi(g) \text{ for a function $g$ such that }\\ 
	&\quad g(P) = 1  \mbox{ for all }P \in E(\ol{\F_q}) \mbox{ satisfying }\phi(P)=P+P_0;
	\end{aligned}\]
	\item it is probabilistic with expected runtime polynomial in $q\log(\# k)\absdeg(D)$. 
\end{itemize} 
\end{Proposition}

We now describe our strategy of proof for the above propositions. Let $\paramdeg := \essdeg_k(D)$ (hence $\paramdeg=3$ for Proposition \ref{secondhalfdescentprop} and $\paramdeg=4$ for Proposition \ref{firsthalfdescentprop}). 
Again, $D$ can be supposed to be irreducible over~$k$ and supported outside the traps, i.e.
\[
D= \p + \ldots + \sigma^{\paramdeg -1}\p\,,
\]
for $\p$ a non-trap point on $E$ such that $[k(\p):k]= \paramdeg$, and $\sigma$ a generator of $\Gal(k(\p)/k)$. 

Let $\tau_{P_0}$ be the translation by $P_0$ on $E$ and let  $h \mapsto h^\phi $ be the automorphism of $k(E)$ that ``applies $\phi$ to the coefficients of $h$'' (denoting $x,y$ be the usual coordinates on $E$, we have $x^\phi =x$, $y^\phi=y $ and $\alpha^\phi=\alpha^q $ for all $\alpha \in k$). Using this notation, we rephrase one of Joux's ideas (\cite{Joux}) in the following way: for every point $P \in E(\ol{\F_q})$ such that $\phi(P)=P + P_0$ and for every function $f \in k(E)$ regular on $P$ we have
\begin{eqn}\label{eq:use_P1}
f(P)^q = f^\phi(\phi(P)) = f^\phi(P+P_0) = (f^\phi \circ \tau_{P_0})(P) \,,
\end{eqn}
hence, for all $a,b,c,d \in k$ such that $cf^{q+1}{+}  df^q {+} af {+} b$ does not vanish on $P$, the rational function
\begin{eqn}\label{eq:def_g_from_f}
g:=\frac{( cf + d)(f^{\phi}\circ \tau_{P_0}) + af + b}{  cf^{q+1}+  df^q + af + b} \,,
\end{eqn}
satisfies 
\begin{eqn}\label{eq:samelog}
\text{$g(P) = 1 \quad $ for all $P\in E(\ol{\F_q})$ such that $\phi(P)=P + P_0$.}
\end{eqn}
Hence, for a function $g$ as in (\ref{eq:def_g_from_f}), one of the requirements of Propositions \ref{firsthalfdescentprop} (respectively \ref{secondhalfdescentprop}) is automatically satisfied. In the algorithm we look for a function $g$ of that form. 
We now look for conditions on $f$ and $a,b,c,d$ implying that the function $g$ and the divisor 
\begin{eqn} \label{def:D'}
D':= D - \divi(g) \,,
\end{eqn}
have the desired properties. 

One of these properties is ``$\essdeg_k(D')\leq \paramdeg-1$'', for which it is enough that ``$[k(P):k]\le \paramdeg -1$ for all the points $P$ in the support of \changee{$D'$}. 
In particular the support of $D'$ contains all the poles of $g$, which are either poles of $f$, poles of $f^\phi \circ \tau_{P_0}$ or zeroes of $cf^{q+1}{+}  df^q {+} af {+} b$.  Since the zeroes and poles of a function $h\in k(E)$ are defined over an extension of $k$ of degree at most $\deg(h)$, the following conditions are sufficient to take care of the poles of $g$:
\begin{enumerate}[label=(\Roman*), itemsep=0pt, start=1]
\item \label{item1:degree} the function $f$ has at most $\paramdeg{-}1$ poles counted with multiplicity;
\item \label{item2:split} the polynomial $cT^{q+1} + dT^q + aT + b$ splits into linear factors in $k[T]$.
\end{enumerate}

Another requirement is for $\p$ and all its conjugates to be zeroes of $g$.  Assuming \ref{item1:degree} and \ref{item2:split}, this is equivalent to  
\begin{enumerate}[label=(\Roman*), start=3]
\item\label{item3:gamma} $ \smt abcd \cdot f(\sigma^i\p) = -f^\phi(\sigma^i\p + P_0)\quad $ for $i=0,1,\ldots, \paramdeg{-}1$,
\end{enumerate}
where $(\smt abcd, x) \mapsto \smt abcd \cdot x = \tfrac{ax+b}{cx+d}$ is the usual action of $\PGL$ on $\P^1$. Notice that the definition of $g$ only depends on the class of $\smt abcd$ in $\PGL(k)$. 

Conditions \ref{item1:degree}, \ref{item2:split} and \ref{item3:gamma} together imply that $\essdeg(D')\leq \paramdeg{-}1$: a point $P$ in the support of $D'$ is either a pole of $f$, a pole of $f^\phi \circ \tau_{P_0}$, a zero of $cf^{q+1}{+}  df^q {+} af {+} b$ or a zero of the numerator of (\ref{eq:def_g_from_f}); the only case left to treat is the last, where we have $[k(P):k]\leq \epsilon-2$ because  the divisor of zeroes of the numerator of (\ref{eq:def_g_from_f}), which has degree at most $2(\paramdeg{-}1)$, is larger than the sum of the conjugates of $P$ summed with the conjugates of $\p$. \changee{Notice that we are using another idea from \cite{Joux}, namely that, if $f$ has low degree (i.e. few poles), then the numerator of (\ref{eq:def_g_from_f}) has low degree too; the denominator has probability about $1/q^3$ of splitting into linear polynomials in $f$.}

Condition \ref{item1:degree} easily implies that $\absdeg(D')$ is at most $2q\epsilon$. 

Finally, as noticed when defining $g$, we want
\begin{enumerate}[label=(\Roman*), start=4]
\item\label{item4:traps} for every point $P$ on $E$ such that $\phi(P)=P+P_0$, the function $f$ is regular on $P$ and  $cf^{q+1}{+}df^q{+}af{+}b$ does not vanish on $P$. 
\end{enumerate}

We showed that if $(f, \smt abcd )$ satisfies the conditions \ref{item1:degree}, \ref{item2:split}, \ref{item3:gamma}, \ref{item4:traps}, then the formulas (\ref{eq:def_g_from_f}) and (\ref{def:D'}) give $g$ and $D'$ that satisfy the requirements of Proposition \ref{secondhalfdescentprop}, respectively Proposition \ref{firsthalfdescentprop}. In Section \ref{sec:3-2}, respectively Section \ref{sec:43}, we prove that there are many such pairs $(f, \smt abcd )$ and we give a procedure to find them when $\paramdeg=3$, respectively $\paramdeg=4$. In both cases we proceed as follows: 
\begin{itemize}
\item\label{fase1f} We choose a family of functions $f$ satisfying \ref{item1:degree} and we parametrize them with $k$-points on a variety $\mathcal F$.
\item\label{fase2C} We impose some conditions slightly stronger than \ref{item2:split}, \ref{item3:gamma}, \ref{item4:traps},  describing a variety $\calC \subset \mathcal F{\times}\PGL {\times }\A^1$: for all point $(f,\smt abcd,z) \in \calC(k)$, the pair $(f,\smt abcd)$ satisfies \ref{item1:degree}, \ref{item2:split}, \ref{item3:gamma}, \ref{item4:traps}. 
In particular, $\calC$ is a curve in the case $\epsilon=3$, a surface in the case $\epsilon=4$

\item We prove that the geometrically irreducible components of $\calC$ are defined over~$k$ and we deduce that $\calC(k)$ has cardinality at least $\tfrac{1}{2} (\# k)^{\dim \calC}$; this is the point in the proof where we use the technical hypothesis $[k:\F_q]\ge 80$ (details after Equations (\ref{eq:BettiD32}), (\ref{eq:BettiD43})).
\end{itemize}

Using $\calC$ we can describe the algorithm of  Proposition \ref{firsthalfdescentprop} (respectively Proposition \ref{secondhalfdescentprop}) when $D$ is an irreducible divisor defined over~$k$: one first computes equations for $\calC$ (which we describe explicitly), then looks for a point $(f,\smt abcd, z)$ in $\calC(k)$ and finally computes $g$ and $D'$ using the formulas (\ref{eq:def_g_from_f}) and (\ref{def:D'}). \begin{change}
In both cases, the variety $\calC$ can be defined using less than $9$ variables and $6$ polynomials of degree less than $q^3$ and a few non-equqlities. In particular one can look for points on $\calC$ by first choosing the values of certain variables at random and then trying to solve successively for the other variables one by one. We give more detail on this in Subs ections \ref{subsec:points_32} and \ref{subsec:points_43}, where we also prove that this procedure can be done probabilistically in polynomial time in $q\log(\# k)$.
\end{change}

\begin{Remark}\label{rem:traps}
If $\p \notin \Gal(\ol{\F_q}/\F_q) \cdot P_1$ is a point such that $\phi(\p)=\p+P_0$, then Equation \ref{eq:use_P1} implies that conditions \ref{item3:gamma} and \ref{item4:traps} exclude each other. This explains why such points $\p$ create problems to our strategy and need to be marked as \emph{traps}.
\end{Remark}

\section{Intermezzo: a lemma in Galois theory}\label{sec:lemma}
In this section we take a break from our main topic and we prove Lemma \ref{lem:good_curves}, which we use to prove that the varieties $\calC$ used in the algorithms of Propositions \ref{firsthalfdescentprop} and \ref{secondhalfdescentprop} have geometrically irreducible components defined over~$k$. 
Our method for that is to look at the field of constants: 
for any extension of fields $k \subset \K$, its \emph{field of constants} is the subfield of $\K$ containing all the elements that are algebraic over~$k$. In particular, when $k$ is perfect, an irreducible variety $\calC/k$ is geometrically irreducible iff $k$ is equal to the field of constants of the extension $k\subset k(\calC)$.

In the next proposition we study the splitting field of polynomials of the form $cT^{q+1} {+} dT^q {+} aT {+} b$ (as in condition \ref{item2:split}) over fields with a valuation, so that we can apply it to function fields.

\begin{Proposition} \label{prop:pure_fields}
Let $\F_q \subset k$ be an extension of finite fields and let $k \subset \K$ be a field extension with field of constants $k$. Let $v:\K^\times \to \Z$ be a valuation 
with ring of integral elements $\calO_v \subset \K$ and generator $\pi_v$ of the maximal ideal of $\calO_v$. 
Let $a,b,c,d$ be elements of $\calO_v$ such that 
\begin{subeqn}\label{eq:congruence_derivative}
	\begin{aligned}
& v(ad-bc) =1, \qquad v(d^qc - ac^q)=0 \qquad \text{and} \quad \\
&	\changee{ \exists \gamma \in  \K \text{ such that }   \gamma^q \equiv \tfrac ac  \pmod{\pi_v}   \quad  \text{ and } }
\\ &  \forall \lambda \in \calO_v^\times \quad  c\lambda^q - c^q(ad-bc)\lambda^{-1}  \not\equiv  d^qc - ac^q  \pmod{\pi_v^2} \,.
\end{aligned}\end{subeqn}
\changee{Let $\K(F)$ be} the splitting field, over $\K$, of the polynomial
\[
F(T) := cT^{q+1} + dT^q + aT + b   \quad \in \K[T]\,.
\]
\changee{Then the extension $k\subset \K(F)$ has field of constants equal to $k$}.
\end{Proposition}
\begin{proof}
For any field extension $\K \subset \widetilde{\K}$, we denote $\widetilde{\K}(F)$ the splitting field of $F$ over~$\widetilde{\K}$, which is a separable extension of $\widetilde{\K}$ because the discriminant of $F$ is a power of $ad{-}bc$, which is non-zero.
Since the field of constants of $k \subset \K$ is equal to $k$, then $\K' :=\K\otimes_k \ol k$ is a field and the
statement of the proposition is equivalent to the equality 
\[
\Gal(\K(F)/\K) = \Gal(\K'(F)/\K') \,.
\]
By \cite[Theorems $2.5$ and $3.2$]{HidePGL2} there exists a bijection $\{\text{roots of }F\} \leftrightarrow \P^1(\F_q)$ that identifies the action of $\Gal(\K(F)/\K)$ on the roots with the action of a subgroup of $G := \PGLq$ on $\P^1(\F_q)$. We choose such a bijection and we identify $\Gal(\K(F)/\K)$ and $\Gal(\K'(F)/\K')$ with two subgroups of $G$. 
If we prove that $\Gal(\K'(F)/\K')$ contains a Borel subgroup $B$ of $G$ the proposition follows: the only subgroups of $\PGL$ containing $B$ are the whole $G$ and $B$ itself and, since $B$ is not normal inside $G$, we deduce that either $\Gal(\K(F)/\K)=\Gal(\K(F)/\K')=B$ or $\Gal(\K(F)/\K) =\Gal(\K'(F)/\K')=G$.

In the rest of the proof we show that $\Gal(\K'(F)/\K')$ contains a Borel subgroup working locally at $v$. We choose an extension of $v$ to $\K'$ and consider the completion $\K'_v$ of $\K'$. Since $\Gal(\K'_v(F)/\K'_v)$ is a subgroup of $\Gal(\K'(F)/\K')$, it is enough to show that $\Gal(\K'_v(F)/\K'_v)$ is a Borel subgroup to prove the proposition.

Since $ad{-}bc \equiv 0$ and $c\not\equiv 0$ modulo~${\pi_v}$, we have
\begin{eqn}\label{cong:Fabcd}
F(T) \equiv c\Big(T^q + \frac ac\,\Big) \Big(T + \frac dc\,\Big) \pmod{\pi_v} \,,
\end{eqn}
and, since $d^qc \neq ac^q \bmod{\pi_v}$, we deduce that ${-}\frac dc$ is a simple root of $F \bmod{\pi_v}$. By Hensel's Lemma, there exists a root $r_0 \in \K'_v$ of $F$ that is $v$-integral and congruent to ${-}\tfrac dc$ modulo~$\pi_v$. The group $\Gal(\K'_v(F)/\K_v') \subset G$ stabilizes the element of $\P^1(\F_q)$ corresponding to $r_0$, hence it is contained in a Borel subgroup of $G$. Since Borel subgroups have cardinality $q(q{-}1)$, in order to prove the proposition it is enough showing that $[\K'(F): \K']$ is at least $q(q{-}1)$. We show that the inertia degree of $\K' \subset \K'(F)$ is at least $q(q{-}1)$.

\changee{
	By hypothesis \eqref{eq:congruence_derivative}, $\tfrac ac$ is the $q$-th power modulo $\pi_v$ of an element $\gamma$.
} 
Up to the substitution $F(T) \mapsto F(T-\gamma)$, which does not change $\K'_v(F)$ nor the quantities $c$, $ad{-}bc$ and $d^qc{-}ac^q$, we can suppose that 
\[
F(T) \equiv  c\, T^q \Big(T + \frac dc\,\Big) \pmod{\pi_v} \,.
\]
This implies that
$v(d/c)=0$, $v(a)\ge 1$ and $v(b)\ge 1$. If we had $v(b)\geq 2$, then the choice $\lambda := d$ would contradict the last congruence in (\ref{eq:congruence_derivative}). Hence we have $v(b)=1$. The Newton polygon of $F$ tells us that the roots
$r_0, \ldots, r_q $ of $F$ in the algebraic closure $\ol{\K'_v}$ of $\K'_v$ satisfy 
\begin{eqn}\label{eq:lemma_poly_valuation_roots}
	v(r_0)=0\,, \quad v(r_1) = \ldots = v(r_q)= \frac 1 q\,.
\end{eqn}

We now consider the polynomial
\[
F_1 (T):= F(T + r_1) = c_1T^{q+1} + d_1 T^q + a_1 T + b_1 = cT^{q+1} + d_1 T^q + a_1 T \quad \in \ol{\K'_v}[T]\,.
\]
The roots of $F_1$ are $r_i{-}r_1$. Using Equation (\ref{eq:lemma_poly_valuation_roots}), we deduce $v(c_1)=v(d_1)=0$ and $v(a_1) > 0$. Using  $a_1d_1{-}b_1c_1=ad{-}bc$, we see that 
$$v(a_1) = v(a_1d_1 {-} c_1b_1) = v(ad{-}bc) = 1\,.$$ The Newton polygon of $F_1$ tells us that
\[
v(r_2-r_1)=\ldots = v(r_q-r_1)= \frac{1}{q-1} \,.
\] 
This, together with Equation (\ref{eq:lemma_poly_valuation_roots}) and the fact that $\K \subset \K'$ is unramified, implies	 that the inertia degree of $\K'_v \subset \K'_v(F)$ is a multiple of $q(q{-}1)$ and consequently that $\Gal(\K'_v(F)/\K')$ is a Borel subgroup of $G$.
\end{proof}

We now prove that, for certain choices of $\K, a,b,c,d$, Equation (\ref{eq:congruence_derivative}) is satisfied.

\begin{Proposition}\label{prop:we_can_use_other_prop}
Let $\K$ be a field extension of $\F_q$, let $u_1, u_2, u_3,w_1, w_2, w_3$ be distinct elements of $\K$ and
let $a,b,c,d \in \K$ be the elements defined by the following equality in $\text{GL}_2(\K)$
\[
\mt abcd {=} \mt{w_3^q}{w_1^q}{1}{1} \!\!\mt{w_1^q-w_2^q}{0}{0}{w_2^q-w_3^q} \!\!\mt{u_2-u_3}{0}{0}{u_1-u_2} \!\!\mt{1}{-u_1}{-1}{u_3}\!.
\]
Then $\smt abcd$ sends the three elements $u_1,u_2,u_3 \in \P^1(\K)$ to $w_1^q,w_2^q,w_3^q \in \P^1(\K)$ respectively. 

Suppose, moreover, that $\K$ is equipped with a discrete valuation $v:\K^\times \to \Z$, that $u_i, w_i$ are $v$-integral, that $v(w_i{-}w_j) = v(w_3{+}u_i) = v(u_2{-}u_3) = 0 $ for $i \neq j$ and that $v(u_1{-}u_2)=1$. 
Then $a,b,c,d$ 
satisfy (\ref{eq:congruence_derivative}).
\end{Proposition}
\begin{proof}
To prove the first part we notice that, given distinct elements $x,y,z \in \K$, the matrix
\[
N_{x,y,z} := \mt{z}{x}{1}{1} \mt{x-y}{0}{0}{y-z} 
\]
is invertible and acts on $\P^1(\K)$ sending $0,1,\infty = \spv 10$ to $x,y,z$ respectively. Using this definition we have $
\smt{a}{b}{c}d = \det(N_{u_1,u_2,u_3})  N_{w_1^q,w_2^q,w_3^q}  N_{u_1,u_2,u_3}^{-1}$, hence $\smt abcd$ acts on $\P^1(\K)$
sending
\[
u_1\mapsto 0\mapsto w_1^q\,, \quad u_2\mapsto 1\mapsto w_2^q\,, \quad u_3\mapsto \infty \mapsto w_3^q\,.
\]
Now the second part. Computing $\det(N_{u_1,u_2,u_3})$ and $\det(N_{w_1^q,w_2^q,w_3^q})$ we see that 
\[
ad - bc = (u_1-u_2)(u_2-u_3)(u_1-u_3)(w_1-w_2)^q(w_2-w_3)^q(w_1-w_3)^q
\]
hence $v(ad{-}bc)=v(u_1{-}u_2)=1$ (the other factors have valuation zero by hypothesis or, in the case of $u_3{-}u_1$, because they are the sum of $u_3{-}u_2$ and $u_2{-}u_1$, whose valuations are $0$ and $1$). Writing $a, b, c, d$ as polynomials in the $u_i$'s and the $w_i$'s, we check that there is a multivariate polynomial $f$ such that 
\begin{eqn}\label{eq:sviluppo_inv2}
	\begin{aligned}\label{eq:inv2}
		d^qc - ac^q  =& f(u_1,u_2,u_3,w_1,w_2,w_3)\cdot\big(u_1 - u_2\big)^q \\
		& +	(u_1 - u_3)^q (w_1 - w_2)^{q^2} (w_1 - w_3)^q (u_2 + w_2)^q \cdot\big(u_1 - u_2\big) \\
		& - (w_1 - w_2)^{q^2+q}(u_1 -  u_3)^{q+1} (u_1 + w_3)^q \,.
	\end{aligned}
\end{eqn}
Since $v(w_2{-}w_1) = v(u_3{-}u_1) = v(w_3{+}u_1)=0$, we have $v(d^qc{-}ac^q)=0$. Let $\calO_v$ be the integral sub-ring of $\K$, let $\pi_v:=u_1{-}u_2$, which is a generator of the maximal ideal of $\calO_v$. 

\begin{change}
We are left to check the second and third line of \eqref{eq:congruence_derivative}. Before proceeding, we notice that, since $u_2\equiv u_1 \pmod{\pi_v}$, we have the congruences
	\begin{equation*}
		\begin{aligned}
			& a \equiv  w_3^q (w_1 - w_2)^q (u_1 - u_3),  \quad 
			& b \equiv w_3^q (-w_1 + w_2)^q (u_1 - u_3) u_1 \pmod{\pi_v} ,
			\\ 
			&
			c \equiv  (w_1 - w_2)^q (u_1 - u_3) , 
			\quad
			& d \equiv  (-w_1 + w_2)^q (u_1 - u_3) u_1 \pmod{\pi_v} .
		\end{aligned}
	\end{equation*}
	
	For the second line of \eqref{eq:congruence_derivative} it is enough to take $\gamma = w_3$. 
\end{change}
Now suppose by contradiction that there exists $\lambda \in \calO_v^\times$ such that 
\begin{eqn}\label{eq:absurd_cong}
	c\lambda^q - a^q(ad-bc)\lambda^{-1}  \equiv d^qc - ac^q  \pmod{\pi_v^2}\,.
\end{eqn}
Using also that $ad{-}bc\equiv 0 \bmod{\pi_v}$, we deduce that $\lambda$ must satisfy
\[
\lambda^q \equiv \frac{d^qc - ac^q}{c} \equiv 
\Big(-(u_1- u_3)(u_1 +w_3)(w_1-w_2)^{q}\Big)^q \pmod{\pi_v} \,,
\]
If we replace $\lambda$ by some $\lambda' \equiv \lambda$ modulo $\pi_v$, then the congruences (\ref{eq:congruence_derivative}) are still satisfied, hence we may suppose $\lambda = {-}(u_1{-} u_3)(u_1{+}w_3)(w_1{-}w_2)^{q}$. Substituting $\lambda$ and (\ref{eq:inv2}) in (\ref{eq:absurd_cong}) we get
\[
\begin{aligned}
0 &\equiv c^q(ad-bc) + (d^qc - ac^q)\lambda - c\lambda^{q+1}  \\
& \equiv - \pi_v (w_1{-}w_2)^{q^2+q}(w_1{-}w_3)^q(u_1{-} u_3)^{q+1} (w_2{-}w_3)^q(w_3{+}u_3)\qquad \pmod {\pi_v^2}
\end{aligned}
\]
which is absurd because $v(w_i{-}w_j)=v(u_1{-}u_3)=v(w_3{+}u_3)=0$.
\end{proof}

We now prove the main result of this section. Varieties like $\calC$ in the following lemma arise in Sections \ref{sec:3-2} and \ref{sec:43} when imposing conditions \ref{item2:split} and \ref{item3:gamma}. 
\begin{Lemma}\label{lem:good_curves}
Let $\F_q \subset k$ be an extension of finite fields and let $\calB/k$ be a geometrically irreducible variety. Let
$u_1, u_2, u_3$, $w_1, w_2, w_3$ be distinct elements of  $\ol{k}(\calB)$ and suppose there exists an irreducible divisor $\Zp \subset \calB_{\ol{k}}$, generically contained in the smooth locus of $\calB$, such that $u_i,w_i$ are defined on the generic point of $\Zp$ and such that $Z$ is a zero of order $1$ of  $u_1{-}u_2$ and it is not a zero of $w_3{+}u_i, u_2{-}u_3, w_i{-}w_j$ for $i{\neq}j$.

Let $\calC \subset \calB \times \PGL \times \A^1$ be the variety whose points are the tuples $(R,\smt abcd,z)$ such that
\[
\begin{aligned}
&u_i(R) \text{ are defined and distinct,}\,\, w_i(R)\text{ are defined and distinct,} \,\, d^qc - ac^q \neq 0, \\
&
\qquad \qquad \qquad \smt abcd \cdot u_i(R) = w_i^q(R) \text{ for }i =1,2,3 \qquad \text{and} \\
&\qquad (d^qc - ac^q)^{q+1}(z^q-z)^{q^2-q} = c^{q^2+1}(ad - bc)^q \left((z^{q^2}-z)/(z^q-z)\right)^{q+1}  
\end{aligned}
\]
If $\calC$ is defined over~$k$, then its geometrically irreducible components are defined over~$k$ and pairwise disjoint. \begin{change}
Moreover all components have the same dimension as $\calB$ and the natural projection $\calC \to \calB$, $(R,\smt abcd, z) \mapsto R$, is dominant, with finite fibres and degree $q^3-q$.
\end{change}
\end{Lemma}
\begin{proof}
We first look at the variety $\calB_0 \subset \calB \times \PGL$ whose points are the pairs $(R,A)$ such that
\[
\begin{aligned}
&u_i(R) \text{ are defined and distinct,}\quad w_i(R)\text{ are defined and distinct,} \\
&\qquad \qquad A \cdot u_i(R) = w_i^q(R) \text{ for }i =1,2,3\,.
\end{aligned}
\]
Since an element $\PGL$ is uniquely determined by its action on three distinct points of $\P^1$, the projection $\calB_0 \to \calB$ is a birational equivalence, whose inverse, by the first part of Proposition \ref{prop:we_can_use_other_prop}, is given by $R \mapsto \left( R, \smt{a_1}{b_1}{c_1}{d_1}\!\!(R) \right)$, where $a_1,b_1,c_1,d_1 \in \ol k(\calB)$ are defined by the following equality in $\mathrm{GL}_2(\ol k(\calB))$
\[
\mt{a_1}{b_1}{c_1}{d_1} {=} \mt{w_3^q}{w_1^q}{1}{1}\!\! \mt{w_1^q-w_2^q}{0}{0}{w_2^q-w_3^q}\!\! \mt{u_2-u_3}{0}{0}{u_1-u_2} \!\! \mt{1}{-u_1}{-1}{u_3}\!.
\]
Let $v \colon \ol{k}(\calB)^\times \to \Z$ be the valuation that determines the order of vanishing in $\Zp$ of a rational function. The second part of Proposition \ref{prop:we_can_use_other_prop} implies that $a_1,b_1,c_1,d_1$ satisfy (\ref{eq:congruence_derivative}), over the field $\ol k(\calB)$.
In particular we have $c_1 \neq 0$ and $v(c_1)=0$. Hence we can define the following rational functions on $\calC$
\[
a_2:= a_1/c_1\,, \quad b_2 :=  b_1/c_1\,, \quad c_2 := 1\,, \quad d_2 :=  d_1/c_1
\] 
which again satisfy (\ref{eq:congruence_derivative}) over the field $\ol k(\calB)$. The advantage of $a_2,b_2,c_2,d_2$ is that, as we now show, they are defined over~$k$. 

\begin{change} Let $\calB_1$ be the open subset of $\calB_0$ where $d^qc-ac^q\neq 0$, which is non-empty and birational to $\calB$ since the function  $d_2^qc_2-a_2c_2^q$ is non-zero on $\calB$ by \eqref{eq:congruence_derivative}. 
Our variety $\calC$ is a finite \'etale cover of $\calB_1$ of degree $q^3-q$: indeed $\calC$ is defined as the subscheme of $\calB_1\times \A^1$ containing the points $(S,z)$ where $z$ safisfies a polynomial of degree $q^3-q$ and since we imposed the condition $ad{-}bc\neq 0$ (since $\calB_1 \subset \calB \times \PGL $) and $d^qc-ac^q\neq 0$, the polynomial has distinct roots. In particular the projection $\calC \to \calB$ is dominant (its image is $\calB_1$) with finite fibres and degree $q^3-q$.
\end{change} Since $\calC$ is defined over~$k$, its projection $\calB_1$ is defined over~$k$ and since $a/c$ is a rational function on $\calB_1$ defined over~$k$, we deduce that $a_2 = a/c$ lies in $k(\calB_1)=k(\calB)$ and analogously $b_2, c_2, d_2 \in k(\calB)$. 
A fortiori $a_2,b_2,c_2,d_2$ satisfy (\ref{eq:congruence_derivative}) inside the field $\K= k(\calB)$. 

We can now apply Proposition \ref{prop:pure_fields} and deduce that $k$ is the field of constants of the extension $k\subset \Sigma$, where $\Sigma$ is the splitting field of 
\[
F(T) := c_2T^{q+1} + d_2T^q + a_2 T + b_2 \,,
\]
over~$k(\calB)$. 
We deduce that there exists a geometrically irreducible variety $\mathcal E/k$ having field of rational functions $\Sigma$. Let $\pi \colon \mathcal E \dashrightarrow \mathcal B$ be the rational map induced by $k(\calB) \subset \Sigma$ and let $r_0,\ldots, r_q \in \Sigma$ be the roots of $F$, interpreted as rational functions on $\mathcal E$. Using \cite[Lemma $2.3$]{HidePGL2} we see that, for any choice of pairwise distinct integers $0 \le i,j,m\le q$,
\[
\begin{aligned}
& z = z_{i,j,m}:= \frac{r_i-r_j}{r_i-r_m} \in \Sigma = k(\mathcal E) \quad \text{ satisfies } \\
& (d_2^qc_2 -a_2c_2^q)^{q+1}(z^q-z)^{q^2-q} = c_2^{q^2+1}(a_2d_2 - b_2c_2)^q  \left((z^{q^2}-z)/(z^q-z)\right)^{q+1} \,.
\end{aligned}
\]
Hence, for each  choice of pairwise distinct integers $0 \le i,j,m\le q$ we get a map
\[
\phi_{i,j,m} \colon \mathcal E \dashrightarrow \mathcal C, \qquad 
S\longmapsto \left(\pi(S), \smt {a_2}{b_2}{c_2}{d_2}(S),z_{i,j,m}(S)\right) \,.
\]
Since $\mathcal E$ is geometrically irreducible, then all the images $\phi_{i,j,m}(\mathcal E)$ are also geometrically irreducible. 

Moreover, since all the $z_{i,j,m}$ are different, the maps $\phi_{i,j,m} $ are distinct and as many as the degree of the projection $\calC\to\calB$, implying that the union of all the images $\phi_{i,j,m}(\mathcal E)$ is dense inside $\mathcal C$. This implies that a geometrically irreducible component of $\calC$ is the Zariski closure of an image $\phi_{i,j,m}(\mathcal E)$. Since $\mathcal E$ and $\phi_{i,j,m}$ are defined over~$k$, then also  $\phi_{i,j,m}(\mathcal E)$ and the geometrically irreducible components of $\calC$ are defined over~$k$. \begin{change}
	Since $\Sigma = k(\mathcal E)$ is a finite extension of $k(\cal B)$, then $\phi_{i,j,m}(\mathcal E)$ and all irreducible components of $\calC$ have the same dimension of $\calB$.
\end{change}

Finally, we prove that the components of $\calC$ are pairwise disjoint. The projection $\pi \colon \calC \to \calB_1$ has finite fibers whose number of $\ol k$-points counted with multiplicity is $q^3{-}q$ , that is the degree, in $z$, of the polynomial 
\[
(d^qc - ac^q)^{q+1}(z^q-z)^{q^2-q} - c^{q^2+1}(ad - bc)^q \left((z^{q^2}-z)/(z^q-z)\right)^{q+1}.
\]
If, by contradiction, there is a point $(R',\smt {a'}{b'}{c'}{d'}, z')$ lying in the intersection of two components of $\calC$, then the fiber $\pi^{-1}(R', \smt{a'}{b'}{c'}{d'})$ has cardinality smaller than $q^3{-}q$, which is equivalent to the following polynomial in $\ol{\F_q}[z]$ having less than $q^3{-}q$ roots 
\[
G(z):= (d'^qc' - a'c'^q)^{q+1}(z^q-z)^{q^2-q} - c'^{q^2+1}(a'd' - b'c')^q \left((z^{q^2}-z)/(z^q-z)\right)^{q+1} \,.
\]
Since $a'd' {-} b'c'\neq 0$ and $d^qc' {-} a'c'^q \neq 0$, there is no root of $G$ that is also a root of $z^q{-}z$ or $\frac{z^{q^2}{-}z}{z^q{-}z}$. In other words, the roots of $G$ lie outside the finite field $\F_{q^2}$ with $q^2$ elements. Moreover, since $G$ is a $\ol{\F_q}$-linear combination of powers of $z^q{-}z$ and $\frac{z^{q^2}{-}z}{z^q{-}z}$, the set of roots of $G$ is stable under the action of $\PGLq$. Since this action is free on $\ol{\F_q}\setminus\F_{q^2}$, the set of roots is large at least $\# \PGLq = q^3{-}q = \deg G$, which is a contradiction.
\end{proof}

\begin{Remark}\label{rem:split_field}
Let $\F_q \subset k$ be a field extension and let $F= cT^{q+1} {+} dT^q {+} a T {+} b$ be a polynomial with coefficients in $k$ such that, $ad{-}bc \neq 0$ and $a^qc{-}dc^q \neq 0$.
By \cite[Theorem $4.3$ and Lemma $2.3$]{HidePGL2}, the  polynomial  $F$ splits in linear factors over~$k$ if and only if there exists an element $z \in k$  such that
\[
(d^qc - ac^q)^{q+1}(z^q-z)^{q^2-q} = c^{q^2+1}(ad - bc)^q \left((z^{q^2}-z)/(z^q-z)\right)^{q+1}\,. 
\]
In particular, in the notation of Lemma \ref{lem:good_curves}, the field of rational functions of any component of $\calC$ is the splitting field of $F$.
\end{Remark}

\section{Descent 3-to-2}\label{sec:3-2}
In this section we finish the proof of Proposition \ref{secondhalfdescentprop}, started in Section \ref{sec:idea_descent}. Following the notation of Section \ref{sec:idea_descent} when $\paramdeg=3$, let $k$ be a finite extension of $\F_q$ of degree at least $80$, let $\p$ be a non-trap point on $E$ such that $[k(\p):k]=3$, and let $\sigma$ be a generator of $\Gal(k(\p)/k)$. Then, we look for a function $f \in k(E)$ and a matrix $\smt abcd \in \PGL(k)$ satisfying properties \ref{item1:degree}, \ref{item2:split}, \ref{item3:gamma}, \ref{item4:traps}: we describe a curve $\calC$ whose $k$-points give such pairs $(f,\smt abcd)$,  and we prove that  $\#\calC(k)> \tfrac 12 (\# k)$. \changee{We also give an algorithm to find a point in $\calC(k)$ in polynomial time in $q\log(\# k)$}

\subsection{The definition of $\calC$} \label{sec:3-2_defC}
Property \ref{item1:degree} requires that $f \in k(E)$ has at most two poles: we look for $f$ of the form 
\begin{subeqn}\label{eq:f_P}
f_P := \frac{y - y(P)}{x-x(P)}
\end{subeqn}
with $P$ in $E(k)\setminus\{O_E\}$. Indeed $f_P$ has exactly two simple poles: $O_E$ and ${-}P$.

Property \ref{item2:split} requires the polynomial $cT^{q+1} {+ }dT^q {+ }a T {+ }b$ to split completely in $k$, for which, as recalled in Remark \ref{rem:split_field}, 
it is sufficient that $d^qc \neq ac^q$ and that
\begin{subeqn}\label{eq:32_item2}
(d^qc - ac^q)^{q+1}(z^q-z)^{q^2-q} = c^{q^2+1}(ad - bc)^q \left((z^{q^2}-z)/(z^q-z)\right)^{q+1}\,,
\end{subeqn}
for some $z$ in $k$.

Notice that definition (\ref{eq:f_P}) makes sense for $P\in E(\ol{\F_q})\setminus\{O_E\}$ and that we have the following symmetry: $f_P(P')=f_{P'}(P)$ for all $P,P' \in E(\ol{\F_q})\setminus O_E$.
Using this and the fact that  $h^\phi(\phi(P))=h(P)^q$ for all $h \in \ol{\F_q}(E)$ and $P \in E(\ol{\F_q})$, we have 
\[
\begin{aligned}
f_P(\sigma^i\p) = f_{\sigma^i\p}(P)  \,, \quad 
f_P^\phi(\sigma^i\p {+} P_0) =  f_P^\phi(\phi(\sigma^i\Pp)) = f_P(\sigma^i\Pp)^q = f_{\sigma^i\Pp}(P)^q \,,
\end{aligned}
\]
where $\Pp$ is the unique point on $E$ such that $\phi(\Pp) = \p + P_0$. Hence \ref{item3:gamma} is equivalent to 
\begin{subeqn}\label{eq:32_item3}
\smt abcd\cdot f_{\sigma^i \p}(P) = - f_{\sigma^i\Pp}(P)^q  \quad \text{for each } i =0,1,2\,.
\end{subeqn}
We now impose \ref{item4:traps}. Let $\Bp$ be a point on $E$
such that $\phi(\Bp)=\Bp+P_0$. 
If the rational function $cf_P^{q+1}{+}df_P^q{+} af_P {+} b$ vanishes on $\Bp$, then we have $\smt abcd \cdot f_\Bp(P) = -f_\Bp(P)^q$. This, if we also assume Equation (\ref{eq:32_item3}) and that $f_{\sigma^i\p}(P)$ are distinct, implies that the cross ratio of $f_{\p}(P)$, $f_{\sigma \p}(P)$, $f_{\sigma^2 \p}(P)$, $f_\Bp(P)$ equals the cross ratio of $f_{\Pp}(P)^q, f_{\sigma \Pp}(P)^q, f_{\sigma^2 \Pp}(P)^q$, $f_\Bp(P)^q$.
Hence, assuming (\ref{eq:32_item3}),  condition \ref{item4:traps} is implied by 
\begin{subeqn}\label{eq:32_item4bis}
f_{\sigma^i\p}(P)\neq f_{\sigma^j\p}(P)\text{ for all }  i\neq j \in \{0,1,2\}
\end{subeqn}
together with
\begin{subeqn}\label{eq:32_item4}
\begin{aligned}
&\text{for} \text{ 
		all } \Bp \text{ such that } \phi(\Bp) = \Bp+P_0 : \qquad  \S\neq {-}B \quad \text{and} \\
	&\!\CR\big (f_{\p}(P), f_{\sigma \p}(P), f_{\sigma^2 \p}(P), f_\Bp(P)) \neq \CR(f_{\Pp}(P), f_{\sigma \Pp}(P), f_{\sigma^2 \Pp}(P), f_\Bp(P))^q \,,
\end{aligned}
\end{subeqn}
where, given four elements $\lambda_1, \lambda_2, \lambda_3, \lambda_4 \in \P^1(\ol{\F_q})$, we denote their cross-ratio by
\[
\CR(\lambda_1,\lambda_2,\lambda_3,\lambda_4) = \frac{(\lambda_3-\lambda_1)(\lambda_4-\lambda_2)}{(\lambda_2-\lambda_1)(\lambda_4-\lambda_3)}  \in \P^1(\ol{\F_q}) \,.
\]
Finally we define  $E':=E\setminus\{O_E, {-}\p, {-}\Pp,\ldots, {-}\sigma^2\p, {-}\sigma^2\Pp\}$, so that $f_{\sigma^i\Pp}$ and $f_{\sigma^i\p}$ are regular on $E'$, and we define $\calC \subset E' \times \PGL \times \A^1$ as the curve made of points $(P,\smt abcd,z)$ that satisfy Equations (\ref{eq:32_item3}), (\ref{eq:32_item2}), (\ref{eq:32_item4bis}) and (\ref{eq:32_item4}), and $d^qc {-}ac^q\neq 0$. 

Notice that even though the equations $\smt abcd f_{\sigma^i \p}(P) = -f_{\sigma\Pp}^q(P)$ have coefficients in the field $k(\p)$, the curve $\calC$ is defined over~$k$: the Galois group of $ k \subset k(\p)$ permutes these equations $\smt abcd f_{\sigma^i \p}(P) = -f_{\sigma\Pp}^q(P)$.
We constructed $\calC$ so that, for any point $(P,\smt abcd,z) \in \calC(k)$, the pair $(f_P,\smt abcd)$ satisfies properties \ref{item1:degree}, \ref{item2:split}, \ref{item3:gamma} and \ref{item4:traps}.
\subsection{The irreducible components of $\calC$}
In this subsection we prove that all the geometrically irreducible components of $\calC$ are defined over~$k$. We can leave out (\ref{eq:32_item4}) from the definition of $\calC$. Our strategy is applying Lemma \ref{lem:good_curves} to the variety $\calB = E'$, using the rational functions $u_i =f_{\sigma^{i-1}\p}$,
$w_i ={-}f_{\sigma^{i-1}\Pp}$ and the irreducible divisor $\Zp$ equals to the point ${-}\p{-}\sigma{\p} \in \calB(\ol{\F_q}) \subset E(\ol{\F_q})$.

Notice that, given distinct points $P', P'' \in E(\ol{\F_q})\setminus\{O_E\}$, the function $f_{P'}{-}f_{P''}$ is regular at $O_E$ and moreover $(f_{P'}{-}f_{P''})(O_E)=0$.  Since the sum of zeroes and poles of a rational function is equal to $O_E$ in the group $E(\ol{\F_q})$, we deduce that, given distinct points $P', P'' \in E(\ol{\F_q})\setminus\{O_E\}$,
\begin{subeqn}\label{eq:zeroes_superuseful}
\begin{aligned}
	&f_{P'} {-} f_{P''} \text{ has two simple poles, namely ${-}P'$ and ${-}P''$} \\  &\text{and two zeroes counted with multiplicity, i.e. $O_E$ and ${-}P'{-}P''$.}
\end{aligned}
\end{subeqn}
Let $\Zp := {-}\p{-}\sigma{\p}$. By (\ref{eq:zeroes_superuseful}) and the fact that $\p$ is not a trap, the point $\p$ is not a pole of any of the $u_i$ and the $w_i$ and it is not a zero of any of the functions 
$u_2{-}u_3$, $w_3{+}u_i$ and $w_i{-}w_j$ for $i \neq j$: if, for example, ${-}f_{\sigma \Pp}$ is not regular on $\Zp$, then $\Zp = {-}\Pp$. Hence, using that $\sigma$ acts as $\phi^l$ on $E(\ol{\F_q})$ for $l:=[k:\F_q]$, we have
\[
Q + P_0 = \phi(\Pp) = \phi(-\Zp) = \phi^{l+1}(\p) + \phi(\p)  \, \implies  \, \phi^{l+1}(\p) = (1-\phi)(\p) + P_0 \,, 
\]
hence
\begin{subeqn}\label{eq:use_traps}
\begin{aligned}
	& \phi^{3}(\p) = \phi^{3l+3}(\p) = \phi^{2l+2}\left( (1-\phi)(\p) + P_0 \right) = ((1-\phi){\circ} \phi^{2l+2})(\p) + P_0  \\
	&  = ((1-\phi){\circ}\phi^{l+1})\left( (1-\phi)(\p) + P_0 \right) + P_0  = ((1-\phi){\circ}(1-\phi))(\phi^{l+1}(\p)) + P_0 \\
	& = (1-\phi)^2\left( (1-\phi)(\p) + P_0 \right) + P_0 = (1-\phi)^3(\p) + P_0 \,,
\end{aligned}
\end{subeqn}
implying that
\[
((2\phi-1){\circ}(\phi^2-\phi+1))(\p) = (\phi^3 + (\phi-1)^3)(\p) = P_0\,,
\]
which contradicts the hypothesis that $\p$ was not a trap point. Moreover, by (\ref{eq:zeroes_superuseful}), the function $f_{\p}{-}f_{\sigma \p}$ has a simple zero in $\Zp$.
 
\begin{change}
Hence, applying Lemma \ref{lem:good_curves} with $\calB = E'$, and $u_i =f_{\sigma^{i-1}\p}$,
$w_i ={-}f_{\sigma^{i-1}\Pp}$, we deduce that, up to forgetting (\ref{eq:32_item4}) from the definition of $\calC$, then  all the geometrically irreducible components of $\calC$ are curves defined over~$k$ that map dominantly on $E'$. This does not change after imposing condition (\ref{eq:32_item4}), which only restricts to an open dense of $E'$.
\end{change}

\subsection{$k$-rational points on $\calC$ and where to find them}\label{subsec:points_32}
We now prove that $\#\calC(k)$ is larger than $\tfrac 1{2} \# k$. The curve $\calC$ is contained in the open subset of $(E \setminus \{O_E\}) \times \PGL \times \A^1$ made of points $((x,y), \smt abcd, z)$ such that $c \neq 0$. Hence  $\calC$ is contained in $\A^6$, with variables $x,y$, $a,b,d$, $z$ and it is defined by the following equations:
\begin{itemize}
\item $0=p_1:=W(x,y)$, the Weierstrass equation defining $E$;
\item $0=p_2 :=(d^q {-} a)^{q+1}(z^q{-}z)^{q^2{-}q} - (ad {-} b)^q (\tfrac{z^{q^2}{-}z}{z^q{-}z})^{q+1}$, 
the dehomogenization of (\ref{eq:32_item2}) in $c$; 
\item $0=p_i(x,y,a,b,d)$ for $i = 3,4,5$, obtained by (\ref{eq:32_item3})  after dehomogenizing in $c$,  substituting $f_{\sigma^i\p}(P)$ and $f_{\sigma^i\Pp}(P)$ by their expressions in $x,y$ and clearing denominators;
\item a number of conditions $0\neq q_j$ ensuring that $\S\neq {-}\sigma^i\p$, $\S \neq {-} \sigma^i\Pp$, $d^q {-} a\neq 0$, $ad-b\neq0$, that $f_{\sigma^i\p}(P)$ are pairwise distinct 
and that (\ref{eq:32_item4}) is satisfied.
\end{itemize} 
In particular, $\calC$ can be seen as a closed subvariety of $\A^7$, with variables $x,y$, $a,b,d,z$ and $t$, and it is  defined by the equations $p_1=0, \ldots, p_5=0$ and $0=p_6:=tq_1\cdots q_r -1$. 

Let $\calC_1, \ldots, \calC_s$ be the irreducible components of $\calC$. 
By \cite[Remark $11.3$]{Npoints}, 
we have
\begin{subeqn}\label{eq:Weil32}
\# \calC(k) \ge \# \calC_1(k) \ge \# k - (\totaldeg-1)(\totaldeg-2)(\# k)^{\frac 12} - K(\calC_1) \,,
\end{subeqn}
where $\totaldeg$ is the degree of $\calC_1$ and $K(\calC_1)$ is the sum of the Betti numbers of $\calC$ relative to the compact $\ell$-adic cohomology. Since $\calC_1$ is a component of $\calC$ then 
\begin{subeqn}\label{eq:degD32}
\totaldeg 
\leq \deg(p_1)\cdots \deg(p_6) \,.
\end{subeqn}
Since $\calC$ is the disjoint union of the $\calC_i$, the Betti numbers of $\calC$ are the sums of the Betti numbers of the $\calC_i$ and using \cite[Corollary of Theorem 1]{Katz} we deduce that 
\begin{subeqn}\label{eq:BettiD32}
K(\calC_1) \leq K(\calC) \leq 6 \cdot 2^{6}\cdot \left(3+6\max_{i=1,\dots ,6}\{\deg(p_i)\}\right)^{8}\,.
\end{subeqn}
Since $\deg p_1 \le 3$, $\deg p_2 \le q^3{+}q$, $\deg p_3, \ldots, \deg p_5 \le q{+}2$, and finally $\deg p_6 \le 8q^2+29q+29$, then Equations (\ref{eq:Weil32}), (\ref{eq:degD32}) and (\ref{eq:BettiD32}) imply that $\#\calC(k)> \tfrac 12 (\# k)$ when $\# k \ge q^{80}$ and $q \ge 3$.

\begin{change}
To finish the proof Proposition \ref{firsthalfdescentprop} we describe a procedure to find a $k$-rational point on $\calC$. 
The algorithm in Proposition \ref{firsthalfdescentprop} then works as explained at the end of Section \ref{sec:idea_descent}.

We look for a point on $\calC$ in a probabilistic way, by successively finding the roots of univariate polynomials (e.g. by Cantor-Zassenhaus \cite{FactorCZ} or by von zur Gathen-Shoup, \cite{FactorShoupGathen}). More in detail we proceed as follows.
\begin{itemize}
	\item First take a random $x\in k$ and look for all $y\in k$ such that $P = (x,y)$ lies on $E$; if the Weiestrass equation $p_1$ does not have its two roots $y$ in $k$ we change $x$ with a new random element of $k$.
	\item For all $y$ computed above, compute $\smt abcd \in \PGL(k)$ satisfying \eqref{eq:32_item3}; this is a matter of solving linear equations (the polynomials $p_3,p_4,p_5$) which always have a unique solution when the values $f_{\sigma^i\p}(P)$ are distinct and the values $f_{\sigma^i\Pp}(P)$ are distinct (is they are not distinct we restart from the first step, in light of \eqref{eq:32_item4bis}).
	\item For all $y,\smt abcd$, compute an element $z\in k$ satisfying \eqref{eq:32_item2}, or equivalently $p_2(z)=0$; the polynomial $p_2(z)$ has degree at most $q^3-q$ and if it has no roots	 $z\in k$, we restart from the first step.
	\item Finally, check if inequalities \eqref{eq:32_item4} hold (this is equivalent to looking for a zero $t$ of $p_6$): if they hold then $(P,\smt abcd, z)$ is a point on $\calC$, if they never hold we restart from the first step.
\end{itemize}
The expected running time of each step is polynomial in $q\log(\# k)$ (see e.g. \cite[Theorem 3.5 and 3.7]{FactorShoupGathen} for an upper bound on the expected running time of factorization). Moreover, for a given $x\in k$, the probability that the above steps go through without starting over, finding a point in $\calC(k)$, is at least $\frac{1}{4q^3}$: this is equivalent to saying that there are at least $\frac{\# k}{4q^3}$ values of $x$ that lift to a point in $\calC(k)$, or equivalently that the map $\calC(k) \to \A^1(k)$ given by the $x$-coordinate has at least $\frac{\# k}{4q^3}$ distinct images, which is true because $\# \calC(k) > \tfrac 12 \# k$, and the map $x\colon \calC \to \A^1$ has finite fibres and degree $2(q^3-q)$, since it is the composition of the degree-$2$ map $x\colon E \to \P^1$ and the projection $\calC \to E$ which has degree $q^3{-}q$ by Lemma \cref{lem:good_curves}.

Hence, the expected number of times each step is repeated is at most $4q^3$, and the expected running time of the above procedure is polynomial in $q\log(\# k)$.

\end{change}

\section{Descent 4-to-3}\label{sec:43}
In this section we finish the proof of Proposition \ref{firsthalfdescentprop}, started in Section \ref{sec:idea_descent}. Following the notation of Section \ref{sec:idea_descent} with $\paramdeg=4$, let $k$ be a finite extension of $\F_q$ of degree at least $80$, let $\p$ be a non-trap point on $E$ such that $[k(\p):k]=4$, and let $\sigma$ be a generator of $\Gal(k(\p)/k)$. Then, we look for a function $f \in k(E)$ and a matrix $\smt abcd \in \PGL(k)$ satisfying properties \ref{item1:degree}, \ref{item2:split}, \ref{item3:gamma}, \ref{item4:traps}: we describe a surface $\calC$ whose $k$-points give such pairs $(f,\smt abcd)$, we prove that there are many $k$-points on $\calC$, \changee{and that one such point can be computed in polynomial time in $q\log(\# k)$}.

\subsection{The definition of $\calC$}\label{subsec:def43}

Property \ref{item1:degree} requires that $f \in k(E)$ has at most $3$ poles: we look for $f$ of the form 
\begin{subeqn}\label{eq:43_def_f}
f=f_{\alpha,\beta,\S} := \frac{f_\S + \alpha}{f_{\widetilde{\S}}+\beta} \,,
\end{subeqn}
where $\alpha, \beta$ are elements of $k$, the points $\S, \widetilde{\S}$ lie in $E(k)\setminus \{O_E\}$ and $f_\S$ is the rational function defined in (\ref{eq:f_P}). We see directly from the definition that the function $f_{\alpha, \beta,\S}$ has at most three poles counted with multiplicity, namely ${-}\S$ and the zeroes of $f_{\widetilde{\S}}{+}\beta$.

Notice that (\ref{eq:43_def_f}) makes sense for any $\S \in E(\ol{\F_q})$ and $\alpha, \beta \in \ol{\F_q}$.

\changee{Let $\Pp \in E(\ol{\F_q})$ be the unique point such that $\phi(\Pp) = \p + P_0$.} 
 For the rest of the section we let $\alpha, \beta$ and $\S$ vary and we fix $\widetilde{\S}$ so that 
$f_{\p}(\widetilde{\S})$, $f_{\sigma \p}(\widetilde{\S})$, $ f_{\sigma^2 \p}(\widetilde{\S})$, $f_{\sigma^ 3 \p}(\widetilde{\S})$, $f_{\Pp}(\widetilde{\S})$, $f_{\sigma \Pp}(\widetilde{\S})$, $ f_{\sigma^2 \Pp}(\widetilde{\S})$, $f_{\sigma^ 3 \Pp}(\widetilde{\S})$ are pairwise distinct and \changee{all distinct from $f_{B}(\widetilde{\S})$ for all points $B$ such that $\phi(B) = B+P_0$.
There is at least one such point $\widetilde{\S}$ because $\# (E(k)\setminus\{O_E\})$ is larger than the number of points we are excluding, which, using (\ref{eq:zeroes_superuseful}), is at most $16(2q +8)$.}
 We sometimes write $f$ for $f_{\alpha,\beta,\S}$.

By Remark \ref{rem:split_field}, for condition \ref{item2:split} to be true it is enough that $d^qc {-} ac^q \neq 0$, and that 
\begin{subeqn}\label{eq:43_item2}
(d^qc - ac^q)^{q+1}(z^q-z)^{q^2-q} = c^{q^2+1}(ad - bc)^q \left((z^{q^2}-z)/(z^q-z)\right)^{q+1}\,,
\end{subeqn}
for some $z\in k$.

Using the definition of $\Pp$ and that $h^\phi(\phi(P))=h(P)^q$ for all $h \in \ol{\F_q}(E)$ and $P \in E(\ol{\F_q})$, we have
\[
\begin{aligned}
-f^\phi(\sigma^i\p {+} P_0) = - f^\phi(\phi(\sigma^i\Pp)) = -f(\sigma^i\Pp)^q \,.
\end{aligned}
\]
In particular,  property \ref{item3:gamma} is equivalent to
\begin{subeqn}\label{eq:43_item3}
\smt abcd\cdot f_{\alpha,\beta, \S}(\sigma^i\p) = - f_{\alpha,\beta, \S}(\sigma^i\Pp)^q \quad \text{for } i =0,1,2,3\,.
\end{subeqn}
The above equation can be further manipulated.
Since cross-ratio is invariant under the action of $\PGL$ on $\P^1$, the above equation implies that either the cross-ratio of $f(\sigma^0\p), \ldots,f(\sigma^3\p)$
is equal to the cross ratio of
$f(\sigma^0\Pp), \ldots$, $f(\sigma^3\Pp)$, or that both cross-ratios are not defined. Conversely, supposing that these cross ratios are defined and equal, condition  (\ref{eq:43_item3}) is equivalent to the same condition, but only for $i=0,1,2$.
In other words, if $f(\sigma^i\p)$ are pairwise distinct and $f(\sigma^i\Pp)$ are pairwise distinct, 
Equation (\ref{eq:43_item3}) is equivalent to 
\begin{subeqn}\label{eq:43_cross_ratio}
\CR\big( f_{\alpha,\beta, \S}(\sigma^0\p), \ldots , f_{\alpha,\beta, \S}(\sigma^3\p )\big)  = \CR\big( f_{\alpha,\beta, \S}(\sigma^0\Pp)^q, \ldots , f_{\alpha,\beta, \S}(\sigma^3\Pp )^q\big)\,,
\end{subeqn}
together with 
\begin{subeqn}\label{eq:43_item3_revisited}
\smt abcd\cdot f_{\alpha,\beta, \S}(\sigma^i\p) = - f_{\alpha,\beta, \S}(\sigma^i\Pp)^q \quad \text{for } i =0,1,2\,.
\end{subeqn}

Assuming Equation (\ref{eq:43_item3_revisited}) and the distinctness of $f(\sigma^i \p)$ and $f(\sigma^i\Pp)$,  condition \ref{item4:traps} is implied by 
\begin{subeqn}\label{eq:43_item4_main} 
\begin{aligned}
	& \text{for} \text{ all } \Bp \text{ such that } \phi(\Bp) = \Bp+P_0 \text{ we have }  \quad \S\neq {-}B\,,  \quad \beta {+} f_{\widetilde{\S}}(\Bp) \neq 0, 
	\\ & \,\,\,  \CR\big (f(\p), f(\sigma\p), f(\sigma^2\p), f(\Bp)\big)  \neq   \CR\big (f(\Pp)^q\!, f(\sigma\Pp)^q\!, f(\sigma^2\Pp)^q\!, f(\Bp)^q \big) .
\end{aligned}
\end{subeqn}
\begin{change}
	Indeed the first two conditions imply that $B$ is not a pole of $f = f_{\alpha, \beta,\tilde P}$; 
	moreover if $cf^{q+1}+df^q +af+b$ vanished on $B$, we would have
	\[
	\tfrac{af(B)+b}{cf(B)+d} = -f^q(B),
	\]
	i.e. the homography $\smt abcd\colon \P^1 \to \P^1$ would send send $f(B) \mapsto - f^q(B)$ and, since we also know that it sends $f(\sigma^i Q) \mapsto -f(\sigma^i R)^q$, the cross ratio of $f^q(Q), f^q(\sigma^1Q), f^q(\sigma^2Q), f^q(B)$ would be equal to the cross ratio of  $f^q(R), f^q(\sigma^1R), f^q(\sigma^2R), f^q(B)$, contradicting \eqref{eq:43_item4_main} . 
\end{change}
Finally we define $E':=E \setminus\{O_E,{-}\sigma^0 \p, {-} \sigma^0 \Pp,\ldots, {-}\sigma^3\p, {-} \sigma^3\Pp \}$   and $\calC$ to be the surface inside $\A^2\times E' \times\PGL \times  \A^1$ made of points $(\alpha, \beta, P, \smt abcd, z)$ that satisfy Equations (\ref{eq:43_cross_ratio}), (\ref{eq:43_item3_revisited}), (\ref{eq:43_item2}) and (\ref{eq:43_item4_main}), and such that 
$\beta + f_{\widetilde{\S}}(\sigma^i\p) \neq 0$, $\beta +  f_{\widetilde{\S}}(\sigma^i\Pp) \neq 0$, $d^qc-ac^q\neq0$, the $f(\sigma^i\p)$ are distinct and
the $f(\sigma^i\Pp)$ are distinct. 

The definition of $E'$ and the conditions $\beta + f_{\widetilde{\S}}(\sigma^i\p) \neq 0$, $\beta +  f_{\widetilde{\S}}(\sigma^i\Pp) \neq 0$,  ensure that  $f(\sigma^i\p)$ and $f(\sigma^i\Pp)$ are well defined.
Arguing as in subsection \ref{sec:3-2_defC}, we see that $\calC$ is defined over~$k$. By construction, for all $(\alpha, \beta, P, \smt abcd, z) \in \calC(k)$, the pair $(f_{\alpha, \beta,P}\smt abcd)$ satisfies \ref{item1:degree}, \ref{item2:split} and \ref{item3:gamma} and \ref{item4:traps}.

\subsection{Irreducibility of a projection of $\calC$}
Before studying the irreducible components of $\calC$, we study the closure in $\P^2 \times E$ of the projection of $\calC$ in $\A^2 \times E$. Let $\calB'\subset \A^2 {\times} E'$ be the variety whose points are the tuples $(\alpha, \beta, \S)$ such that 
\[
\begin{aligned} 
& f_{\alpha, \beta, \S}(\sigma^i\p) \text{ are pairwise distinct}\,, 
\quad 
f_{\alpha, \beta, \S}(\sigma^i\Pp) \text{ are pairwise distinct}, \\
& \qquad \qquad \qquad f_{\widetilde{\S}}(\sigma^i\p)+\beta \neq 0\,, \quad  f_{\widetilde{\S}}(\sigma^i\Pp) + \beta \neq 0\,, \\
& \CR\big( f_{\alpha,\beta, \S}(\sigma^0\p), \ldots , f_{\alpha,\beta, \S}(\sigma^3\p )\big)  = \CR\big( f_{\alpha,\beta, \S}(\sigma^0\Pp)^q, \ldots , f_{\alpha,\beta, \S}(\sigma^3\Pp )^q\big)\,,
\end{aligned}
\]
and let $\calB$ be the closure of $\calB'$ inside $\P^2 \times E$.
The Galois group of $\ol k /k$ permutes the above inequalities and leaves the last equality stable, hence $\calB'$ and $\calB$ are defined over $k$. 
 
\begin{change}
Since the variety $\calB'$ is defined by one equality of rational functions together with  some  non-equalities inside a dimension$~3$ variety, we expect $\calB'$, hence $\calB$, to be a surface. We argue more in detail. In the open $U\subset \A^2\times E'$ where the non-equalities hold, $\calB'$ is defined by the 
vanishing of a regular function $F$ (the non-equalities imply that the two cross-ratios are regular on $U$).
Then $\calB'$  is a surface if and only if $F$ is not invertible nor identically zero in $k[U]$. It is enough to see this in a localization of $ k[U]$, e.g. when we invert all non-zero elements of $k[E][\beta]$. 
In this setting, after clearing denominators which are invertible on $U$, the function $F$ becomes a polynomial in $\alpha$ and the non-equalities can be written in the form $L(\alpha)\neq 0$ with $L$ being either linear polynomial 
in $k(E)(\beta)[\alpha]$ or an element of $k(E)(\beta)^\times$. 
We obtain our claim from the fact that if $L$ is a ``constant'' polynomial it does not divide all coefficients of $F$ and if $L$ is linear, then it is coprime to $F$  (for example, modulo $f_{\alpha, \beta, \S}(\sigma\Pp){-}f_{\alpha, \beta, \S}(\Pp)$, the function $F$ becomes the product of $ f_{\alpha,\beta, \S}(\sigma^2\p){-} f_{\alpha,\beta, \S}(\p)$, $ f_{\alpha,\beta, \S}(\sigma^3\p){-} f_{\alpha,\beta, \S}(\sigma\p)$, $ (f_{\alpha,\beta, \S}(\sigma\Pp){-} f_{\alpha,\beta, \S}(\Pp))^q$ and  $(f_{\alpha,\beta, \S}(\sigma^3\Pp){-} f_{\alpha,\beta, \S}(\sigma^2\Pp))^q$ which are powers of linear polynomials in $\alpha$ and are not a multiple of our chosen modulus as it can be checked by looking at the coefficients of $\alpha^0$ and $\alpha^1$).

\end{change}

In the remainder of the subsection we prove that for all but a few choices of $\S \in E(k)$ the curve $\calB_\S := \calB \cap (\P^2{\times} \{\S\})$ is reduced and geometrically irreducible, which implies the same for $\calB$. 

We first write an equation for $\calB_\S$ in $\P^2$. 

Using the definition of $f_{\alpha,\beta,P}$ we get 
\[
\begin{aligned}
&f_{\alpha, \beta, P}(\sigma^i\p){-}f_{\alpha, \beta,P}(\sigma^j\p)  =  \frac{L_{i,j}(\alpha, \beta,1)}{\big( l_i{+}\beta\big) \big( l_j{+}\beta\big)}, \\
&f_{\alpha, \beta, P}(\sigma^i\Pp){-}f_{\alpha, \beta,P}(\sigma^j\Pp)  =  \frac{R_{i,j}(\alpha, \beta,1)}{\big( r_i{+}\beta\big) \big( r_j{+}\beta\big)}, 
\end{aligned}
\]
where $l_i := f_{\widetilde{\S}}(\sigma^i\p)$, $r_i := f_{\widetilde{\S}}(\sigma^i\Pp)$ and $L_{i,j},R_{i,j}$ are the linear polynomials 
\begin{subeqn}\label{eq:def_Lij_Rij}
\begin{aligned} 
	L_{i,j} 
	&:=  \big( l_j {-} l_i \big) \alpha + \big( f_{\sigma^i\p}(\S) {-} f_{\sigma^j\p}(\S)\big)\beta + \big( f_{\sigma^i\p}(\S) l_j {-} f_{\sigma^j\p}(\S) l_i \big) \gamma , \\
	R_{i,j} 
	&:=  \big( r_j {-} r_i \big) \alpha + \big( f_{\sigma^i\Pp}(\S) {-} f_{\sigma^j\Pp}(\S)\big)\beta + \big( f_{\sigma^i\Pp}(\S) r_j  {-} f_{\sigma^j\Pp}(\S) r_i \big) \gamma \,. 
\end{aligned}
\end{subeqn}
Rewriting Equation (\ref{eq:43_cross_ratio}) with this notation, we see that $\calB_\S$ is the vanishing locus of the homogeneous polynomial
\begin{subeqn}\label{eq:eq_B}
M(\alpha, \beta, \gamma):= 
L_{0,2}L_{1,3}R_{0,1}^q R_{2,3}^q - L_{0,1}L_{2,3}R_{0,2}^q R_{1,3}^q \quad \in \ol{\F_q}[\alpha, \beta, \gamma]\,.
\end{subeqn}
Notice that for each pair $(i,j) \in \{(0,1), (0,2), (1,3), (2,3)\}$ the varieties $\{L_{i,j}=0\}$ and $\{R_{i,j}=0\}$ are lines inside $\P^2$ and that it is easy to determine the intersections $\calB_\S\cap \{L_{i,j}{=}0\}$ and $\calB_\S\cap \{R_{i,j}{=}0\}$: such divisors are linear combinations of the points $X_k$'s defined in Figure \ref{figure} as intersections between lines in $\P^2$. The following proposition says that the points $X_k$ are well-defined and distinct.
\begin{figure}\caption{The intersections $X_i$ of the curve $\calB_\S$ with certain lines $L_{i,j}, R_{i,j}$.} \label{figure}
\includegraphics[width=\linewidth]{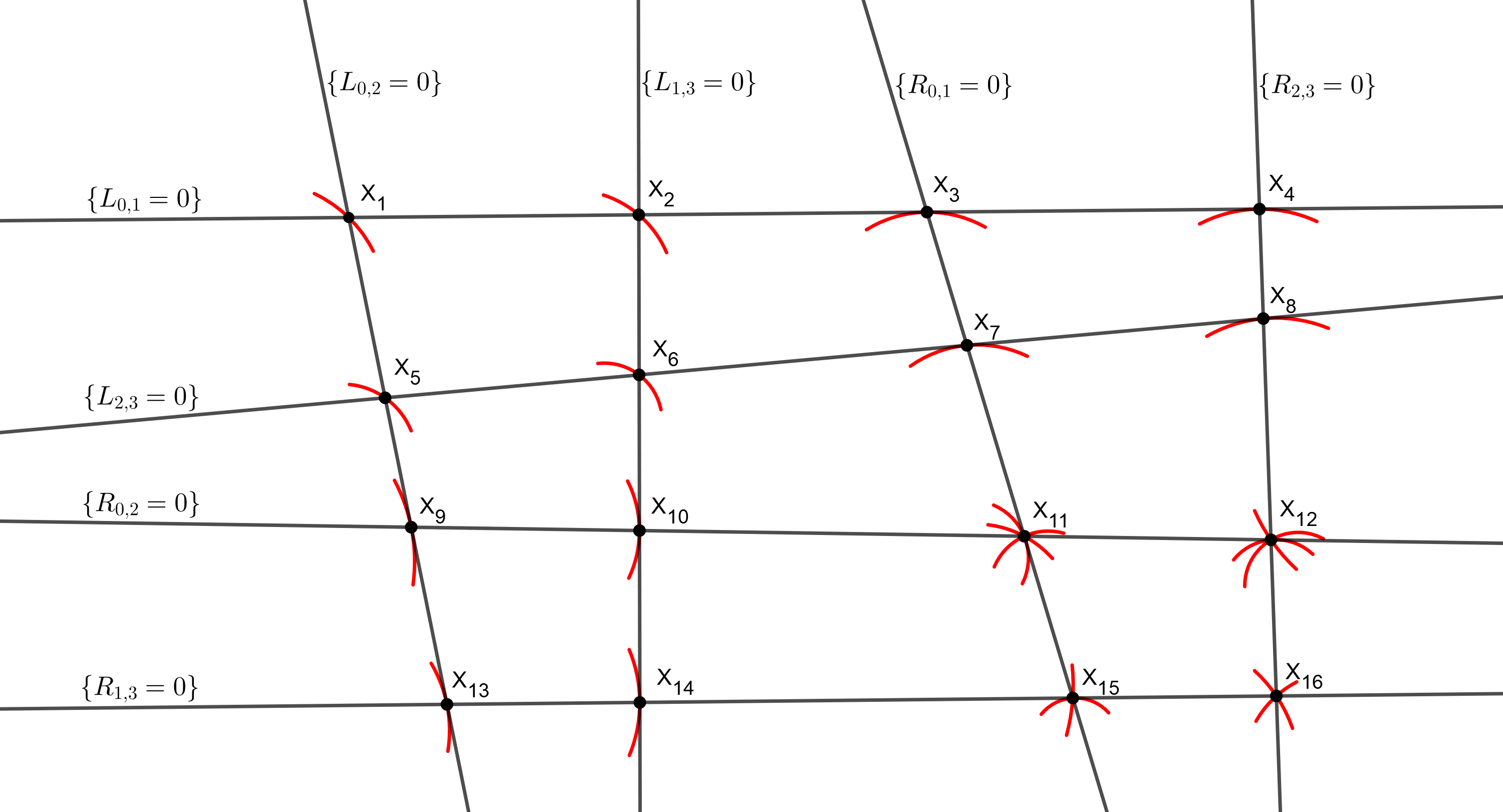}
\end{figure}
\begin{subClaim}\label{hope_distinct_points}
Consider the lines $\{L_{i,j}=0\}$ and $\{R_{i,j}=0\}$ and the points $X_i$ in Figure \ref{figure}. 
For all but at most $450$ choices of $\S \in E(k)$, those lines 
are distinct and the points $X_i$ are distinct.
\end{subClaim}
\begin{proof}
First we see that the points $\sigma^0\p, \sigma^0\Pp,\ldots$,  $\sigma^3\p, \sigma^4\Pp$ are pairwise distinct: clearly $\sigma^0\p, \ldots$, $ \sigma^3\p$ are distinct and  $\sigma^0\Pp, \ldots,  \sigma^3\Pp$ are distinct and if we had $\sigma^i \p =  \sigma^j\Pp$, then, for $l:=[k:\F_q]$ and $m:=i{-}j$, we would have
\[\begin{aligned}
\p + P_0 = \phi(\Pp) = &\phi(\sigma^{{i-j}} \p) = \phi(\phi^{l(i-j)}\p) = \phi^{lm +1}(\p) \\
& \implies \quad \phi^4(\p) =\phi^{4(lm+1)}(\p) = \p + 4P_0\,,
\end{aligned}
\]
which is absurd because $\Pp$ is not a trap and consequently $\phi^4(\p) \neq \p +4P_0$.

If the lines $\{L_{0,2}=0\}$ and $\{L_{1,3}=0\}$ were equal, then the matrix of their coefficients 
\[ 
n(\S) = \left(\begin{matrix}
l_2 {-} l_0  & (f_{\p}{-}f_{\sigma^2 \p})(\S) &  (l_2f_{\p}{-}l_0f_{\sigma^2 \p})(\S) \\
l_3 {-} l_1  & (f_{\sigma \p}{-}f_{\sigma^3 \p})(\S) &  (l_3f_{\sigma \p}{-}l_1f_{\sigma^3 \p})(\S)
\end{matrix} \right)
\]
would have rank $1$, which, computing the deteminant of a submatrix of $n$, would imply $P$ to be a zero of the rational  function $$(l_0-l_2) (f_{\sigma^3 \p}{-}f_{\sigma \p}) - (l_1{-}l_3)(f_{\sigma^2 \p}{-}f_{\p}),$$ that has five poles counted with multiplicity, since the points $\p, \ldots, \sigma^3\p$ are distinct and since $l_i\neq l_j$ by the choice of $\widetilde{\S}$.
Hence for all but at most $5$ choices of $\S \in E(k)$, the matrix $n(\S)$ has rank $2$ and consequently the lines $\{L_{0,2}=0\}$ and $\{L_{1,3}=0\}$ are distinct.

Similar arguments prove that, for any other pair of lines $\Lambda,\Lambda'$ in Figure \ref{figure}, $\Lambda \neq \Lambda'$  for all but at most five choices of $\S$. \begin{change} This implies that, up to excluding at most $140$ choices of $P$, we can suppose that the eight lines in Figure \ref{figure} are pairwise distinct. 
Under this assumption, we now show that, for all $i \neq j$, we have $X_i \neq X_j$, for all but $6$ choices of $\S \in E(k)$. 
The number of total cases ``$i \neq j$'' to be treated is at most $48 = {4 \choose 2}\cdot 4 \cdot 2$ since if two points collide, then three lines in Figure \ref{figure} (two of which vertical/horizontal and one horizontal/vertical) are concurrent. This implies that up to excluding $140 +6\cdot 48 < 450$ points, then all points $X_i$ are distinct. We only treat a couple of the $48$ cases $X_i \neq X_j$, but the same ideas can be applied in the remaining cases.
\end{change}

If $X_{9} = X_{12}$, then the lines $\{R_{0,2}=0\}$, $\{R_{2,3}=0\}$ and $\{L_{0,2}=0\}$ are concurrent, hence the following matrix, that contains their coefficients, is not invertible 
\[
M = M(\S) = \left(\begin{matrix}
r_2 {-} r_0  & (f_{\Pp}{-}f_{\sigma^2 \Pp})(\S) &  (r_2f_{\Pp}{-}r_0f_{\sigma^2 \Pp})(\S) \\
r_3 {-} r_2  & (f_{\sigma^2\Pp}{-}f_{\sigma^3 \Pp})(\S) &  (r_3f_{\sigma^2\Pp}{-}r_2f_{\sigma^3 \Pp})(\S) \\
l_2 {-} l_0  & (f_{\p}{-}f_{\sigma^2 \p})(\S) &  (l_2f_{\p}{-}l_0f_{\sigma^2 \p})(\S) \\
\end{matrix}\right)\,,
\]  
implying that $\S$ is a zero of the rational function $\det(M)$. Writing out $\det(M)$, and using that $\sigma^0\p, \sigma^0\Pp,\ldots$,  $\sigma^3\p, \sigma^4\Pp$ are pairwise distinct, we see that there is a rational function $g$, regular in ${-}\sigma^2 \Pp$, such that 
\[
\det(M) = (l_2 {-} l_0)(r_0 {-} r_3) f_{\sigma^2\Pp}^2 + f_{\sigma^2 \Pp}\,g\,.
\]
Since $l_0 \neq l_2$ and $r_0 \neq r_3$, the rational function $\det(M)$ has a pole of order $2$ in ${-}\sigma^2\Pp$ and in particular $\det(M)$ is a non-zero rational function with at most $6$ poles counted with multiplicity. Hence $\det(M)$ has at most $6$ zeroes, implying that $X_9 \neq X_{12}$, for all but $6$ choices of $\S \in E(k)$.

If $X_3 = X_7$, then the lines $\{L_{0,1}=0\}$, $\{L_{2,3}=0\}$ and $\{R_{0,1}=0\}$ are concurrent, hence the following matrix, that contains the coefficients of $L_{0,1}$, $L_{2,3}$ and $R_{0,1}$, is not invertible
\[
N = N(\S) = \left(\begin{matrix}
l_1 {-} l_0  & (f_{\p}{-}f_{\sigma \p})(\S) &  (l_1f_{\p}{-}l_0f_{\sigma \p})(\S) \\
l_3 {-} l_2  & (f_{\sigma^2\p}{-}f_{\sigma^3 \p})(\S) &  (l_3f_{\sigma^2\p}{-}l_2f_{\sigma^3 \p})(\S) \\
r_1 {-} r_0  & (f_{\Pp}{-}f_{\sigma \Pp})(\S) &  (r_1f_{\Pp}{-}r_0f_{\sigma \Pp})(\S) \\
\end{matrix}\right)\,.
\]  
As before, in order to prove that $X_3 \neq X_4$ for all but at most $6$ choices of $\S \in E(k)\setminus\{O_E\}$ it is enough proving that $\det(N(\S))$, considered as a rational function of $\S$, is not identically zero. We suppose by contradiction that $\det(N)$ is identically zero. 
Denoting  $N_{i,j}$ the $(i,j)$-minor of $N(\S)$, it easy to see the rational functions $N_{i,3}$ are not null. Consequently, the only way for $\det(N)$ to be zero is that the third column $N$ is linearly dependent from the first two.
Hence, there are rational functions $A,B \in \ol{\F_q}(E)$ such that
\begin{subeqn}\label{eq:linear_system_Hope}
	\begin{cases}
		\big(l_1 {-} l_0\big)  \cdot A +  \big(f_{\p}{-}f_{\sigma \p}\big) \cdot 	B = l_1f_{\p}{-}l_0f_{\sigma \p} \\
		\big(l_3 {-} l_2\big)  \cdot A +  \big(f_{\sigma^2\p}{-}f_{\sigma^3 \p} \big) \cdot 	B = l_3f_{\sigma^2\p}{-}l_2f_{\sigma^3 \p} \\
		\big(r_1 {-} r_0\big)  \cdot A +  \big(f_{\Pp}{-}f_{\sigma \Pp}\big) \cdot 	B = r_1f_{\Pp}{-}r_0f_{\sigma \Pp} \\
	\end{cases}
\end{subeqn}
and, using Cramer's rule, we have
\[
B = \frac{N_{1,2}}{N_{1,3}} = \frac{N_{2,2}}{N_{2,3}}= \frac{N_{3,2}}{N_{3,3}}\,.
\]
We easily compute the poles of the rational functions $N_{i,j}$ and check that they all vanish in $\widetilde{\S}$ and $O_E$, (for $O_E$ it is enough noticing that $(f_P {-} \tfrac yx)(O_E)=0$ for all $P \in E(\ol{\F_q}) \setminus\{O_E\}$). These computations give 
\[
\begin{aligned}
\divi(N_{1,j}) &= D_{1,j} + \widetilde{\S} + O_E - ({-}\Pp) - ({-}\sigma\Pp) - ({-}\sigma^2\p) - ({-}\sigma^3\p) \,,\\
\divi(N_{2,k})&= D_{2,j} + \widetilde{\S} + O_E - ({-}\p) - ({-}\sigma\p) - ({-}\Pp) - ({-}\sigma\Pp) \,,\\
\divi(N_{3,j})&= D_{3,j} + \widetilde{\S} + O_E - ({-}\p) - ({-}\sigma\p) - ({-}\sigma^2\p) - ({-}\sigma^3\p) \,,\\ 
\end{aligned}
\]
for certain positive divisors $D_{i,j}$ of degree $2$. Consequently
\[  
\divi(B) = D_{1,2} - D_{1,3} = D_{2,2} - D_{2,3} =D_{3,2} - D_{3,3}\,.
\]
Since the functions $f_{\p}$, $f_{\sigma \p}$, $f_{\sigma^2 \p}$ and $f_{\sigma^3 \p}$ are $\ol{\F_q}$-linearly independent, then $N_{1,2}$ and $N_{1,3}$ are not $\ol{\F_q}$-multiples and $B$ is not constant. Since every non-constant rational function on $E$ has at least two poles, we deduce that $D_{1,3} = D_{2,3} = D_{3,3}$ is the divisor of poles of $B$. In particular, in the group $E(\ol{\F_q})$, the sum of the poles of $N_{1,3}$ is equal to the sum of the poles of $N_{2,3}$ and to the sum of the poles of $N_{3,3}$. Writing it out, we get the following equalities in the group $E(\ol{\F_q})$
\[
\p + \sigma \p = \sigma^2\p + \sigma^3 \p = \Pp + \sigma \Pp\,.
\]
Hence, using (\ref{eq:zeroes_superuseful}),  ${-}\p{-}\sigma\p$ is a zero of $N_{3,3}$ and consequently the two poles of $B$ are ${-}\p{-}\sigma\p$ and ${-}\p{-}\sigma\p{-}\widetilde{\S}$. By looking at (\ref{eq:linear_system_Hope}), we deduce that $A$ has exactly one simple pole, namely ${-}\p{-}\sigma\p{-}\widetilde{\S}$, which is absurd. Hence $\det(N(\S))$ is not identically zero.
\end{proof}

We now study the geometrically irreducible components of $\calB_\S$ for the points $P$ such that the conclusions of Claim \ref{hope_distinct_points} hold (there are such $P$'s because $450$ is way smaller than $\#E(k)$). 
Equation (\ref{eq:eq_B}), that defines $\calB_\S$, gives the following divisor-theoretic intersection 
\begin{subeqn}\label{eq:inters_B_R_L_31}
\calB_\S \cap \{L_{0,2} = 0 \} = X_1 + X_5 + qX_9 + qX_{13}\,.
\end{subeqn}
Since $X_i$'s are distinct, the point $X_1$ has multiplicity $1$ in the above intersection and consequently it is a smooth point of $\calB_\S$.  With analogous arguments we see that all the points $X_i$, except the ones of shape $\{ R_{i,j} =0\} \cap \{ R_{l,m}=0 \}$, are smooth. This is used in the following Claim.

\begin{subClaim}\label{hope_no_conics}
Assuming the conclusions of Claim \ref{hope_distinct_points}, the curve $\calB_\S$ does not contain any conic defined over~$k$.
\end{subClaim}
\begin{proof}
Suppose there exists such a conic $\mathcal Q$. Since $X_9$ is a smooth point of $\calB_\S$, if $X_9\in \mathcal Q$, then $\mathcal Q$ is the only component of $\calB_\S$ passing through $X_9$, hence $X_9$ appears in $\calB_\S \cap \{ L_{0,2}=0 \}$ with multiplicity at most $2<q$, contradicting Equation (\ref{eq:inters_B_R_L_31}). Hence $\mathcal Q$ does not contain $X_9$ nor, by a similar argument, $X_{13}$.
This, together with Equation (\ref{eq:inters_B_R_L_31}), implies that $X_1$ and $X_5$ belong to $\mathcal Q$. Analogously $X_2$ and $X_6$ belong to $\mathcal Q$.

We notice that $X_1,X_2,X_5,X_6$ are in general position (by Claim \ref{hope_distinct_points}) and that we know two conics passing through them, namely $\{L_{0,1}L_{2,3} =0\}$ and $\{L_{0,2}L_{1,3} = 0\}$. Hence, if we choose a homogenous quadratic polynomial  $F \in k[\alpha, \beta, \gamma]$ defining $\mathcal Q$, there are $\lambda_0, \lambda_1 \in \ol{\F_q}$ such that
\[
F = \lambda_0 L_{0,1}L_{2,3} + \lambda_1 L_{0,2}L_{1,3}\,.
\]
Our claim now follows by extending $\sigma$ to an element in $\Gal(\ol{\F_q}/k)$ and looking at its action on the above equation.
For each $i,j \in \{0,1,2,3\}$ we have $\sigma L_{i,j} = L_{i+1,j+1} = {-}L_{j+1,i+1}$, considering the indices modulo $4$, hence
\begin{subeqn}\label{eq:bohboh}
\lambda_0 L_{0,1}L_{2,3} + \lambda_1 L_{0,2}L_{1,3} = F = \sigma F = 
\sigma(\lambda_0) L_{2,3}L_{3,0} + \sigma(\lambda_1)L_{0,2}L_{1,3}\,.
\end{subeqn}
\begin{change}
We notice that three lines of the form $\{L_{a,b}=0\}$, $\{L_{a,c}=0\}$ $\{L_{a,d}=0\}$ always have common point since the $3\times 3$ matrix containing their coefficients has determinant zero. When $a=1$ we deduce that  $\{L_{1,2}=0\}$ passes through the intersection $X_2 = \{L_{1,3}=0 \}\cap \{L_{0,1}=0\}$, and when $a=2$ we deduce that also $X_5$ belongs to $\{L_{1,2}=0\}$. In other words $\{L_{1,2}=0\}$ is the line through $X_2,X_5$. Analogously  $\{L_{3,0}=0\}$ is the line through $X_1,X_6$.
\end{change}

In particular, the polynomials $L_{i,j}$ in \eqref{eq:bohboh} are pairwise coprime, which implies $\lambda_0 = \sigma(\lambda_0) = 0$.  Hence, $\mathcal Q = \{F=0\} =  \{L_{0,2}L_{1,3} = 0 \}$, which is absurd because $\{L_{0,2}L_{1,3} = 0 \}$ is not contained in $\calB_\S$.
\end{proof}

Claim \ref{hope_no_conics} also implies that $\calB_\S$ does not contain a line of $\P^2$. Suppose that $\Lambda$ is a line contained in $\calB_\S$.  Neither $X_9$ nor $X_{13}$ are contained in $\Lambda$ since they are smooth points of $\calB_\S$ and, by Equation \eqref{eq:inters_B_R_L_31}, the unique components of $\calB_\S$ passing through them must have degree at least $q$ inside $\P^2$. Hence $\Lambda \cap \{L_{0,2}=0 \} \in \{X_1, X_5\}$ and consequently  
\begin{subeqn}\label{eq:lines_conj_inters}
(\Lambda \cup \sigma^2 \Lambda) \cap \{L_{0,2}=0 \}  = X_1 + X_5\,.
\end{subeqn}
This implies that $\sigma^2\Lambda \neq \Lambda$ and that $\sigma^2\Lambda$ and $\Lambda$ are all the $\Gal(\ol{\F_q}/k)$-conjugates of $\Lambda$: if $\Lambda$ had another conjugate $\Lambda'$, then,  since $\calB_\S$ is defined over~$k$, also $\Lambda'$ would be a component of $\calB_\S$
and, by the same argument as before, $\Lambda' \cap \{L_{0,2}=0 \} = X'  \in \{X_1, X_5\}$, which, together with Equation (\ref{eq:lines_conj_inters}), implies that two or more components of $\calB_\S$ pass through $X'$, contradicting the smoothness of $X_1$ and $X_5$. We deduce that $\Lambda {\cup} \sigma^2\Lambda$ is a conic defined over~$k$ and contained in $\calB_\S$, contradicting Claim \ref{hope_no_conics}.

By a similar argument, no conic $\mathcal Q$ is a component of $\calB_\S$: if this happens, since conics have degree $2<q$ in $\P^2$, then $X_9,X_{13}$ do not belong to any of the $\Gal(\ol{\F_q}/k)$-conjugates of $\mathcal Q$, thus, by Equation (\ref{eq:inters_B_R_L_31}), for all $\tau \in \Gal(\ol{\F_q}/k)$ we have
\[
\tau(\mathcal Q) \cap \{L_{0,2}=0 \} = X_1 + X_5 = \mathcal Q \cap \{L_{0,2}=0 \} 
\]
hence, by the smoothness of $X_1$ and $X_5$, $\mathcal Q$ is defined over~$k$, contradicting Claim \ref{hope_no_conics}.

We now suppose that $\calB_\S$ is not geometrically irreducible. Let $\calB_1, \ldots, \calB_r$ be the geometrically irreducible components of $\calB_\S$. As we already proved, each $\calB_i$ has degree at least $3$, hence the intersection $\calB_i \cap \{L_{0,2}=0\}$ is a sum of at least $3$ points counted with multiplicity. By Equation (\ref{eq:inters_B_R_L_31}), this implies that $\calB_i$ is  passing through $X_9$ or $X_{13}$ hence each $\calB_i$ has degree at least $q$. Since the sum of the degrees of the $\calB_i$'s is equal to $2q{+}2 < 3q$, we deduce that $r=2$ and that either $\deg(\calB_1)= \deg(\calB_2) = q+1$ or, up to reordering, $\deg(\calB_1)= q$ and $\deg(\calB_2) = q+2$. 

If  $\deg(\calB_1)= \deg(\calB_2) = q+1$, Equation (\ref{eq:inters_B_R_L_31}) implies that, up to reordering, $X_1 \in \calB_1(\ol{\F_q})$ and $X_5 \in \calB_2(\ol{\F_q})$. Since $\calB_\S$ is defined over~$k$, then $\Gal(\ol{\F_q}/k)$ acts on $\{\calB_1,\calB_2\}$ and because of the cardinality of such a set, then $\sigma^2$ acts trivially. In particular $X_5 = \sigma^2X_1$ belongs to $\sigma^2 \calB_1(\ol{\F_q}) = \calB_1(\ol{\F_q})$, hence $X_5 \in \calB_1(\ol{\F_q}) \cap \calB_2(\ol{\F_q})$, contradicting the smoothness of $X_5$. This contradiction implies that 
\[
\deg(\calB_1)= q, \quad \deg(\calB_2) = q+2 \,.
\]
For each $i\in\{1,2\}$ let $M_i \in \ol{\F_q}[\alpha, \beta, \gamma]$ be a homogeneous polynomial defining $\calB_i$.
\begin{subClaim}\label{claim:shape_Mi}
There exist homogenous polynomials $F_1, F_2, G_2, N_1, N_2$ in $\ol{\F_q}[\alpha, \beta, \gamma]$ of respective degree $1,1,1,q-4, q-2$ such that
\begin{align}
\label{eq:M1F1} M_1 &= F_1^q + L_{0,1}L_{2,3}L_{0,2}L_{1,3}N_1, \\
\label{eq:M2F2}
M_2 &= F_2^qL_{0,1}L_{2,3} + G_2^qL_{0,2}L_{1,3} + L_{0,1}L_{2,3}L_{0,2}L_{1,3}N_2
\end{align}
\end{subClaim}
\begin{proof}
We start from the first part. Since $\deg \calB_1=q$ and since $X_1, X_5, X_9$ and $X_{13}$ are smooth, Equation (\ref{eq:inters_B_R_L_31}) implies that $\calB_1 \cap \{L_{0,2}=0\}$ is either $qX_{13}$ or $qX_9$. Hence $M_1 \bmod{L_{0,2}}$ is a $q$-th power, i.e. there are homogeneous polynomials $A_1, B_1$, with $A_1$ linear, such that
\[
M_1 = A_1^q + B_1 L_{0,2}\,.
\]
Similarly to $\calB_1 \cap \{L_{0,2}=0\}$, we have that $\calB_1 \cap \{L_{1,3}=0\}$ is either $qX_{14}$ or $qX_{10}$, hence there exists a linear polynomial $A_2$ such that
\begin{align} \label{eq:a_congruence1}
A_2^q \equiv M_1 \equiv A_1^q + B_1 L_{0,2}  \implies  
B_1 L_{0,2} \equiv (A_2 -  A_1)^q (\bmod{L_{1,3}})
\end{align}
We notice, that, since $L_{1,3} = l_\alpha \alpha + l_\beta\beta+ l_\gamma \gamma$ is linear, with $l_\alpha = l_3-l_1\neq 0$, the map  $F \mapsto F\left( -\frac{ l_\beta\beta+ l_\gamma \gamma}{l_\alpha} ,\beta, \gamma \right)$ gives an isomorphism of $\ol{\F_q}[\alpha,\beta, \gamma]/L_{1,3}$ with $\ol{\F_q}[\beta, \gamma]$, which is a UFD. Moreover this isomorphism sends $(A_2 {-}  A_1) \bmod{L_{1,3}}$ to a homogeneous polynomial of degree at most $1$, which means that  $(A_2 {-}  A_1) \bmod{L_{1,3}}$ is either null or irreducible.
We deduce that, in the last congruence of (\ref{eq:a_congruence1}), either both sides are zero or the right hand side gives the prime factorization of the left hand side. 
In both cases we have $B_1\equiv \lambda_1 L_{0,2}^{q-1}$ for some  $\lambda_1 \in \ol{\F_q}$, hence
\[
B_1 = \lambda_1 L_{0,2}^{q-1} + B_2 L_{1,3} \implies M_1 = (A_1 +\lambda_1 L_{0,2}) ^q + B_2 L_{0,2}L_{0,3} = A_3^q +  B_2 L_{0,2}L_{0,3}
\]
for certain homogeneous polynomials $A_3, B_2$, with $A_3$ linear.
Again, the cycle $\calB_1 \cap \{L_{0,1}=0\}$ is either $qX_{3}$ or $qX_{4}$, hence, for a some linear homogenous polynomial $A_4$, we have
\[
\begin{aligned}
&A_4^q \equiv M_1 \equiv  A_3^q + B_2 L_{0,2} L_{1,3}
\quad \implies \quad 
B_2 \, L_{0,2}\, L_{1,3} \equiv (A_4 -  A_3)^q (\bmod{L_{0,1}})\,.
\end{aligned}
\]
As before, the last congruence can be interpreted as an equation in the UFD  $\ol{\F_q}[\alpha,\beta, \gamma]/L_{1,3}$ and, since $(A_4 {-}  A_3) \bmod{L_{0,1}}$ is either irreducible or null, then either both sides are zero or the right hand side gives the prime factorization of the left hand side. The latter is not possible, since the points  $X_1 = \{L_{0,1}=0\}\cap\{L_{0,2}=0\}$ and $X_2 = \{L_{0,1}=0\}\cap\{L_{1,3}=0\}$ are distinct and consequently $L_{0,2}\bmod{L_{0,1}}$ and $L_{1,3}\bmod{L_{0,1}}$ are relatively prime. We deduce that $B_2$ is divisible by $L_{0,1}$. By a similar argument $B_2$ is also divisible by $L_{2,3}$, hence Equation (\ref{eq:M1F1}).

\changee{
We now turn to (\ref{eq:M2F2}). The existence of such an equality is equivalent to the existence of linear homogeneous polynomials $F_2,G_2$ such that 
\[
M_2 \equiv  F_2^qL_{0,1}L_{2,3} + G_2^qL_{0,2}L_{1,3}  \pmod{L_{0,1}L_{2,3}L_{0,2}L_{1,3}}.
\]
Since we are working under the conclusion of Claim \ref{hope_distinct_points}, the polynomial $L_{i,j}$ above are coprime, hence the natural map 
$$ 
{\ol{\F_q}[\alpha,\beta,\gamma]}/{(L_{0,1}L_{2,3}L_{0,2}L_{1,3})} \lto \left({\ol{\F_q}[\alpha,\beta,\gamma]}/{L_{0,1}L_{2,3}}\right) \times \left({\ol{\F_q}[\alpha,\beta,\gamma]}/{L_{1,3}L_{0,2}}\right)
$$
is injective and consequently it is enough proving the existence of polynomials $F_2,G_2$ satisfying
\begin{align}
	M_2 \equiv  F_2^qL_{0,1}L_{2,3}  \pmod{L_{0,2}L_{1,3}} \label{eq:techn_cong_Lij}
	\\
	M_2 \equiv  G_2^qL_{0,2}L_{1,3}  \pmod{L_{0,1}L_{2,3}}.
	\label{eq:techn_cong_Lij2}
\end{align}
We now prove \eqref{eq:techn_cong_Lij}.
}
Since $\deg \calB_2=q+2$ and since $X_1, X_5, X_9$ and $X_{13}$ are smooth, Equation (\ref{eq:inters_B_R_L_31}) implies that $\calB_2 \cap \{L_{0,2}\}$ is either $X_1 {+} X_5 {+} qX_{13}$ or $X_1 {+} X_5 {+}qX_9$, hence we can write $ M_2 = A_5^qL_{0,1}L_{2,3} + B_3 L_{0,2}$
for some homogeneous polynomials $A_5, B_3$, with $A_5$ linear. In a similar fashion, 
$\calB_2 \cap \{L_{1,3}\}$ is either $X_2 {+} X_6 {+} qX_{14}$ or $X_2 {+} X_6 {+}qX_{10}$, hence,
\[\begin{aligned}
	&L_{0,1} L_{2,3}  A_6^q \equiv M_1\equiv L_{0,1} L_{2,3}  A_5^q + B_3 L_{0,2} \\
	& \qquad  \implies 
	B_3 L_{0,2} \equiv L_{0,1}L_{2,3} (A_6 - A_5)^q \,(\bmod{L_{1,3}} )\end{aligned}
\]
As before, in the last equation either both sides are zero or the right hand side gives the prime factorization of the left hand side. In both cases $B_3$ is congruent to a scalar multiple of $L_{0,1}L_{2,3} L_{0,2}^{q-1}$: if  $B_3\equiv 0$  this is obvious, otherwise we need to use that the polynomials $L_{0,1}$, $L_{2,3}$ and $L_{0,2}$ are relatively prime modulo $L_{1,3}$ because the lines $\{L_{0,1}=0\}$, $\{L_{2,3}=0\}$ and $\{L_{0,2}=0\}$ all have different intersection with $\{L_{1,3}=0\}$. 
Hence
\[
M_2 = A_7^qL_{0,2}L_{1,3} + B_4L_{0,1}L_{2,3}.
\]
for certain $A_7, B_4 \in \ol{\F_q}[\alpha, \beta, \gamma]$. \begin{change} This implies congruence  \eqref{eq:techn_cong_Lij} which, together \eqref{eq:techn_cong_Lij2} that can be proven analogously, implies Equation \eqref{eq:M2F2}. \end{change}
\end{proof}

Let $F_1, F_2, G_1, N_1$ and $N_2$ as in Claim \ref{claim:shape_Mi}.
Up to multiplying $M_1$ with an element of $\ol{\F_q}^\times$, we can suppose that $M = M_1M_2$. Reducing this equality modulo $L_{0,2}L_{1,3}$ we see that
\[
L_{0,2}L_{1,3} \text{ divides } L_{0,1}L_{2,3}(F_1F_2 + R_{0,2}R_{1,3})^q.
\] 
Since the $L_{i,j}$'s in the above equation are coprime, then $L_{0,2}L_{1,3}$ divides $F_1F_2 {+} R_{0,2}R_{1,3}$. Since $F_1F_2 {+} R_{0,2}R_{1,3}$ is homogenous of degree at most $2$, then it is a scalar multiple of $L_{0,2}L_{1,3}$. Using a similar argument with  $L_{0,1}L_{2,3}$ we prove that there exist $\lambda, \mu \in \ol{\F_q}$ such that
\begin{equation}\label{eq:patch1}
F_1F_2 + R_{0,2}R_{1,3} = \lambda L_{0,2}L_{1,3}, \quad F_1G_2 - R_{0,1}R_{2,3} = \mu L_{0,1}L_{2,3} \,.
\end{equation}
We have $\lambda \neq 0$, otherwise $F_1$ would be a scalar multiple of either $R_{0,2}$ or $R_{1,3}$: in the first case Equation \ref{eq:M1F1} would imply that $\calB_1$ contains $X_9$ but not $X_{14}=\tau(X_9)$, implying that $\tau(\calB_1)$ is a component of $\calB$ different from $\calB_1$, that is $\tau(\calB_1)=\calB_2$ which contradicts the inequality $\deg(\calB_2)>\deg(\calB_1)$;  in the second case Equation \ref{eq:M1F1} would imply that $\calB_1$ contains $X_{13}$ but not $X_{10}=\tau(X_{13})$, leading to the same contradiction.

Using Equations (\ref{eq:M1F1}), (\ref{eq:M2F2}) and (\ref{eq:patch1}) and the equality $M_1M_2{=}M$, we see that
\begin{equation} 
\begin{aligned} \label{boh43}
0 & = \frac{M_1M_2 - M}{L_{0,1}L_{2,3}L_{0,2}L_{1,3}} = \\
&= \mu^q L_{0,1}^{q-1}L_{2,3}^{q-1} +  \lambda^q L_{0,2}^{q-1}L_{1,3}^{q-1} + F_1^q N_2 + F_2^qN_1L_{0,1}L_{2,3} +  \ldots \\
& \qquad \ldots + G_2^qN_1L_{0,2}L_{1,3} + N_1N_2 L_{0,1}L_{2,3}L_{0,2}L_{1,3}\\
& \equiv \lambda^q (L_{0,2}L_{1,3})^{q-1} + F_1^qN_2 + G_2^qN_1L_{0,2}L_{1,3} \pmod{L_{0,1}}.
\end{aligned}
\end{equation}
As already observed in the proof of Claim \ref{claim:shape_Mi}, $\ol{\F_q}[\alpha,\beta, \gamma]/L_{0,1}$ is isomorphic to $\ol{\F_q}[\beta, \gamma]$ through the map $F \mapsto F\left( -\frac{ l_\beta\beta+ l_\gamma \gamma}{l_\alpha} ,\beta, \gamma \right)$, where $l_\alpha, l_\beta$ and $l_\gamma$ are the coefficients of $L_{0,1}$. In particular, $\ol{\F_q}[\alpha,\beta, \gamma]/L_{0,1}$ is a UFD and for any $F \in \ol{\F_q}[\alpha, \beta, \gamma]/L_{0,1}$ we denote $\tilde F$ its image in $\ol{\F_q}[\beta, \gamma]$ through the above map. With this notation  (\ref{boh43}) implies that  both $\tilde L_{0,2}$ and $\tilde L_{1,3}$ divide $\tilde N_2\tilde{F_1^q}$.
More precisely $\tilde L_{0,2}$ and $\tilde L_{1,3}$ divide $\tilde N_2$, because $\tilde F_1$ is relatively prime with both $\tilde L_{0,2}$ and $\tilde L_{1,3}$, since $\calB_1 \cap L_{0,1}$ does not contain $X_1 = \{L_{0,2} {=}0\}\cap \{L_{0,1} {=}0\}$ nor $X_3 = \{L_{1,3} {=}0\} \cap \{L_{0,1} {=}0\}$. 
Moreover, since  the points $X_1 = \{L_{0,2} {=}0\}\cap \{L_{0,1} {=}0\}$ and $X_3 = \{L_{1,3} {=}0\} \cap \{L_{0,1} {=}0\}$ are distinct,  $\tilde L_{0,2}$ is relatively prime with $\tilde L_{1,2}$ and we can write $\tilde N_2 = \tilde L_{0,2}\tilde L_{1,3}N_3$ for some $N_3 \in \ol{\F_q}[\beta, \gamma]$.
Substituting in (\ref{boh43}) we get
\[
\lambda^q \tilde L_{0,2}^{q-2} \tilde L_{1,3}^{q-2} + \tilde F_1^q N_3 + \tilde G_2^q \tilde N_1 =0.
\]
Since $\tilde L_{0,2}$ and $\tilde L_{1,2}$ are coprime and since $\lambda\neq 0$, this contradicts Lemma \ref{lem:hopehopehope} below. 

In particular the assumption of the reducibility of $\calB$, together with the conclusions of Claim \ref{hope_distinct_points}, led to contradiction. We deduce that for all but at most $450$ choices of $\S \in E(k)$ the curve $\calB_\S$ is geometrically irreducible. Since $\# E(k) > 450$  and since all the components of $\calB$ project surjectively to $E$, we deduce that $\calB$ is reduced and geometrically irreducible.

\begin{subLemma}\label{lem:hopehopehope}
Let $L_1, L_2 \in \ol{\F_q}[\beta, \gamma]$ be relatively prime homogenous linear polynomials. Then, there exist no homogenous polynomials $A,B,C,D$ in $\ol{\F_q}[\beta, \gamma]$ such that 
\[ 
L_1^{q-2}L_2^{q-2} = A^qB + C^qD.
\]
\end{subLemma}
\begin{proof}
The zeroes of $L_1$ and $L_2$ in $\P^1$ are distinct, hence, up to a linear transformation we can suppose that their zeroes are $0$ and $\infty$. In particular, up to scalar multiples we can suppose $L_1 =\beta$ and $L_2 = \gamma$, implying that $A^qB + C^qD= \beta^{q-2}\gamma^{q-2}$. This is absurd because any monomial appearing in $A^q$ or in $B^q$ is either a multiple of $\beta^q$ of a multiple of $\gamma^q$, hence the same is true for all the monomials appearing in  $ A^qB + C^qD$.
\end{proof}

\subsection{The irreducible components of $\calC$}
In this subsection we prove that all the geometrically irreducible components of $\calC$ are defined over~$k$. To do so, we can ignore (\ref{eq:43_item4_main}) in the definition of $\calC$. The strategy is applying Lemma \ref{lem:good_curves} to the variety $\calB$, using the rational functions 
\[
\begin{aligned}
&u_1,u_2,u_3 \colon \calB \dashrightarrow \P^1\,, \qquad u_i(\alpha, \beta, 1,\S) = f_{\alpha, \beta,\S} (\sigma^{i-1}\p)\,, \\ 
&w_1,w_2,w_3 \colon \calB \dashrightarrow \P^1\,, \qquad w_i(\alpha, \beta, 1, \S) =  -f_{\alpha, \beta, \S} (\sigma^{i-1} \Pp)\,,
\end{aligned}
\]
and the irreducible divisor $\Zp \subset \calB$ being the Zariski closure of 
\begin{subeqn}\label{eq:def_Z_43}
\left\{(\alpha, \beta, P) \in (\A^2{\times}E')(\ol{\F_q}): \begin{matrix} 
	P = - \p - \sigma\p - \sigma^3\p - \widetilde{P}\,, \\
	\alpha = \frac{(f_{\p}(P){-}f_{\sigma\p}(P))\beta + l_1f_{\p}(P) {-} l_0f_{\sigma\p}(P)}{l_0 {-} l_1} \end{matrix} \right\}.
\end{subeqn}
\begin{subClaim}\label{claim_43_use_lemma}
The variety $\Zp$ is generically contained in the smooth locus of $\calB$ and the rational function $u_1{-}u_2$ vanishes on $\Zp$ with multiplicity $1$.
\end{subClaim}
\begin{proof}
We restrict to an open subset $U \subset \P^2\times E$ containing the generic point of $\Zp$. 
Up to shrinking $U$, the rational functions $u_i, w_i$ can be extended to regular functions on $U$ using the definition (\ref{eq:43_def_f}) of $f_{\alpha, \beta, \S}$, and we have
\[
\begin{aligned}
u_1 - u_2 &= \frac{L_{0,1}(\alpha, \beta,1,\S )}{\big( l_0 + \beta\big) \big( l_1 + \beta\big)}\,,
\end{aligned}
\]
where $L_{i,j}(\alpha, \beta, \gamma, \S) \in \ol{\F_q}[U]$ is defined as in (\ref{eq:def_Lij_Rij}), as well as $R_{i,j}(\alpha, \beta, \gamma, \S)$. 
Since we can assume that $l_0{+}\beta, l_1{+}\beta$ are invertible on $U$ and since $\Zp$ is generically smooth, it is enough showing that $\Zp\cap U$ is a component of $(\calB \cap U) \cap {\{L_{0,1} = 0\}}$ having multiplicity one. Up to shrinking $U$, the  $\calB \cap U$ is the vanishing locus, inside $U$, of 
\[
M(\alpha, \beta, \S) := (L_{0,2}L_{1,3}R_{0,1}^q R_{2,3}^q - L_{0,1}L_{2,3}R_{0,2}^q R_{1,3}^q)(\alpha, \beta, 1, \S)  \quad \in \ol{\F_q}[U]\,.
\]
Since restricting of $M$ to $\{L_{0,1}=0\}$ is the same as restricting $L_{0,2}L_{1,3}R_{0,1}^q R_{2,3}^q$, 
it is enough showing that $L_{0,2}$, $R_{0,1}$, $R_{2,3}$ do not vanish on $\Zp$ and that $\{L_{1,3} = 0\} \cap \{L_{0,1}=0\}$ contains $\Zp\cap U$ with multiplicity $1$. We start from the latter. Eliminating the variable $\alpha$ we see that, up to shrinking $U$,  $\{L_{1,3} = 0\} \cap \{L_{0,1}=0\}$ is defined by the equations 
\begin{subeqn}\label{eq:eqs_inters_for_lemma_43}
	\lambda (\S) =0  \quad \text{and} \quad  (l_1 - l_0)\alpha + (f_{\p}(\S) - f_{\sigma\p}(\S))\beta +  l_1 f_{\p}(\S) - l_0 f_{\sigma\p}(\S) =0 \,,
\end{subeqn}
where
\[
\lambda (\S) := (l_1 {-} l_0)f_{\sigma^3\p}(\S) + (l_3 {-} l_1) f_{\p}(\S) + (l_0 {-} l_3)f_{\sigma\p}(\S) \quad \in \ol{\F_q}(E)\,.
\]
The function $\lambda$ has three simple poles, namely ${-}\p, {-}\sigma \p, {-}\sigma^3\p$, and we easily verify that $\lambda(\widetilde{\S})=\lambda(O_E)=0$. 
We deduce that $\S={-}\p{-}\sigma\p{-}\sigma^3\p{-}\S_0$ is a simple zero of $\lambda$. This, together with the fact that the second equation in (\ref{eq:eqs_inters_for_lemma_43}) is equal to the second equation in the definition (\ref{eq:def_Z_43}) of $\Zp$, implies that $\{L_{1,3} = 0\} \cap \{L_{0,1}=0\}$ contains $\Zp \cap U$ with multiplicity~$1$.

We now suppose by contradiction that $R_{0,1}$ vanishes on $\Zp \cap U$. Substituting $\alpha$ and $\S$ in $R_{0,1}$ as in the definition (\ref{eq:def_Z_43}) of $\Zp$, we see that
\[
R_{0,1}(\alpha, \beta, 1, P)|_{\Zp\cap U} = \frac{\lambda_0({-}\p{-}\sigma\p{-}\sigma^3\p{-}\widetilde{\S})}{l_0-l_1} \beta + \frac{\lambda_1({-}\p{-}\sigma\p{-}\sigma^3\p{-}\widetilde{\S})}{l_0-l_1}\,,
\]
where
\[
\begin{aligned}
\lambda_0(\S) := & (r_1-r_0)(f_{\p} - f_{\sigma\p})(\S) - (l_1-l_0)(f_{\Pp} - f_{\sigma\Pp})(\S) \,, \\
\lambda_1(\S) := & (r_1-r_0) (l_1f_{\p}(\S) - l_0f_{\sigma\p}(\S)) - (l_1-l_0)(r_1 f_{\Pp}(\S) - r_0f_{\sigma\Pp}(\S)) \,,
\end{aligned}
\]
and we deduce that both $\lambda_0$ and $\lambda_1$ vanish on $P={-}\p{-}\sigma\p{-}\sigma^3\p{-}\widetilde{\S}$. Both $\lambda_0$ and $\lambda_1$ have $4$ poles and $4$ zeroes counted with multiplicity: they have the same poles they share three zeroes, namely $O_E, \widetilde{\S}$ and ${-}\p{-}\sigma\p{-}\sigma^3\p{-}\widetilde{\S}$. Since, in the group on $E(\ol{\F_q})$, the sum of the zeroes of an element of $\ol{\F_q}(E)^\times$ is equal to the sum of the poles, then $\lambda_0$ and $\lambda_1$ also share the fourth zero, hence $\lambda_0$ and $\lambda_1$ differ by a multiplicative constant in $\ol{\F_q}$. This is absurd because $l_0 \neq l_1$ and because the functions $f_{\p}$, $f_{\sigma\p}$, $f_{\Pp}$, $f_{\sigma\Pp}$ are  $\ol{\F_q}$-independent.

A similar argument implies that $R_{2,3}$ does not vanish on $\Zp \cap U$, while the case of $L_{0,2}$ is easier. Substituting $\alpha$ and $\S$ in $L_{0,2}(\alpha, \beta,1,\S)$ as in the definition (\ref{eq:def_Z_43}) of $\Zp$ we get
\[
L_{0,2}(\alpha, \beta, 1, P)|_{\Zp\cap U} = \frac{ (\beta + l_0) \lambda_2 ({-}\p{-}\sigma\p{-}\sigma^3\p{-}\widetilde{\S})}{l_0-l_1}\,,
\]
where
\[
\begin{aligned}
&\lambda_2(\S) := (l_2 {-} l_1)f_{\p}(\S) + (l_0 {-} l_2)f_{\sigma\p}(\S) + ( l_1{-}l_0)f_{\sigma^2\p}(\S) \quad \in \ol{\F_q}(E)  \,.
\end{aligned}
\]
Analogously to $\lambda$, we see that the zeroes of $\lambda_2$ are $\widetilde{\S},O_E$ and ${-}\p{-}\sigma\p{-}\sigma^2\p{-}\S_0$, hence $\lambda_2$ does not vanish on ${-}\p{-}\sigma\p{-}\sigma^3\p{-}\S_0$, implying that $L_{0,2}$ does not vanish on $\Zp \cap U$. 
\end{proof}

We can show that $u_2{-}u_3$, $w_3{+}u_3$, $w_3{+}u_1$ and $w_i{-}w_j$ do not vanish on $\Zp$ using similar arguments to the ones used in the last part of the above proof.
\begin{change}
	Hence, by Lemma \ref{lem:good_curves}, up to forgetting  \eqref{eq:43_item4_main} from the definition of $\calC$,  all the components of $\calC$  are surfaces defined over~$k$ that map dominantly on $\calB$. This does not change after imposing condition (\ref{eq:43_item4_main}), which only restricts to an open dense of $\calB$ (indeed $\calB$ is irreducible and the equations in (\ref{eq:43_item4_main}), seen as homogeneous equations in $\P^2_{\ol k(E)}$, 
	define varieties of the same degree as $\calB$ but different from it, as it can be checked by looking at the intersection with the generic line $\{R_{0,2}=0\}$, as in Claim \ref{hope_distinct_points}.)
\end{change}

\subsection{$k$-rational points on $\calC$  and where to find them}\label{subsec:points_43}
Finally we prove that $\# \calC(k)$ is larger than $\tfrac{1}{2}(\# k)^2$.
The surface $\calC$ is contained in the open subset of $\A^2 \times (E  \setminus \{O_E\}) \times \PGL \times \A^1$ made of points $(\alpha, \beta, (x,y), \smt abcd, z)$ such that $c \neq 0$. Hence  $\calC$ is contained in $\A^8$, with variables $\alpha, \beta$, $x,y$, $a,b,d$, $z$ and it is defined by the following equations:. 
\begin{itemize}
\item $0=p_1:=W(x,y)$, the Weierstrass equation defining $E$;
\item $0=p_2 :=(d^q {-} a)^{q+1}(z^q{-}z)^{q^2{-}q} - (ad {-} b)^q (\tfrac{z^{q^2}{-}z}{z^q{-}z})^{q+1}$, 
the dehomogenization of (\ref{eq:43_item2}) in $c$; \item $0=p_i(\alpha, \beta, x,y,a,b,d)$ for $i = 3,4,5,6$, obtained by (\ref{eq:43_item3}) after dehomogenizing in $c$,  substituting $f_{\sigma^i\p}, f_{\sigma^i\Pp}$ by their expressions in $\alpha, \beta,x,y$ and clearing denominators;
\item a number of conditions $0\neq q_j$ ensuring that $\S\neq - \sigma^i\p, \S \neq - \sigma^i\Pp$, $\beta + f_{\widetilde{\S}}(\sigma^i\p) \neq 0$, $\beta +  f_{\widetilde{\S}}(\sigma^i\Pp) \neq 0$, $d^q-a\neq0$, $ad-b\neq0$, that (\ref{eq:43_item4_main}) is satisfied, that $f_{\alpha, \beta,\S}(\sigma^i\p)$ are distinct and that $f_{\alpha, \beta,\S}(\sigma^i\Pp)$ are distinct.
\end{itemize} 
In particular, $\calC$ can be seen as a closed subvariety of $\A^9$, with variables $\alpha,\beta$, $x,y$, $b,c,d,z$ and $t$  defined by the seven equations $p_1=0, \ldots, p_6=0$ and $0=p_7:=tq_1\cdots q_r -1$. Let $\calC_1, \ldots, \calC_s$ be the geometrically irreducible components of $\calC$. By \cite[Remark $11.3$]{Npoints}, 
we have
\begin{subeqn}\label{eq:Weil43}
\# \calC(k) \ge \# \calC_1(k) \ge (\# k)^2 - (\totaldeg-1)(\totaldeg-2)(\# k)^{\frac 32} - K(\calC_1) (\# k)\,,
\end{subeqn}
where $\totaldeg$ is the degree of $\calC_1$ and $K(\calC_1)$ is the sum of the Betti numbers of $\calC$ relative to the compact $\ell$-adic cohomology. Since $\calC_1$ is a component of $\calC$ then 
\begin{subeqn}\label{eq:degD43}
\totaldeg 
\leq \deg(p_1)\cdots \deg(p_7) \,.
\end{subeqn}
Since $\calC$ is the disjoint union of the $\calC_i$, the Betti numbers of $\calC$ are the sums of the Betti numbers of the $\calC_i$. Hence, using \cite[Corollary of Theorem 1]{Katz}
\begin{subeqn}\label{eq:BettiD43}
K(\calC_1) \leq K(\calC) \leq 6 \cdot 2^{7}\cdot \left(3+7\max_{i=1,\dots ,7}\{\deg(p_i)\}\right)^{10}\,.
\end{subeqn}Combining Equations (\ref{eq:Weil43}), (\ref{eq:degD43}), (\ref{eq:BettiD43}) and the inequalities $\deg p_1 \le 3$, $\deg p_2 \le q^3{+}q$, $\deg p_3, \ldots, \deg p_6 \le 2q{+}3$, $\deg p_7 \le 16q^2 {+} 37 q{+} 75$, we deduce that $\#\calC(k)> \tfrac 12 (\# k)^2$ when $\# k \ge q^{80}$ and $q \ge 3$.

\begin{change}
As in Section \ref{subsec:points_32}, we give a probabilistic procedure to find a $k$-rational point on $\calC$, which, as explained at the end of Section \ref{sec:idea_descent}, is used in the algorithm of Proposition \ref{secondhalfdescentprop}.
We find such a $k$-rational point by successively finding the roots of univariate polynomials (e.g. by Cantor-Zassenhaus \cite{FactorCZ} or by von zur Gathen-Shoup \cite{FactorShoupGathen}) or by solving linear systems. In detail we proceed as follows.
\begin{itemize}
	\item Choose random $\alpha,x \in k$.
	\item Look for $y\in k$ such that $P = (x,y)$ lies on $E$; this is a matter of solving the quadratic equation $p_1=0$, and if such a $y$ does not exists, restart from the first step.
	\item For all $y$ as above, look for $\beta\in k$ satisfying \eqref{eq:43_cross_ratio}; this equation, after clearing denominators becomes a polynomial of degree at most $2q+2$ and if for all $y$ we find no root $\beta\in k$, restart from the first step.
	\item Check if the values $f(\sigma^i\p)$ are distinct and the values $f(\sigma^i\Pp)$ are distinct: if they are not distinct we restart from the first step, otherwise we compute $\smt abcd \in \PGL(k)$ satisfying \eqref{eq:43_item3} (or equivalently the polynomials $p_3,p_4,p_5$); this is a matter of solving linear equations in $a,b,c,d$ and there is always a solution if $f(\sigma^i\p)$ are distinct and $f(\sigma^i\Pp)$ are distinct.
	\item For all $(y,\beta,\smt abcd)$, compute a $z\in k$ satisfying \eqref{eq:43_item2}. That polynomial has degree at most $q^3-q$ and if, for all $(y,\beta,\smt abcd)$, it has no root $z\in k$, we restart from the first step.
	\item Finally we check if  (\ref{eq:43_item4_main}) holds and if 
	$\beta + f_{\widetilde{\S}}(\sigma^i\p) \neq 0$, $\beta +  f_{\widetilde{\S}}(\sigma^i\Pp) \neq 0$, $d^qc-ac^q\neq0$: if they all hold then $(\alpha, \beta, P,\smt abcd, z)$ is a point on $\calC$, if it never holds we restart from the first step.
\end{itemize}
Since the degree of the polynomials involved is at most $q^3$, the expected running time of each step is polynomial in $q\log(\# k)$ (see e.g. \cite[Theorem 3.5 and 3.7]{FactorShoupGathen}). 
To prove that the expected running time of the above procedure is polynomial in $q\log(\# k)$ it is enough proving that the expected number of times each step is repeated is polynomial in $q$. Indeed, we prove that
for given $x,\alpha \in k$, there is a probability at least $\frac{1}{64q^4}$ that the above steps go through, finding a point in $\calC(k)$. 

Since $\alpha$ and $x$ are chosen uniformly (randomly), this is equivalent to saying that there are at least $\frac{(\# k)^2}{64q^4}$ values $(\alpha,x)$ that lift to a point in $\calC(k)$.
We notice that the map $(\alpha,x)\colon \calC \to \A^2$, up to excluding at most $450$ sub-curves of $\calC$, is finite to one, with fibres containing at most $2(2q+2)(q^3-q)$ points: indeed the map $(x,\alpha)$ is the composition of 
\begin{itemize}
\item the projection $\calC \to \calB'$ which has finite fibres and degree $(q^3-q)$ by Lemma \cref{lem:good_curves};
\item  the map $\colon \calB' \to  \A^1\times E$, $(\alpha, \beta, P) \mapsto (\alpha, P)$ which has degree $2q+2$ and, up to excluding the bad $P$'s in Claim \ref{hope_distinct_points}, it has finite fibres (indeed $\calB_P$ is necessarily a curve since $\calB$ irreducible surface mapping dominantly on $E$ and if $P$ is ``good'', then $\calB_P$ contains no lines as proven using \eqref{eq:lines_conj_inters} hence the map $\alpha$ is has finite fibres on $\calB_P$) ;
\item the  degree-$2$  map $(x,\alpha)\colon E\times \A^1 \to \A^1\times \A^1$.
\end{itemize}
Denote $\calC'$ the open of $\calC$ obtained excluding the points above a ``bad'' $P$. 
Since $B_P$ is a curve in $\A^2$ of degree $2q+2$, each of the ``bad'' $P$'s gives at most $(\# k)(2q+2)(q^3-q)$ points on $\calC$, which using that $\# k>q^{80}$, implies that $\calC(k)-\calC'(x)$ containes less than $\frac{(\# k)^2}{4}$ points. Since $\#\calC(k)> \frac{(\# k)^2}{2}$  we deduce $$\#\calC'(k)> \frac{(\# k)^2}{4} .$$
We proved that the map $(\alpha,x)\colon \calC'\to \A^2$ has finite fibres contining at most $2(2q+2)(q^3-q)$ points, hence at least
$$ \frac{(\# k)^2}{8(2q+2)(q^3-q)} > \frac{(\# k)^2}{64 q^4} $$
points $(\alpha,x)\in \A^2(k)$ can be lifted to $\calC'(k)$ (hence to $\calC(k)$).

\end{change}

\section{Some explicit computations}\label{sec:example}

To give a better view on the algorithm, we show in explicit examples how we can compute elliptic presentations, find and represent traps and elements in the factor base, and how the descent procedure works. 
Once the framework and descent procedure is set, the algorithm in Section \ref{sec_algo} reduces the logarithm of every element to the logarithms of elements in the factor base and performs linear algebra, following a classical strategy (see also \cite{ZKG} and \cite{KW}); these general steps are described without explicit computations. 

Our algorithm is designed to prove that the discrete logarithm problem is at most quasi-polynomially hard on fields of small characteristic, but not designed for practical efficiency. Other algorithms have proved their strength in the real world (see the Introduction). 
We performed our computations in PARI/GP (\cite{PARI2}), using the script available at \cite{my_repo}.


To illustrate how to find elliptic presentations of (extensions of) finite fields, suppose first we want to compute discrete logarithms in a field $F = \F_{2^{93}}$. The field $F$ itself does not admit a good elliptic presentation since the ``base field'' $\F_q$ could only be $F$, $\F_{2^{31}}$, $\F_{2^3}$ or $\F_2$: in the first two options $q$ is not small compared to $\# F$ and the algorithm, which computes logs up to $\F_q^\times$, would not be interesting; the other two options are not feasible, since there are no elliptic curves over $\F_2$ or $\F_8$ with a point of order $31$.
Instead, compatibly with Proposition \ref{propfieldswithpresentation}, we can embed $F$ in  $K = \F_{2^{93 \cdot 2}}$ which admits an elliptic presentation with $q=2^6$: by \cite[Theorem 4.1]{waterhouse1969abelian} there exists at least one elliptic curve $E/\F_{64}$ with $\# E(\F_{64}) = 62$, hence with a point $P_0$ of order $31$. Representing $\F_{64}$ as 
$\F_2[\ti]/(\ti^6 + \ti^5 + \ti^3 + \ti^2 + 1)$ and enumerating all possible Weierstrass equations over $\F_{64}$, we find the following $E$ and $P_0$: 
$$E : Y^2 + X Y + X^3 + \ti X^2 + (\ti^4 + \ti^3 + 1) = 0  \,, \quad P_0 = (\ti^5 + \ti^3 + 1, \ti^5 + \ti^2 + \ti) $$

The second example we look at has smaller degree, allowing for slightly more compact description. The rest of this section is dedicated to
$$K = \F_{2^{30}} = \F_{2^{3\cdot 10}} ,$$
which admits an elliptic presentation as an extension of 
$$\F_q = \Fqq = \F_2[\ti]/\ti^3 + \ti^2 + 1.$$
Again with enumeration we find an elliptic curve  $E/\Fqq$ and a point $P_0\in E(\Fqq)$ of order $10$, namely 
 $$
 E : Y^2 + XY + X^3 + X^2 + (\ti + 1)\,, 
  \quad P_0 = (\ti^5 + \ti^3 + 1, \ti^5 + \ti^2 + \ti) .
$$ 
 In this setting the Frobenius on $E$ is $\phi\colon (x,y)\mapsto (x^{8},y^{8})$. 
To look for points $(x,y)\in E(\ol{\F_2})$ such that  $\phi(x,y) = (x,y) + P_0$, we notice that by definition they satisfy equation 
\begin{equation*}
x^8 = \frac{x^3 + (\ti^2 + \ti + 1)x^2 + (y + \ti^2)x + (y^2 + (\ti^2 + \ti)y + (\ti + 1))}{x^2 + (\ti + 1)} ,
\end{equation*} 
together with another equation involving $y^8$.
Picking one such point and writing down the functions in $\F_q[X,Y]$ vanishing on it, we get, as in Definition \ref{def:ell_pres}, the elliptic presentation
\[
\begin{aligned}
&\qquad\qquad \qquad \; \mathfrak M = (\mu(X),\nu(X,Y)), \qquad \text{with} 
\\
\mu(X) = & \, X^{10} + (\ti^2 + \ti) X^9 + (\ti + 1) X^8 + (\ti^2 + \ti + 1) X^7 + (\ti + 1) X^6 +
\\ & \, + X^5 + (\ti^2 + \ti + 1) X^4 + \ti^2 X^3 + \ti^2 X^2 + \ti^2 X + \ti^2 \,,
\\
\nu(X,Y) = & \, Y + \ti X^{10} + (\ti^2 + \ti) X^8 + X^2 + (\ti^2 + 1) X \,.
\end{aligned}
\]
In particular, using that $\nu$ is monic in $Y$, we have
\begin{eqn}\label{eq:iso_K_M}
K = \F_{2^{30}} \cong  \Fqq[X,Y]/\mathfrak{M}  = \Fqq[X]/\mu(X).
\end{eqn} 
If we started with a different representation of $K$ (e.g. a representation $\F_2[\zeta]/p(\zeta)$), we can explicitly compute an isomorphism as above by first fixing an embedding $\Fqq \hookrightarrow K$ (see for example \cite{bill}, in practice we use  PARI/GP built-in function \texttt{ffembed}) and then finding a root of $\mu$ in $K$.

Using the isomorphism \eqref{eq:iso_K_M}, the purpose is to compute the logarithms of elements in $(\Fqq[X,Y]/\mathfrak{M})^\times/\Fqq^\times$ (the information coming from $\Fqq^\times$ can be recovered by enumeration), which can be seen as rational functions on $E$, up to scalars. As explained in Section \ref{sec:divs}, we can identify these rational functions with their divisors of zeroes and poles. Accordingly, we use a factor base made of irreducible divisors on $E$, that is divisors $D$ on $E$ that are sums of points in a $\Gal(\ol{\F_2}/\Fqq)$-orbit. 

Before listing elements in such a factor base, we recall Mumford representation for divisors on (hyper)elliptic curves, which allows us use $\Fqq[X,Y]$, without representing points in $E(\ol{\F_2})$. Either $D = \{ O_E\}$, or the points in an irreducible divisor $D$ are ``affine points'' $(x_i, y_i)$ and, since they are all Galois-conjugate, then all $x_i$'s have the same minimal polynomial $F(X)$ over $\Fqq$;  moreover the $x_i,y_i$ satisfy an algebraic relation  $G(x_i,y_i)=0$, with $G \in \F_q[X,Y]$: the Weierstrass polynomial $W(X,Y)$ is an example and, since the $x_i$'s satisfy $F(x_i)=0$, if $W(X,Y) \bmod F(X)$ is reducible, one of the factors gives a relation with degree~$1$ in $Y$. In other words $D$ is either the point at infinity or it is the vanishing locus of a maximal ideal of $\Fqq[X,Y]$ of the form 
\begin{equation}\label{eq:mum}
	(F(X),G(X,Y))  ,
\end{equation}
with $F$ irreducible and $G$ monic of degree $1$ or $2$  in $Y$ and irreducible modulo $F$. In this case, $\deg D = \deg_X(F)\cdot \deg_Y(G)$.

The factor base in the main algorithm, in Section \ref{sec_algo}, is composed of all irreducible divisors having degree $1,2,4,8,\cdots, 2^8$ and all irreducible divisors supported on the set of \emph{trap} points (see Definition \ref{def:traps}). 
For example the list of irreducible divisors of degree $1$ starts as 
\[
\{O_E\},\; (X, Y + \ti^2 + \ti) ,\;  (X + 1, Y + \ti^2),\;  (X + \ti, Y + \ti^2) \ldots, 
\]
and, in degree $2$ it starts as
\[
\begin{gathered}
 (X + \ti^2 {+} \ti {+} 1, Y^2 + (\ti^2 {+} \ti {+} 1) Y {+} \ti), \;  (X^2 + X + \ti^2, Y + (\ti^2 {+} 1) X + \ti^2 {+} \ti {+} 1) , 
 \\ 
 (X^2 + \ti X + (\ti^2 {+} \ti), Y + (\ti^2 {+} 1) X + \ti^2 {+} \ti)  \ldots 
\end{gathered}
\]
To give a sample of the set of trap points (which is $\Gal(\ol{\F_q}/\F_{q})$-stable, hence it is a union of supports of irreducible divisors) we look at the points $Q = (x,y)$ such that $(2\phi-1)(Q) = P_0$: this condition is equivalent to 
\[
(x,y) = \left(\frac{x^{4q} {+} x^{3q} {+} x^{2q} {+} y^qx^q {+} y^{2q}}{x^{2q}}, \frac{x^{6q} {+} x^{5q} {+} y^qx^{4q} {+} x^{3q} {+} yq^{2q}x^{2q} {+} y^{3q}}{x^{3q}} \right),
\]
and after clearing denominators and searching irreducible ideals in $\Fqq[X,Y]$ containing the resulting polynomials, we find the trap divisors
\[
\begin{gathered}
 (X^4 {+} (\ti^2 {+} 1) (X^2 {+}  X {+} 1),\; Y {+} ((\ti^2 {+} \ti) X^3 {+} (\ti {+} 1) X^2 {+} \ti X {+} \ti^2 {+} 1)), 
 \\ (X^3 {+} (\ti^2 {+} 1) X^2 {+} X {+} (\ti^2 {+} \ti),\; Y {+} X^2 {+} (\ti^2 {+} \ti) X {+} \ti^2 {+} 1 ),  \ldots 
 \end{gathered}
\]
Of course, the list of all divisors in the factor base, including the traps, is longer, containing at least $\tfrac1{256} \# E(\F_{q^{256}})  > 2^{300}$ elements.

Now an example of the lift-and-descend process, the machinery used in the relation collection phase of the main algorithm. 
Notice that in our small and explanatory example, $\deg(K/\Fqq) = 10$ is smaller than $2^8$, which  is a constant of the algorithm. For the sake of this example we assume that the factor base only contains traps and divisors of degree $1,2,4$.  
Let $g,h$ elements in $K^\times$, with $g$ a generator. For example, again using the isomorphism \eqref{eq:iso_K_M} we can use the following polynomials
$$ 
g= X^3+1 \bmod \mu(X), \qquad h = X^5+\ti \bmod \mu(X).
$$
We can check that $g$ is a generator either listing its powers modulo $\mu$ or in practice we constructed an isomorphism $\F_8[X]/\mu\cong \F_{2^{30}}$, with PARI/GP's construction of $\F_{2^{30}}$, and we used built-in PARI/GP's function to check.

Following Step $3$ of the main algorithm, for random integers $a,b$ we want to write $\log_{\frakM,g}(g^ah^b)$ as the logarithm of a sum of divisors in the factor base. Take, for example $a = 2024$ and $b=2025$. As a first step we pass from $g^ah^b$ to a sum of divisors with degree a power of $2$: we first take $u_0(X) \equiv g^ah^b \,(\bmod\,\mu(x))$ of smallest possible degree and then we look for an irreducible polynomial $\tilde u(X)$ of shape $u_0(X) + (r_1X^{16-\deg(u_0)}+r_2)\mu(X)$ for random $r_1,r_2 \in \Fqq^\times$.  After 17 random tries  we find  
$$ 
\begin{aligned}
\tilde u =&\;\; (\ti {+} 1) X^{16} {+} (\ti^2 {+} \ti {+} 1) X^{15} {+} (\ti^2 {+} 1) X^{14} {+} \ti^2 X^{13} {+} (\ti^2 {+} 1) X^{12} +
\\ &\;\; {+} (\ti {+} 1) X^{11}  {+} (\ti^2 {+} 1) X^{10} {+} X^9 {+} \ti^2 X^8 {+} (\ti^2 {+} \ti {+} 1) X^7 {+} (\ti {+} 1) X^6 +
\\ &\;\; {+} (\ti^2 {+} \ti {+} 1) X^5 {+} (\ti^2 {+} 1) X^4 {+} X^3 {+} (\ti^2 {+} 1) X^2 {+} \ti X {+} 1 .
\end{aligned}
$$
Since $\tilde u \equiv u_0 \equiv g^ah^b$ modulo $\mu$, computing the irreducible components of $\divi (\tilde u)$ (we call irreducible components  of a divisor $D$ the irreducible divisors appearing with non-zero coefficient when we write $D$ as a linear combination of irreducible divisors) we get 
$$
\begin{gathered}
\log_{\frakM}(\divi g^ah^b) = \log_{\frakM}( \divi \,\tilde u) = \log_{\frakM}( D_1 +D_2 - 32 O_E)  \quad \text{with }
\\
 D_1 = (-1){\cdot } D_2 = (\tilde u(X),  Y + X^{15} {+} X^{14} {+} (\ti {+} 1) X^{13} \!{+} \ti X^{12} \!{+} (\ti^2 {+} \ti {+} 1) X^{11} \! {+} \ti X^7  
 \\ \,\,\,\,\,\, {+} (\ti^2 {+} \ti {+} 1) X^{10} {+}  \ti^2 X^9   {+} \ti X^5 {+} (\ti^2 {+} \ti {+} 1) X^4 {+} (\ti^2 {+} \ti) X^3 {+} \ti X^2 {+} \ti X {+} \ti {+} 1 )  ,
 \\ \text{irreducible of degree $16$.}
\end{gathered}
$$
The irreducibility of $\tilde u$ a priori implies $\divi (\tilde u) +32 O_E$   is either irreducible or the sum of two irreducible divisors of degree $16$.

Next step is descent: finding a divisor $D'$ satisfying  
$\log_{\frakM}(D_1) = \log_{\frakM} D'$ 
and $\essdeg_{\F_q} D' \mid 8$, i.e the irreducible components of $D'$ either have degree dividing $8$ or are supported on the set of traps.

To find $D'$, we 
base change to 
$$
k =  \F_{q^{{16}/4}} \cong \bigslant{\F_2[\xi]}{(\xi^{12} {+} \xi^{11} {+} \xi^{10} {+} \xi^9 {+} \xi^8 {+} \xi^7 {+} \xi^6 {+} \xi^5 {+} \xi^4 {+} \xi^3 {+} \xi^2 {+} \xi {+} 1 )},
$$
so that $D_1$ is the sum of the $\Gal(k/\F_q)$-conjugates of the following divisor, irreducible over $k$
$$
D_{1,k} = (X^4 + a_3X^3 +  a_2X^2 + a_1X + a_0, Y + b_3X^3 + b_2X^2 + b_1 X + b_0 ),
$$
where the coefficients are, 
$a_3 = \xi^6 {+} \xi^3 {+} \xi^2 {+} \xi$, 
$a_2 = \xi^{10} {+} \xi^7 {+} \xi^5 {+} \xi^4 {+} \xi^3$, 
 $a_1 = \xi^{10} {+} \xi^9 {+} \xi^8 {+} \xi^5$, 
$a_0 = \xi^{11} {+} \xi^{10} {+} \xi^9 {+} \xi^8 {+} \xi^7 {+} \xi^4 {+} \xi^3 {+} \xi^2 {+} \xi$, 
$b_3 = \xi^6 {+} \xi^4 {+} \xi^2$, 
$b_2 = \xi^8 {+} \xi^4 {+} \xi^3 {+} \xi^2 {+} \xi {+} 1$, 
\\
$b_1 = \xi^{11} {+} \xi^9 {+} \xi^8 {+} \xi^7 {+} \xi^6 {+} \xi^3 {+} \xi$, 
and 
$b_0 = \xi^{11} {+} \xi^9 {+} \xi^8 {+} \xi^7 {+} \xi^6 {+} \xi^5 {+} \xi^3 {+} \xi$. 
This divisor $D_{1,k}$ corresponds to $\tilde D$ in \eqref{eq:g=1_onGaloisconjugates}.


Following Section \ref{sec:idea_descent}, the strategy to compute a possible $D'$ is to find a divisor $D'_{k}$ defined over $k$ that has the same log as $D_{1,k}$ (in the sense of \eqref{eq:samelog} and of Propositions \ref{firsthalfdescentprop} and \ref{secondhalfdescentprop}) and whose irreducible components over $k$ have degree either  $1$ or $2$, so that we can take $D'$ equal to the sum of all $\Gal(k/\F_q)$-conjugates of $D'_k$. We look for $D'_k$ in two steps, first finding an analogous $D''_k$ whose irreducible components of degree $1,2$ or $3$ and then finding $D'_k$ from $D''_k$. 

In our example, following  Section \ref{sec:43},  we found
\begin{equation*} 
\begin{aligned}
D''_k & =  D_{1,k} - \divi \left( g \right), \quad
\text{with}\quad
 g = \tfrac{( cf + d)(f^{\phi}\circ \tau_{P_0}) + af + b}{  cf^{q+1}+  df^q + af + b}, \quad f = \frac{ \tfrac{Y-y(P)}{X-x(P)} + \alpha }{ \tfrac{Y-y({\widetilde P})}{X-x({\widetilde P})} + \beta} \,,
\end{aligned}
\end{equation*}
and with the values
\begin{equation*}
\begin{gathered}
a= \xi^9 {+} \xi^8 {+} \xi^7 {+} \xi^6 {+} \xi^4 {+} \xi^2 {+} 1,  \quad
b= \xi^{11} {+} \xi^{10} {+} \xi^9 {+} \xi^6 {+} \xi^5 {+} \xi^3 {+} \xi {+} 1,
\quad  c = 1,
\\   
d =  \xi^{10} {+} \xi^6 {+} \xi^5 {+} \xi^2 , \quad \alpha = \xi^{11} {+} \xi^{10} {+} \xi^9 {+} \xi^8 {+} \xi^7 , \quad \beta =  \xi^{10} {+} \xi^7 {+} \xi^4 {+} \xi^3 {+} \xi^2 {+} 1 ,
\\
P = (\xi^{10} {+} \xi^7 {+} \xi^6 {+} \xi^5 {+} \xi^3 {+} 1, \; \xi^{11} {+} \xi^9 {+} \xi^7 {+} \xi^3 {+} \xi), \quad  
\\ 
\widetilde{P} = (\xi^{11} {+} \xi^{10} {+} \xi^9 {+} \xi^8 {+} \xi^4 {+} \xi^3 {+} \xi^2, 
\; \xi^7 {+} \xi^6 {+} \xi^5 {+} \xi^4 {+} \xi^3 {+} 1).
\end{gathered}
\end{equation*}
Following Sections \ref{sec:idea_descent} and \ref{sec:43}, we explain why such $f$ and $D'_k$ work and how we computed them. 
No matter who $f$ and $a,b,c,d$ are, except for some restriction that is generically true, the definition of $D''_k$ and $g$  ensures that $D'_k$ and $D_{1,k}$ have the same logarithm (see Section \ref{sec:idea_descent}, above Equation \ref{eq:samelog}). 
Moreover all functions $f\colon E \to \P^1$ with the above shape have degree at most $3$,
since their ``numerator'' and ``denominator'' both have degree $2$ and a simple pole at infinity. 
Then, we looked for $P$, $\widetilde P$, $\alpha$, $\beta$, $a$, $b$, $c$ and $d$ such that the ``numerator'' of $g$ vanishes on $D_1$ and the polynomial $cT^{q+1}+dT^q+aT+b$ has all its roots $r_i$ in $k$: in this way we can be sure that  the irreducible components of $D_k''$ have degree $\le 3$ since the ``denominator'' of $g$ is the product of functions $f-r_i$ and the ``numerator'' of $g$ has degree at most $6$, so all its zeroes, except from $D_{1,k}$, are contained in irreducible divisors of degree $\le 2$, while its poles are also poles of $f$ or of $f^\phi\circ \tau_{P_0}$, which have degree $3$. 
These conditions about $P$, $\widetilde P$, $\alpha$, $\beta$, $a$, $b$, $c$ and $d$ are polynomial conditions (see Section \ref{subsec:def43}), hence to compute these points we essentially looked for $k$-rational points on a certain variety. In practice we first fixed a point $\widetilde P$ (a generic one is good, following Section \ref{subsec:def43}), then picked random values for $x_P$ and $\alpha$ and tried to solve for the other variables satisfying the relevant polynomial equations (see the second part of Section \ref{subsec:points_43}) until we found a point. We found the points $P,\alpha,\beta,a,b\ldots$ above after trying $616$ random values $(x,\alpha)$.

Notice that in Section \ref{sec:43}, to prove Proposition \ref{secondhalfdescentprop}, we use two additional hypothesis that we still haven't remarked in the computation of $D''$: the first one is $[k:\F_q]\ge 80$, the second that $D_{1,k}$ is not supported on the traps. The second hypothesis is true in our example, while, for better readability and because of computational constraints, we are using a small example, of degree smaller than $80$. We were anyway able to apply the procedure in Section \ref{sec:43} and find $D''_k$.

After $D''_k$, we compute $D'_k$ by considering the irreducible components of $D''_k$ that have degree $3$ and applying to them a similar procedure (actually simpler, see Section \ref{sec:3-2}) to the one we applied to $D_{1,k}$. Applying a ``norm'' we get $D'$, the result of the descent procedure, which has 80 irreducible components, and starts as
\[
\begin{aligned}
D' & =  \!\!\!\!\!\!  \sum_{\sigma \in \Gal(k/\F_q)} \!\!\!\!\!\! \!\! \sigma(D'_k) 
=  -(X^8 {+} (\ti^2 {+} 1) X^7 {+} (\ti^2 {+} 1) X^6 {+} (\ti {+} 1) X^5 {+} (\ti {+} 1) X^4 {+} 
\ldots
\end{aligned}
\]
We can apply the same descent step to $D_2$, so that, after adding $D'- 32\,O_E$, we find a divisor of essential degree $8$ (or a divisor) with the same log as $\divi(\widetilde u)$ hence as $g^a h^b$. 
Repeating the descent on the result we find a divisor $\mathcal D$ whose logarithm is equal to $\log_{\frakM,g}(g^a h^b)$ and whose (non-trap) irreducible components have degree dividing $4$. In other words $\mathcal D$ is a linear combination of elements of the factor base (recall in this example we are using a ``mock factor base'', smaller than in the main algorithm, in Section \ref{sec_algo}). 
Replaying this game for various pairs  $(a,b)$ we could collect enough relations and try to perform linear algebra to compute $\log_{\frakM,g}(h)$ as usual. 


\bigskip

\paragraph{Funding} The author was supported by the MIUR Excellence Department Projects MatMod@TOV and Math@TOV awarded to the Department of Mathematics, University of Rome Tor Vergata, the ``National Group for Algebraic and Geometric Structures, and their Applications" (GNSAGA - INdAM), the PRIN PNRR 2022 ``Mathematical Primitives for Post Quantum Digital Signatures" and by ``Programma Operativo Nazionale (PON) Ricerca e Innovazione” 2014-2020.

\paragraph{Acknowledgements}
\addcontentsline{toc}{section}{Acknowledgements}
I thank Ren\'e Schoof for introducing me to this research problem and for the useful ideas that lead to substantial simplifications of the proof. 

\bibliographystyle{elsarticle-num} 
\bibliography{biblio}

\begin{thebibliography}{10}
\expandafter\ifx\csname url\endcsname\relax
  \def\url#1{\texttt{#1}}\fi
\expandafter\ifx\csname urlprefix\endcsname\relax\def\urlprefix{URL }\fi
\expandafter\ifx\csname href\endcsname\relax
  \def\href#1#2{#2} \def\path#1{#1}\fi

\bibitem{DifHel}
W.~Diffie, M.~Hellman, New directions in cryptography, IEEE Trans. Inform.
  Theory 22 (1976).

\bibitem{BGJT}
R.~Barbulescu, P.~Gaudry, A.~Joux, E.~Thom\'e, A quasi-polynomial algorithm for
  discrete logarithm in finite fields of small characteristic, in: Annual
  International Conference on the Theory and Applications of Cryptographic
  Techniques, 2014, pp. 1--16.

\bibitem{Joux}
A.~Joux, A new index calculus algorithm with complexity {L}(1/4+o(1)) in small
  characteristic, in: Conference on Selected Areas in Cryptography, 2013, pp.
  355--379.

\bibitem{ZKG}
R.~Granger, T.~Kleinjung, J.~Zumbr\"{a}gel, On the discrete logarithm problem
  in finite fields of fixed characteristic, Transactions of the American
  Mathematical Society 370~(5) (2018) 3129--3145.

\bibitem{comp1}
R.~Barbulescu, C.~Bouvier, J.~Detrey, P.~Gaudry, H.~Jeljeli, E.~Thom{\'e},
  M.~Videau, P.~Zimmermann, Discrete logarithm in {G}{F}(2809) with {FFS}, in:
  H.~Krawczyk (Ed.), Public-Key Cryptography -- PKC 2014, Springer Berlin
  Heidelberg, Berlin, Heidelberg, 2014, pp. 221--238.

\bibitem{comp2}
F.~G{\"o}lo{\u{g}}lu, R.~Granger, G.~McGuire, J.~Zumbr{\"a}gel, Solving a
  $6120$-bit {DLP} on a desktop computer, in: T.~Lange, K.~Lauter,
  P.~Lison{\v{e}}k (Eds.), Selected Areas in Cryptography -- SAC 2013, Springer
  Berlin Heidelberg, Berlin, Heidelberg, 2014, pp. 136--152.

\bibitem{compFrancisco1}
G.~Adj, A.~Menezes, T.~Oliveira, F.~Rodríguez-Henríquez,
  \href{https://www.sciencedirect.com/science/article/pii/S1071579714001294}{Weakness
  of $\mathbb{F}_3^{6 \cdot 1429}$ and $\mathbb{F}_2^{4 \cdot 3041}$ for
  discrete logarithm cryptography}, Finite Fields and Their Applications 32
  (2015) 148--170, special Issue : Second Decade of FFA.
\newblock \href {https://doi.org/https://doi.org/10.1016/j.ffa.2014.10.009}
  {\path{doi:https://doi.org/10.1016/j.ffa.2014.10.009}}.
\newline\urlprefix\url{https://www.sciencedirect.com/science/article/pii/S1071579714001294}

\bibitem{JP}
A.~Joux, C.~Pierrot, Improving the {P}olynomial time {P}recomputation of
  {F}robenius {R}epresentation {D}iscrete {L}ogarithm {A}lgorithms - simplified
  setting for small characteristic finite fields, in: Advances in Cryptology -
  ASIACRYPT 2014 20-th International Conference on the Theory and Application
  of Cryptology and Information Security. Proceedings, Part I, Kaoshiung,
  Taiwan, R.O.C., 2014, pp. 378--397.

\bibitem{gologlu}
F.~G{\"o}lo{\u{g}}lu, R.~Granger, G.~McGuire, J.~Zumbr{\"a}gel, On the function
  field sieve and the impact of higher splitting probabilities, in: R.~Canetti,
  J.~A. Garay (Eds.), Advances in Cryptology -- CRYPTO 2013, Springer Berlin
  Heidelberg, Berlin, Heidelberg, 2013, pp. 109--128.

\bibitem{HidePGL2}
A.~W. Bluher, On $x^{ q+1} + ax + b$, Finite Fields and Their Applications
  10~(3) (2004) 285--305.

\bibitem{Micheli}
G.~Micheli, On the selection of polynomials for the {DLP} quasi-polynomial time
  algorithm for finite fields of small characteristic, SIAM Journal on Applied
  Algebra and Geometry 3~(2) (2019) 256--265.

\bibitem{tesi}
G.~Lido, Discrete logarithm over finite fields of small characteristic,
  Master's thesis, Universit\'a di Pisa, available at
  https://etd.adm.unipi.it/t/etd-08312016-225452. (2016).

\bibitem{EllPres}
J.~M. Couveignes, R.~Lercier, Elliptic periods for finite fields, Finite Fields
  and Their Applications 15~(1) (2009) 1--22.

\bibitem{KW2}
T.~Kleinjung, B.~Wesolowski, Discrete logarithms in quasi-polynomial time in
  finite fields of fixed characteristic, Journal of the American Mathematical
  Society 35~(2) (2022) 581--624.

\bibitem{KW}
T.~Kleinjung, B.~Wesolowski, A new perspective on the powers of two descent for
  discrete logarithms in finite fields, The Open Book Series 2~(1) (2019)
  343--352.

\bibitem{JEll}
A.~Joux, C.~Pierrot, Algorithmic aspects of elliptic bases in finite field
  discrete logarithm algorithms, Advances in Mathematics of Communications
  18~(5)  1266--1302.

\bibitem{granger2021computation}
R.~Granger, T.~Kleinjung, A.~Lenstra, B.~Wesolowski, J.~Zumbr{\"a}gel,
  Computation of a 30750-bit binary field discrete logarithm, Mathematics of
  Computation 90~(332) (2021) 2997--3022.

\bibitem{compFrancisco2}
G.~Adj, I.~Canales-Martínez, N.~Cruz-Cortés, A.~Menezes, T.~Oliveira,
  L.~Rivera-Zamarripa, F.~Rodríguez-Henríquez, Computing discrete logarithms
  in cryptographically-interesting characteristic-three finite fields, Advances
  in Mathematics of Communications 12~(4)  741--759.

\bibitem{Ruc}
H.~G. R\"uck, A note on elliptic curves over finite fields, Mathematics of
  Computation 49~(179) (1987) 301--304.

\bibitem{waterhouse1969abelian}
W.~C. Waterhouse, Abelian varieties over finite fields, in: Annales
  scientifiques de l'{\'E}cole normale sup{\'e}rieure, Vol.~2, 1969, pp.
  521--560.

\bibitem{traps}
Q.~Cheng, D.~Wan, J.~Zhuang, Traps to the {BGJT}-algorithm for discrete
  logarithms, LMS Journal of Computation and Mathematics 17 (2014) 218--229.

\bibitem{Sil}
J.~H. Silverman, The arithmetic of elliptic curves, Vol. 106, Springer, Science
  and Business Media, 2009.

\bibitem{Ber}
E.~Berlekamp, Factoring polynomials over large finite fields, Math. Comp 24
  (1970) 713--735.

\bibitem{Dirichlet}
D.~Wan, Generators and irreducible polynomials over finite fields, Mathematics
  of Computation 66~(129) (1997) 1195--1212.

\bibitem{LenIso}
W.~Lenstra, Finding isomorphisms between finite fields, Math. Comp 56 (1991)
  329--347.

\bibitem{RosSch}
J.~B. Rosser, L.~Schoenfeld, Approximate formulas for some functions of prime
  numbers, Illinois Journal of Mathematics 6~(1) (1962) 64--94.

\bibitem{Npoints}
S.~R. Ghorpade, G.~Lachaud, \'{E}tale cohomology, {L}efschetz theorems and
  number of points of singular varieties over finite fields, Moscow
  Mathematical Journal 2~(3) (2002) 589--631.

\bibitem{Katz}
N.~M. Katz, Sums of {B}etti numbers in arbitrary characteristic, Finite Fields
  and their Applications 7~(1) (2001) 29--44.

\bibitem{FactorCZ}
D.~G. Cantor, H.~Zassenhaus, \href{https://doi.org/10.2307/2007663}{A new
  algorithm for factoring polynomials over finite fields}, Math. Comp. 36~(154)
  (1981) 587--592.
\newblock \href {https://doi.org/10.2307/2007663} {\path{doi:10.2307/2007663}}.
\newline\urlprefix\url{https://doi.org/10.2307/2007663}

\bibitem{FactorShoupGathen}
J.~von~zur Gathen, V.~Shoup,
  \href{https://doi.org/10.1007/BF01272074}{Computing {F}robenius maps and
  factoring polynomials}, Comput. Complexity 2~(3) (1992) 187--224.
\newblock \href {https://doi.org/10.1007/BF01272074}
  {\path{doi:10.1007/BF01272074}}.
\newline\urlprefix\url{https://doi.org/10.1007/BF01272074}

\bibitem{PARI2}
{The PARI~Group}, Univ. Bordeaux, {PARI/GP version \texttt{2.15.4}}, available
  from \url{http://pari.math.u-bordeaux.fr/} (2023).

\bibitem{my_repo}
G.~Lido, \href{https://github.com/guidoshore/log_small_example/}{{G}{P} code
  for key steps of discrete logarithm algorithm using elliptic presentations}.
\newline\urlprefix\url{https://github.com/guidoshore/log_small_example/}

\bibitem{bill}
B.~Allombert,
  \href{https://www.sciencedirect.com/science/article/pii/S1071579701903442}{Explicit
  computation of isomorphisms between finite fields}, Finite Fields and Their
  Applications 8~(3) (2002) 332--342.
\newblock \href {https://doi.org/https://doi.org/10.1006/ffta.2001.0344}
  {\path{doi:https://doi.org/10.1006/ffta.2001.0344}}.
\newline\urlprefix\url{https://www.sciencedirect.com/science/article/pii/S1071579701903442}

\end{thebibliography}

\end{document}